\documentclass[11pt,a4paper, reqno]{amsart}

\usepackage{a4wide}

\usepackage{amsfonts,amsthm,amssymb,amsmath,amscd,amsopn,mathabx,mathrsfs}
\usepackage[pdftex]{graphicx}
\usepackage{graphics}
\usepackage[matrix,arrow]{xy}
\usepackage[active]{srcltx}
\usepackage[draft]{todonotes}
\usepackage{tensor}

\makeatletter
\@namedef{subjclassname@2020}{\textup{2020} Mathematics Subject Classification}
\makeatother

\usepackage[colorlinks=true]{hyperref}
\hypersetup{
colorlinks   = true,
citecolor    = blue
}

\usepackage{enumitem, hyperref}\hypersetup{colorlinks}

\makeatletter
\def\reflb#1#2{\begingroup
    #2%
    \def\@currentlabel{#2}%
    \phantomsection\label{#1}\endgroup
}
\makeatother

\usepackage{color}

\definecolor{darkred}{rgb}{1,0,0} 
\definecolor{darkgreen}{rgb}{0,0.8,0}
\definecolor{darkblue}{rgb}{0,0,1}

\hypersetup{colorlinks,
linkcolor=darkblue,
filecolor=darkgreen,
urlcolor=darkred,
citecolor=darkblue}

\newtheorem{thm}{Theorem}
\numberwithin{thm}{section}
\numberwithin{equation}{section}
\newtheorem{theorem}{Theorem}
\newtheorem*{theorem*}{Theorem}
\newtheorem{corollary}{Corollary}
\newtheorem*{corollary*}{Corollary}
\newtheorem{lemma}{Lemma}

\newtheorem{proposition}{Proposition}

\newtheorem*{conjecture*}{Conjecture}

\newtheorem*{question*}{Question}
\newtheorem{definition}{Definition}
\newtheorem*{definition*}{Definition}

\newtheorem*{definitions*}{Definitions}

\newtheorem*{rem*}{Remark}
\theoremstyle{remark}
\newtheorem{remark}[thm]{Remark}

\newtheorem*{remark*}{Remark}

\newtheorem*{remarks*}{Remarks}
\newtheorem*{example*}{Example}

\newtheorem*{examples*}{Examples}

\newcommand{\R}{\mathbb{R}}
\newcommand{\Z}{\mathbb{Z}}
\newcommand{\Q}{\mathbb{Q}}
\newcommand{\C}{\mathbb{C}}
\newcommand{\N}{\mathbb{N}}
\newcommand{\T}{\mathbb{T}}

\def\CP{{\mathbb C}P}
\def\RP{{\mathbb R}P}
\def\HP{{\mathbb H}P}
\def\CaP{{\text Ca}P}
\def\PP{{\mathbb P}}

\newcommand{\lcm}{\text{lcm}}

\newcommand{\Id}{\mathit{Id}}

\newcommand{\Hess}{\text{Hess}\,}

\newcommand{\ep}{\epsilon}
\newcommand{\ga}{\gamma}

\newcommand{\om}{\omega}
\newcommand{\Om}{\Omega}

\newcommand{\Sp}{\mathrm{Sp}}

\newcommand    \TSp     {\widetilde{\mathrm{Sp}}}

\newcommand\numberthis{\addtocounter{equation}{1}\tag{\theequation}}

\def\cz{{\mu}}
\def\rs{{\mu_{RS}}}
\def\P{{\mathcal P}}
\def\PP{{\mathscr P}}
\def\HC{{\mathrm{HC}}}
\def\SH{{\mathrm{SH}}}
\def\CC{{\mathrm{CC}}}
\def\HF{{\mathrm{HF}}}
\def\H{{\mathrm{H}}}

\def\L{L^{2m+1}_q(\ell_0,\dots,\ell_m)}

\def\Im{{\text{Im}}}

\def\mi{{\hat\mu}}

\newcommand{\p}{\partial}
\def\vk{\vec{k}}
\def\Nov{{\Lambda}}

\def\bG{{\overline G}}
\def\bH{{\overline H}}
\def\barf{{\overline f}}
\def\cmin{{k_{\text min}}}
\def\sec{{\mathfrak s}}

\def\index{{\text{index}}}

\def\r{{\rho}}
\def\D{{\Delta}}
\def\s{{\sigma}}
\def\lo{{\ell_0}}
\def\l{{\ell}}
\def\cmin{{c_{\text{min}}}}
\def\bmin{{b_{\text{min}}}}
\def\b{{\beta}}
\def\del{{\delta}}
\def\ci{{\nu}}
\def\br{{\overline r}}
\def\bD{{\Delta'}}
\def\bk{{\overline k}}
\def\tPhi{{\tilde\Phi}}
\def\sec{{\mathfrak s}}

\begin{document}

\title[Multiplicity of closed Reeb orbits]{Multiplicity of closed Reeb orbits on\\ contact manifolds with periodic equivariant symplectic homology}

\author[Miguel Abreu]{Miguel Abreu}
\author[Leonardo Macarini]{Leonardo Macarini}

\address{Center for Mathematical Analysis, Geometry and Dynamical Systems,
Instituto Superior T\'ecnico, Universidade de Lisboa, 
Av. Rovisco Pais, 1049-001 Lisboa, Portugal}
\email{mabreu@math.tecnico.ulisboa.pt}

\address{IMPA,
Estrada Dona Castorina, 110, Rio de Janeiro, 22460-320, Brazil}
\email{leonardo@impa.br}


\thanks{MA was partially funded by FCT/Portugal and the Recovery and Resilience Plan (PRR) through projects 2023.13969.PEX, UID/04459/2025 and UID/PRR/04459/2025. LM was partially supported by FAPERJ and CNPq/Brazil.}

\begin{abstract}
We consider closed contact manifolds $(M,\xi)$ with periodic positive equivariant symplectic homology. This is a very large class of contact manifolds and, to the best of our knowledge, includes all currently known examples admitting Reeb flows with finitely many closed orbits for which this homology is well defined. Under weak and homologically natural index assumptions on a non-degenerate contact form $\alpha$ on $M$, we establish a sharp lower bound $r_M$ for the number of simple closed Reeb orbits of $\alpha$. The quantity $r_M$ is completely defined in terms of the positive equivariant symplectic homology of $M$. Moreover, we show that this bound is attained if and only if $\alpha$ is lacunary, i.e., the Conley-Zehnder indices of all closed orbits have the same parity. Consequently, the invariant $r_M$ admits a clean dynamical characterization: whenever $M$ admits a non-degenerate lacunary contact form, $r_M$ equals the number of its simple closed Reeb orbits and is therefore independent of the choice of such a form. In particular, in the lacunary case the number of such orbits is a contact invariant completely determined by Floer theory. We compute $r_M$ for a broad class of examples, including several prequantizations of symplectic orbifolds, and show that in this case $r_M=\dim \H_*(M/S^1;\mathbb{Q})$, thereby giving a topological characterization of the invariant. Motivated by these results, we conjecture that any contact form with finitely many closed Reeb orbits is necessarily non-degenerate and lacunary, and that the underlying contact manifold is a prequantization of this type.
\end{abstract}

\maketitle

\tableofcontents

\section{Introduction}
\label{sec:intro}

The main theme of this paper is the multiplicity problem for geometrically distinct closed Reeb orbits of non-degenerate contact forms satisfying certain weak and homologically natural index conditions on a very large class of contact manifolds, namely those with \emph{periodic} positive equivariant symplectic homology. To the best of our knowledge, this class contains \emph{all currently known examples of contact manifolds admitting Reeb flows with finitely many closed orbits} for which equivariant symplectic homology is a well-defined invariant.

We establish a sharp lower bound for the number of simple closed Reeb orbits on such manifolds, expressed entirely in terms of the positive equivariant symplectic homology. Moreover, we show that this bound is attained if and only if the contact form is \emph{lacunary}, that is, the indices of all closed orbits have the same parity. Thus, the bound admits a clean \emph{dynamical} characterization: whenever a non-degenerate lacunary contact form exists, the bound is precisely the number of its simple closed orbits. In particular, in the lacunary case this number is independent of the contact form and depends only on the contact structure, thereby defining a genuine contact invariant. This invariant can be viewed as a contact analogue of the sum of the Betti numbers in Arnold's conjecture (see the discussion at the end of the introduction). We also compute it in several examples.

Our computation shows that when $M$ is a suitable prequantization bundle over a symplectic orbifold, i.e., it admits a periodic Reeb flow, the invariant $r_M$ also admits a purely \emph{topological} interpretation: it can be read off directly from the equivariant singular homology of $M$. Thus $r_M$ simultaneously carries three distinct and complementary meanings -- \emph{homological} (in terms of positive equivariant symplectic homology), \emph{dynamical} (as a count of simple closed Reeb orbits of lacunary forms), and \emph{topological} (in terms of the equivariant topology of $M$) -- reflecting the interaction between the contact dynamics and the underlying topology.

Let us introduce the problem and its roots. The multiplicity question for geometrically distinct closed Reeb orbits originates in Hamiltonian dynamics and goes back at least a century. In its modern formulation, the problem asks for a lower bound, ideally sharp, for the number of such orbits of a contact form $\alpha$ on a given contact manifold $(M^{2n+1},\xi)$. The form $\alpha$ is usually assumed to satisfy additional conditions which play both conceptual and technical roles. In this work, for instance, we focus on non-degenerate contact forms, a Morse-type condition. Two fundamental tools for studying this problem are \emph{(positive) equivariant symplectic homology} and \emph{linearized contact homology}. The latter was introduced by Eliashberg, Givental and Hofer in their seminal paper \cite{EGH}. The former was introduced by Viterbo \cite{Vit} and subsequently developed by Bourgeois and Oancea \cite{BO09,BO10,BO13a,BO13b,BO17}. As shown in \cite{BO17}, these homologies are isomorphic whenever linearized contact homology is well-defined.

When the linearized contact homology (with rational coefficients) is unbounded, that is, when the dimension of the corresponding vector spaces tends to infinity along some sequence of degrees $k_i \to \infty$, every contact form on $M$ has infinitely many simple closed orbits \cite{HM}; cf. \cite{McL1}. We refer to this situation as the \emph{homologically unbounded} case; see the survey \cite{Mac}.

Thus, for the multiplicity problem the interesting case is the \emph{homologically bounded} one. To the best of our knowledge, every homologically bounded contact manifold currently known has \emph{periodic} positive equivariant symplectic homology, that is, the positive equivariant symplectic homology groups become periodic after some degree; see Definition \ref{def:periodic}. A particular example is given by a prequantization $S^1$-bundle $(M^{2n+1},\xi)$ over a closed integral symplectic manifold $(B,\om)$ such that $\H_{\text{odd}}(B;\Q)=0$. This class includes important examples such as the standard contact sphere.

For such prequantizations, several multiplicity results are known; see the survey \cite{Mac}. Let us briefly describe the current state of the art for the standard contact sphere $(S^{2n+1},\xi)$. Conjecturally, every contact form $\alpha$ supporting $\xi$ has at least $n+1$ simple closed Reeb orbits. This conjecture has been proved when $n=1$ \cite{CGH,GHHM,LL}; however, it remains far from being settled when $n\geq2$. In general, without non-degeneracy or index/action assumptions, it is not even known whether there exists more than one simple closed Reeb orbit when $n\geq2$. When $\alpha$ is non-degenerate, it is easy to see that there must be at least two such orbits (see, e.g., \cite{HL,Gu,AGKM}), but the existence of three simple orbits on, say, $S^5$ is already a difficult open problem.

The situation changes dramatically once one imposes restrictions on the indices of closed Reeb orbits. In a series of papers starting with a groundbreaking work of Long and Zhu \cite{Lon02,LZ}, several multiplicity results were established under what is usually called the dynamical convexity assumption; see \cite{AM2,GG,GM,GKa,Wa13,Wa} and the references therein. For $S^{2n+1}$ this assumption requires that all closed Reeb orbits have Conley-Zehnder index at least $n+2$, and it follows from geometrical convexity. Under this assumption, it has been shown that a non-degenerate dynamically convex contact form on $S^{2n+1}$ must have at least $n+1$ simple closed Reeb orbits, while without the non-degeneracy assumption the number of orbits is at least $\lceil (n+1)/2\rceil +1$. In a recent outstanding work, Cineli, Ginzburg and G\"urel \cite{CGG} proved the bound $n+1$ for any dynamically convex contact form on $S^{2n+1}$, degenerate or not.

In \cite{DLLW}, the lower bound $n+1$ for non-degenerate contact forms $\alpha$ on $S^{2n+1}$ was established under index assumptions that are much weaker than dynamical convexity. More precisely, the main condition requires that, if $n$ is odd, $\alpha$ has no good periodic orbit $\ga$ with $\cz(\ga)=0$, and if $n$ is even, $\alpha$ has no good periodic orbit $\ga$ with $\cz(\ga)\in\{0,-1,1\}$. In \cite{GGM2}, these results were extended to sharp lower bounds for a broader class of prequantization bundles, including, in addition to the sphere, the unit cosphere bundle of a compact rank one symmetric space. We emphasize that multiplicity results for general contact manifolds in dimensions greater than three are very scarce.

In the present work, we obtain sharp lower bounds for the number of periodic orbits under an even weaker index assumption on $\alpha$ which, unlike the previous hypothesis, is homologically natural for general contact manifolds, in the sense that it excludes only homologically \emph{unnecessary} closed orbits; see Remarks \ref{rmk:unnecessary orbits 1} and \ref{rmk:unnecessary orbits 3}. Moreover, we prove that the lower bound is attained if and only if $\alpha$ is lacunary. Furthermore, we consider the \emph{much broader} class of contact manifolds with periodic positive equivariant symplectic homology, which includes prequantizations of symplectic \emph{orbifolds} as well as several other examples.

As already mentioned, the conjectural sharp lower bound for the number of closed orbits on $S^{2n+1}$ is $n+1$. More generally, for a prequantization bundle $M$ over a symplectic manifold, the expected bound is $r_M = \dim \H_*(M/S^1;\Q)$; see \cite{AM5,GGM2}. (Recall that $S^{2n+1}$ is a prequantization of $\CP^n$ and $\dim \H_*(\CP^n;\Q)=n+1$.) A natural and important question is what the corresponding bound should be in general.

To the best of our knowledge, every currently known example of a contact form with finitely many closed orbits is non-degenerate and lacunary. Moreover, the underlying contact manifold $M$ is always a prequantization of a symplectic \emph{orbifold}, and the bound is again
\begingroup
\setlength{\abovedisplayskip}{6pt}
\setlength{\belowdisplayskip}{6pt}
\begin{equation*}
r_M = \dim \H_*(M/S^1;\Q).
\end{equation*}
\endgroup
We prove that this bound holds under our weak and homologically natural index assumptions and that it admits an explicit homological expression in terms of the contact Betti numbers of $M$.

More generally, this homological quantity defines an invariant $r_M$ for contact manifolds with periodic positive equivariant symplectic homology. Our main result shows that $r_M$ is a sharp lower bound under the above weak index assumptions, and that equality holds precisely for lacunary contact forms. Consequently, whenever $M$ admits a non-degenerate lacunary contact form, the number of its simple closed Reeb orbits is independent of the form and completely determined by the positive equivariant symplectic homology of $M$. Equivalently, $r_M$ has a dynamical characterization: it is the number of simple closed Reeb orbits of any non-degenerate lacunary contact form on $M$.

To finish this introduction, it is illuminating to compare our results with analogous questions and results in the setting of Hamiltonian diffeomorphisms and Hamiltonian Floer theory. From the proof of Arnold's conjecture, we have that every non-degenerate Hamiltonian diffeomorphism on a suitable closed symplectic manifold $M$ has at least $\dim \H_*(M;\Q)$ periodic points (in fact, at least $\dim \H_*(M;\Q)$ fixed points). Moreover, it follows from \cite{Sh} that, under suitable assumptions, the existence of more than $\dim \H_*(M;\Q)$ periodic points implies the existence of infinitely many. The result in \cite{Sh} is stated for possibly degenerate Hamiltonian diffeomorphisms, and the relevant bound is expressed in terms of the local Floer homology of the periodic points. However, analogously to the situation of Reeb flows discussed above, all currently known examples of Hamiltonian diffeomorphisms with finitely many periodic points are, to the best of our knowledge, both non-degenerate and lacunary in the obvious analogous sense.

From this perspective, the invariant $r_M$ should be viewed as the analogue, in the contact setting, of $\dim \H_*(M;\Q)$ in Arnold's conjecture: it provides the homological lower bound for the number of simple periodic Reeb orbits.

This bound is clearly not given by $\dim \H_*(M;\Q)$ but it has a topological interpretation. As a matter of fact, we expect -- although this is far from being known -- that if a contact manifold admits a contact form with finitely many closed Reeb orbits, then it should also admit a periodic Reeb flow. In that case, the lower bound for the number of periodic orbits is given by the sum of the Betti numbers of the \emph{equivariant} singular homology $\H^{S^1}_*(M;\Q)$ with respect to the circle action generated by the Reeb flow. The present work shows in particular that, under natural assumptions which are satisfied by all currently known examples with finitely many closed Reeb orbits, this quantity coincides with $r_M$.

\section{Main results}
\label{sec:results}

\subsection{Multiplicity of closed orbits}
Let $(M^{2n+1},\xi)$ be a closed contact manifold. Along this work, we will assume that $(M,\xi)$ satisfies one of the following conditions:
\begin{itemize}
\item[\reflb{cond:F}{(F)}] $(M,\xi)$ has a strong symplectic filling $(W,\Omega)$ such that $c_1(TW)=0$ and the map induced by the inclusion $\H_1(M;\Q) \to \H_1(W;\Q)$ is trivial.
\item[\reflb{cond:NF}{(NF)}] $(M,\xi)$ satisfies $c_1(\xi)=0$, $\H_1(M;\Q)=0$ and admits a non-degenerate \emph{index-admissible} contact form, that is, a non-degenerate contact form such that every \emph{contractible} closed orbit has index bigger than $3-n$.
\end{itemize}

As explained in Section \ref{sec:ESH}, these conditions ensure that $(M,\xi)$ admits a well-defined positive equivariant symplectic homology $\HC_*(M)$, endowed with an integer grading given by the Conley-Zehnder index, which is independent of the choice of trivializations of the determinant line bundles $\Lambda^{n+1}_\C TW$ and $\Lambda^{n}_\C \xi$. Throughout, we consider $\HC_*(M)$ with coefficients in the universal Novikov field; see Section \ref{sec:ESH}.

Under some circumstances, the hypotheses involving $\H_1(W;\Q)$ and $\H_1(M;\Q)$ can be dropped; see Remarks \ref{rmk:homotopy}, \ref{rmk:H^1}, and \ref{rmk:H^1 preqorb} below. In particular, this applies to all examples known to us in which $M$ admits contact forms with finitely many closed orbits.

Condition \ref{cond:F} requires that $M$ admit a ``nice'' symplectic filling, whereas condition \ref{cond:NF} does not assume the existence of a filling, at the expense of requiring that $M$ admit a non-degenerate index-admissible contact form. Under \ref{cond:F}, the homology is defined for any contact form and may depend on the choice of filling $W$. Under \ref{cond:NF}, $\HC_*(M)$ is defined only for index-admissible contact forms and is constructed using Floer trajectories in the symplectization of $M$.

\begin{definition}
\label{def:periodic}
We say that $\HC_*(M)$ is periodic if there exist $P \in \N$ and $\Delta \in \N$ such that $\HC_k(M) \cong \HC_{k+\Delta}(M)$ for every $k\geq P$.
\end{definition}

\emph{Throughout this work, we will assume, without loss of generality, that $\D$ is even.}

\begin{remark}
One is tempted to exchange the notation $P$ and $\D$ since $\D$ is the period. However, we rather keep this notation to highlight the fact that $\HC_*(M)$ is actually \emph{eventually} periodic, or, more precisely, periodic with period $\D$ \emph{after the degree $P$}.
\end{remark}

\begin{remark}
\label{rmk:period HC}
Of course one has the freedom to choose the period $\D$ in the definition above but we can take the minimal one. In this case, in many examples we have $\D=2$ although there are also examples with $\D>2$. We are not aware of any example where $\D=1$.
\end{remark}

There is a huge class of examples of manifolds with periodic positive equivariant symplectic homology. For instance, prequantization $S^1$-bundles of integral symplectic manifolds with lacunary rational homology \cite{GGM2,AM5}, toric contact manifolds \cite{AMM2}, Brieskorn manifolds \cite{KvK,Ust} and subcritical Stein fillable contact manifolds \cite{BO17,Yau} or, more generally, contact manifolds satisfying \ref{cond:F} such that $\SH_*(W)=0$, where $\SH_*(W)$ stands for the (non-equivariant) symplectic homology of $W$ \cite{BO17}. It should be true that every closed contact manifold $M$ admitting a contact form whose Reeb flow generates a locally free circle action has periodic positive equivariant symplectic homology, although it is not known so far. In Section \ref{sec:proof r_M} we show that it is true whenever the symplectic orbifold $M/S^1$ admits a Hamiltonian circle action with isolated fixed points.

Let $b_j=\dim \HC_j(M)$ be the \emph{contact Betti numbers}. Recall that the positive/negative mean Euler characteristic of $M$ is defined as
\[
\chi_\pm(M)=\lim_{k\to\pm\infty} \frac 1k \sum_{j=0}^k (-1)^jb_j
\]
whenever this limit exists. When $\HC_*(M)$ is periodic, the limit $\chi_+(M)$ clearly exists and is given by
\[
\chi_+(M) = \frac{\r}{\Delta},
\]
where $\r=\sum_{j=P}^{P+\Delta-1} (-1)^jb_j = \sum_{j=k}^{k+\Delta-1} (-1)^jb_j$ for any $k\geq P$.

We say that a non-degenerate contact form $\alpha$ is \emph{lacunary} if the indices of every periodic orbit have the same parity. As already mentioned in the introduction, to the best of our knowledge, all the currently known examples of contact forms with finitely many closed orbits are non-degenerate and lacunary \cite{AM5}. Given a non-degenerate lacunary contact form $\alpha$ define its parity $\del \in \Z_2$ as $\cz(\ga)\,(\bmod\,2)$ where $\ga$ is any closed orbit of $\alpha$ and $\cz(\ga)$ is its Conley-Zehnder index \cite{SZ}. In a similar fashion, we say that $\HC_*(M)$ is \emph{lacunary} if the groups $\HC_k(M)$ that do not vanish have the same degree $k\,(\bmod\,2)$. If $M$ admits a non-degenerate lacunary contact form $\alpha$ then $\HC_*(M)$ is also lacunary and with the same parity $\delta$ of $\alpha$ (i.e., $\HC_k(M)=0$ for every $k \neq \del\,(\bmod\,2)$); see Section \ref{sec:ESH}. Note that if $\alpha$ (resp. $\HC_*(M)$) is lacunary then the parity of $\alpha$ (resp. $\HC_*(M)$) is well-defined whenever $\alpha$ has at least one periodic orbit (resp. $\HC_k(M)\neq 0$ for some $k\in \Z$).

Our first main result establishes a lower bound for the number of simple closed orbits of suitable contact forms on manifolds with periodic and lacunary positive equivariant symplectic homology which is an equality if and only if the contact form is lacunary. Although the statement is admittedly somewhat technical, we present it in this general form because its strength lies precisely in its broad applicability. Concrete applications to a variety of examples are given in Section \ref{sec:applications}.

Before stating the theorem, we introduce some notation and terminology. In what follows, $\#\P_\alpha$ denotes the number of simple (i.e., non-iterated) closed orbits of $\alpha$. Recall that the mean index of a periodic orbit $\ga$ is defined by
\[
\mi(\ga) = \lim_{k\to\infty} \frac{1}{k}\,\cz(\ga^k).
\]
A periodic orbit is said to be \emph{good} if the parity of its index agrees with that of the underlying simple closed orbit.

We define a prequantization $S^1$-bundle of a symplectic manifold (resp. orbifold) $(B,\om)$ as a contact manifold $M$ endowed with a contact form $\beta$ whose Reeb flow generates a free (resp. locally free) circle action such that $M/S^1 = B$ and $d\beta = \pi^*\om$, where $\pi \colon M \to B$ is the quotient projection.

We say that a contact manifold $M$ with periodic $\HC_*(M)$ satisfies condition \reflb{cond:PD}{(PD)} if one of the following holds:
\begin{itemize}
\item the period $\D$ of $\HC_*(M)$ is equal to $2$;
\item $M$ is a prequantization of a positive monotone symplectic manifold $B$ with $\H_{\text{odd}}(B;\Q)=0$ and, under \ref{cond:NF}, minimal Chern number greater than $1$;
\item $M$ admits a non-degenerate lacunary contact form.
\end{itemize}

This condition ensures that $\HC_*(M)$ exhibits a symmetry in sufficiently large degrees, reminiscent of a form of Poincar\'e duality (hence the terminology); see Remark \ref{rmk:PD}.

\begin{theorem}
\label{thm:main}
Let $(M^{2n+1},\xi)$ be a closed contact manifold satisfying either \ref{cond:F} or \ref{cond:NF}. Assume that $\HC_*(M)$ is periodic, lacunary with parity $\del$, $b_j \neq 0$ for some $j$, $b_j = 0$ for all $j$ sufficiently negative, and $b_j < \infty$ for all $j \in \Z$.

Let $\alpha$ be a non-degenerate contact form on $M$. Assume that there exists an integer $p \neq \del\,(\bmod\,2)$ such that $\alpha$ satisfies the following conditions:
\begin{itemize}
\item $\alpha$ has no good periodic orbit $\ga$ such that $\cz(\ga)=\pm p$;
\item every closed orbit of $\alpha$ has non-vanishing mean index;
\item under \ref{cond:NF}, $\alpha$ is index-admissible. 
\end{itemize}
If $p>1$ assume that $M$ satisfies \ref{cond:PD}. Then
\[
\#\P_\alpha \geq r_M,
\]
where
\begin{equation*}
r_M=
\begin{cases}
2(-q\chi_+(M) - \sum_{j=-\infty}^q b_j) \text{ if } \del\text{ is odd} \\
2(q\chi_+(M) - \sum_{j=-\infty}^q b_j) + b_q \text{ if } \del\text{ is even,}
\end{cases}
\end{equation*}
and
\[
q=\min\{s\D;\,s\D \geq P,\,s \in \N\}.
\]
Moreover, the equality holds if and only if $\alpha$ is lacunary.
\end{theorem}

\begin{remark}[The bound $r_M$ is well-defined]
The fact that $r_M$ does not depend on the choices of $P$ and $\D$ is an easy consequence of the periodicity of $\HC_*(M)$; see Proposition \ref{prop:periodicity}. We can actually take $q$ as any multiple of $\D$ bigger than or equal to $P$.
\end{remark}

\begin{remark}[The homological, dynamical and topological meanings of $r_M$]
Conceptually, the number $r_M$ measures the finite-dimensional defect between the initial part of the positive equivariant symplectic homology and its eventual periodic average. Thus, $r_M$ is a purely homological quantity, in the same sense that the sum of the Betti numbers is the homological quantity appearing in Arnold's conjecture. One of the main points of Theorem \ref{thm:main} is that this homological quantity admits an exact dynamical realization: whenever a non-degenerate lacunary contact form exists, $r_M$ coincides with the number of its simple closed Reeb orbits. This interpretation is made precise in Corollary \ref{cor:r_M lacunary}  and in the discussion preceding it. Furthermore, as proved in Theorem \ref{thm:r_M} below, we also have a topological characterization of $r_M$: when $M$ is a suitable prequantization of an orbifold, i.e. it admits a periodic Reeb flow, then $r_M=\dim \H_*^{S^1}(M;\Q)$ where the $S^1$-action is the one induced by the periodic Reeb flow.
\end{remark}

\begin{remark}[Naturality of the assumptions on $M$]
To the best of our knowledge, every closed contact manifold known so far that satisfies either \ref{cond:F} or \ref{cond:NF} and admits a contact form with finitely many closed orbits satisfies all the assumptions on $M$ above. Recall that \ref{cond:F} and \ref{cond:NF} are the conditions to have a well-defined equivariant symplectic homology with an integer grading using a suitable trivialization.
\end{remark}

\begin{remark}[The hypothesis that $\alpha$ has no good periodic orbit $\ga$ such that $\cz(\ga)=\pm p$]
\label{rmk:unnecessary orbits 1}
It turns out that $\HC_*(M)$ can be obtained as the homology of a complex generated by the good periodic orbits of $\alpha$ graded by the index; see Section \ref{sec:ESH}. Therefore, if $\HC_k(M)\neq 0$ then every non-degenerate contact form must have a combination of good closed orbits with index $k$ that represents a non-trivial homology class; such orbits are homologically \emph{necessary}. Consequently, the assumption that $\alpha$ has no good periodic orbit $\ga$ such that $\cz(\ga)=\pm p$ for some $p \neq \del\,(\bmod\,2)$ excludes periodic orbits which are \emph{unnecessary} from the homological point of view; cf. Remark \ref{rmk:unnecessary orbits 4}. In particular, these conditions are automatically satisfied by lacunary contact forms, where \emph{all} the periodic orbits are homologically necessary.

Let us emphasize that this is much weaker than the index hypotheses imposed in previous multiplicity results, such as dynamical convexity or the exclusion of whole ranges of low-index orbits. We expect this hypothesis to be technical and ultimately removable. Indeed, because $\HC_*(M)$ is lacunary with parity $\del$, the homology in degrees $\pm p$, with $p\neq \del\,(\bmod\,2)$, vanishes. Hence the excluded orbits are not among those forced by the positive equivariant symplectic homology of $M$. Their presence would have to be compensated by cancellations in the chain complex, and one expects such cancellations to require additional generators. For this reason, we expect the inequality $\#\P_\alpha \geq r_M$, and the characterization of the equality case, to remain true without this assumption. The present proof, however, uses the absence of these homologically unnecessary orbits in an essential way; cf. Remarks  \ref{rmk:unnecessary orbits 3}, \ref{rmk:proof} and \ref{rmk:unnecessary orbits 4}.
\end{remark}

\begin{remark}[The hypothesis that every closed orbit of $\alpha$ has non-vanishing mean index]
\label{rmk:unnecessary orbits 2}
In a similar vein to the previous remark, the assumption that $b_j<\infty$ for every $j \in \Z$ implies that closed orbits with zero mean index are \emph{unnecessary} from the homological point of view. In particular, if $\alpha$ is non-degenerate and lacunary then this assumption forces that every closed orbit of $\alpha$ has non-zero mean index. Indeed, if $\alpha$ has a periodic orbit with vanishing mean index then the contribution of its iterates in the chain complex $\CC_*(\alpha)$ of $\HC_*(M)$ must lie in the degree range $(-n,n)$. Consequently, we have that the dimension of $\CC_*(\alpha)$ in at least one of such degrees must be infinite.
\end{remark}

\begin{remark}[Comparison with previous results]
\label{rmk:unnecessary orbits 3}
In \cite{GGM2}, Ginzburg, G\"urel, and the second author, generalizing the aforementioned work \cite{DLLW}, established a special case of the lower bound in Theorem \ref{thm:main}, namely when $M$ is a suitable prequantization of a smooth symplectic manifold $B$ with lacunary $\H_*(B;\Q)$; see also \cite{DLLW2}. The main assumption in that setting is the absence of closed orbits $\ga$ with $\cz(\ga)\in\{0,-1,1\}$ when $\delta$ is even, and with $\cz(\ga)=0$ when $\delta$ is odd, where in this case $\delta = n \pmod{2}$. However, as explained in Remark \ref{rmk:unnecessary orbits 1}, the exclusion of periodic orbits with index zero when $\delta$ is even may not be possible in general, since there exist examples satisfying the hypotheses of Theorem \ref{thm:main} for which $\HC_0(M)\neq 0$. This motivates a generalization and weakening of that condition to the hypothesis used here, which is more natural from a homological perspective and imposes no such obstruction.

Moreover, even in the case of the sphere $S^{2n+1}$, our result provides a substantial refinement of the results in \cite{DLLW,DLLW2,GGM2}. Indeed, while the condition excluding closed orbits $\ga$ with $\cz(\ga)\in\{0,-1,1\}$ when $n$ is even, and with $\cz(\ga)=0$ when $n$ is odd, can be viewed as a form of ``very weak dynamical convexity'', the hypothesis considered here has a different conceptual nature: it is homological in flavor, amounting to the exclusion of periodic orbits in prescribed degrees $\pm p$ in which $\HC_*(M)$ vanishes. Furthermore, we prove that equality in the lower bound holds if and only if the contact form is lacunary.

The results established in this work are new even in several classical and well-studied examples. To the best of our knowledge, prior to this work, sharp lower bounds for the number of periodic Reeb orbits had not been established in the generality considered here. In the case of toric contact manifolds, for instance, the existing literature produced sharp bounds essentially only for the standard contact sphere and a small number of very special prequantization bundles; no sharp results were available for the broad class of toric contact manifolds covered here; see Section \ref{sec:applications}. Similarly, for several other examples discussed in Section \ref{sec:applications}, the sharpness of the lower bounds obtained here appears for the first time. Thus, beyond providing a unified framework, the present paper yields genuinely new quantitative information in these settings.
\end{remark}

\begin{remark}[The lacunarity assumption]
\label{rmk:non-lacunary}
The lacunarity assumption on $\HC_*(M)$ is used in Theorem \ref{thm:main} mainly to obtain a canonical, sharp value of the lower bound $r_M$. As a matter of fact, we prove Theorem \ref{thm:main} as a consequence of the more general Theorems \ref{thm:main1} and \ref{thm:main2} where we do not need, for the multiplicity problem, to assume either that $\HC_*(M)$ is lacunary or that $b_j\neq 0$ for some $j \in \Z$. The reason for the last hypothesis is simply to define the parity of $\HC_*(M)$ when $\HC_*(M)$ is lacunary. (If $\HC_*(M)$ vanishes identically it is obviously lacunary but it does not have a well-defined parity. However, under this situation, the lower bound below vanishes.) In this more general situation, given an integer $p\geq 0$ we still have the inequality $\#\P_\alpha \geq r_M$ with 
\begin{equation}
\label{eq:r_M general}
r_M=
\begin{cases}
2(-q\chi_+(M) + \sum_{j=-\infty}^q (-1)^jb_j) \text{ if } p=0 \\
2(-q\chi_+(M) + \sum_{j=-\infty}^q (-1)^jb_j) + 2b_{q+p} - b_q \text{ if }\ p\ \text{is even and}\ \geq 2, \\
2(q\chi_+(M) - \sum_{j=-\infty}^q (-1)^jb_j) + 2b_{q+p} + b_q \text{ if } p\ \text{ is odd},
\end{cases}
\end{equation}
with the same main assumption on $\alpha$ as in Theorem \ref{thm:main}: $\alpha$ has no good periodic orbits $\ga$ such that $\cz(\ga)=\pm  p$. In this more general context, $r_M$ depends on the choice of $p$ and the inequality $\#\P_\alpha \geq r_M$ is no longer sharp in general; see Remark \ref{rmk:unnecessary orbits 4}. Assuming that $\D$ is even, as we do in this work, clearly the formula for $r_M$ above coincides with the definition of $r_M$ given in Theorem \ref{thm:main} when $\HC_*(M)$ is lacunary with parity $\del \neq p\,(\bmod\, 2)$ because $q$ is even.
\end{remark}

\begin{remark}[A few words about the proof of Theorem \ref{thm:main}]
\label{rmk:proof}
As already explained in Remark \ref{rmk:non-lacunary} above, Theorem \ref{thm:main} follows from the more general Theorems \ref{thm:main1} and \ref{thm:main2}. Explaining the proofs of these results at this stage would require substantial notation and lead us into rather technical considerations. We therefore limit ourselves to mentioning that the main ingredients are equivariant symplectic homology and index theory. Very roughly speaking, the proof uses in a highly non-trivial way the Morse inequalities for positive equivariant symplectic homology combined with resonance relations and index recurrence properties of the iterates of finitely many simple orbits. The exclusion of closed orbits with the homologically unnecessary degrees allows certain terms to cancel, leading to the desired sharp lower bound. An outline of the argument is provided in Section \ref{sec:proofmain1-outline}.
\end{remark}

\begin{remark}[Decomposition in free homotopy classes]
\label{rmk:homotopy}
Assuming that $M$ satisfies \ref{cond:F} (resp.\ \ref{cond:NF}), we have, as discussed in Section \ref{sec:ESH}, that $\HC_*(M)$ admits a decomposition according to the free homotopy classes in $W$ (resp. $M$). Let $\Gamma$ be a subset of free homotopy classes in $W$ (resp. $M$) that is \emph{closed under iterations}, i.e., if $a \in \Gamma$, then $a^k \in \Gamma$ for every $k \in \N$. Then all the results above remain valid when restricted to closed orbits whose free homotopy classes lie in $\Gamma$, with the obvious modifications: one considers the component $\HC^\Gamma_*(M)$ of $\HC_*(M)$ corresponding to $\Gamma$, and defines a non-degenerate lacunary contact form by requiring that the indices of all closed orbits with free homotopy class in $\Gamma$ have the same parity.

In \cite{AM5,GGM2}, it is considered only contractible orbits. In this case, the condition that $c_1(TW)=0$ (resp. $c_1(\xi)=0$) can be weakened to the hypothesis that  $c_1(TW)|_{\pi_2(W)}=0$ (resp. $c_1(\xi)|_{\pi_2(M)}=0$). Moreover, we can also drop the assumptions regarding $\H_1(W;\Q)$ and $\H_1(M;\Q)$. More generally, we can drop these last hypotheses whenever the image of $\Gamma$ under the tensorized Hurewicz map $\pi_1(W)/[\pi_1(W),\pi(W)] \to \H_1(W;\Z) \otimes \Q$ (resp. $\pi_1(M)/[\pi_1(M),\pi(M)] \to \H_1(M;\Z) \otimes \Q$) vanishes.
\end{remark}

\begin{remark}[Further discussion about lacunarity]
\label{rmk:parities n and del}
To the best of our knowledge, all currently known examples of contact forms with finitely many closed orbits are non-degenerate and lacunary, with parity $\del$ equal to that of $n$. The reason why $\del = n \,(\bmod\,2)$ is that, in all these examples, the periodic orbits are elliptic, and it is well known that $\cz(\ga) = n \,(\bmod\,2)$ whenever $\ga$ is elliptic. However, there exist non-degenerate lacunary contact forms whose parity does not coincide with that of $n$. Indeed, consider a closed Riemannian manifold $N$ of dimension $d$ with negative sectional curvature. Then the Liouville form on $M = S^*N$ is non-degenerate and lacunary with even parity, while $n = d-1$ may be odd (note that $M$ has dimension $2n+1$). This follows from the well-known fact that every closed geodesic on $N$ is non-degenerate and has Morse index zero. Observe, however, that this contact form has infinitely many simple closed orbits, all of which are hyperbolic; cf.\ Remark \ref{rmk:assumption finite b_j} below. In this example, it is known that every contact form on $M$ has infinitely many closed orbits; see \cite{MP}. As already noted in \cite{AM5}, we are not aware of examples of lacunary contact forms with both elliptic \emph{and} hyperbolic orbits. Note also that the natural filling of $S^*N$ given by the unit codisk bundle $D^*N$ does not satisfy, in this example, the hypotheses involving $\H_1(D^*N;\Q)$ and $\H_1(S^*N;\Q)$ used to define the index of closed orbits independently of the choice of a trivialization of $\Lambda^{n}_\C TD^*N$, as in \ref{cond:F}; see Section \ref{sec:ESH}. Here we choose trivializations so that the Conley-Zehnder index coincides with the Morse index; see, for instance, \cite{AS}. Nevertheless, the parity of the index does not depend on the choice of these trivializations.
\end{remark}

\begin{remark}[The assumption that $b_j < \infty$ for every $j \in \Z$ ]
\label{rmk:assumption finite b_j}
The assumption that $b_j < \infty$ for every $j \in \Z$ is necessary to ensure that if $\alpha$ is lacunary then it has finitely many periodic orbits. Indeed, consider, as in Remark \ref{rmk:parities n and del}, a closed Riemannian manifold $N$ with negative sectional curvature. Then $M=S^*N$ is a contact manifold satisfying \ref{cond:F} such that $b_k=0$ for every $k \neq 0$ and $b_0=\infty$. In particular, $\HC_*(M)$ is periodic and lacunary. However, the geodesic flow is lacunary and has infinitely many simple closed orbits. 
\end{remark}

\begin{remark}[The meaning of property \ref{cond:PD}]
\label{rmk:PD}
It follows from Propositions \ref{prop:PD preq} and \ref{prop:PD finite}, together with the periodicity of $\HC_*(M)$, that if $M$ satisfies \ref{cond:PD}, then $\HC_*(M)$ exhibits the following symmetry. Let $s \in \N$ be such that $s\D \geq P$. Then $\HC_*(M)$ is symmetric with respect to reflection about the middle degree $s\D + \D/2$ (recall that $\D$ is assumed to be even) within the range of degrees $[s\D,(s+1)\D]$. More precisely, we have $\HC_{(s\D + \D/2)-j}(M) \cong \HC_{(s\D + \D/2)+j}(M)$ for all $j \in [0,\D/2]$. This is the property used in the proof of Theorem \ref{thm:main}.
\end{remark}

By assumption, there exists $D \in \Z$ such that $b_j=0$ for every $j<D$. The number $r_M$ in Theorem \ref{thm:main} measures, for $\del$ odd, how much the finite sum
\[
\sum_{j=-\infty}^{q} b_j = \sum_{j=D}^{q} b_j
\]
deviates from the average $-q\chi_+(M)$ (note that if $\del$ is odd then $\chi_+(M)$ is negative since $b_{\mathrm{even}}=0$), or, more precisely, its ``asymptotic average'', that is, $q$ times the asymptotic average
\[
-\chi_+(M) = \lim_{k\to\infty} \frac 1k \sum_{j=D}^k b_j.
\]
When $\del$ is even we have the extra term $b_q$.

The computation of $r_M$ can be difficult in general. However, the following immediate consequence of Theorem \ref{thm:main} is one of the main conceptual points of this work and provides the aforementioned dynamical characterization of $r_M$.

\begin{corollary}
\label{cor:r_M lacunary}
Let $(M^{2n+1},\xi)$ be a closed contact manifold satisfying either \ref{cond:F} or \ref{cond:NF}. Assume that $\HC_*(M)$ is periodic, lacunary, $b_j \neq 0$ for some $j$, $b_j = 0$ for all $j$ sufficiently negative, and $b_j < \infty$ for all $j \in \Z$. Suppose that $M$ admits a non-degenerate lacunary contact form $\alpha$ with parity $\del$. Then $\#\P_{\alpha}$ and $\del$ do not depend on the choice of $\alpha$ and
\[
r_M = \#\P_{\alpha}.
\]
In particular, $\#\P_{\alpha}$ is a contact invariant completely determined by $\HC_*(M)$.
\end{corollary}

The assertion that $\#\P_{\alpha}$ is a contact invariant determined by $\HC_*(M)$ is a consequence of Theorem \ref{thm:main} and the fact that if $M$ admits a non-degenerate lacunary contact form then $\HC_*(M)$ does not depend on the filling; see Section \ref{sec:ESH}.

Using this, one can easily compute $r_M$ in, as far as we know, \emph{every currently known example of contact manifold admitting Reeb flows with finitely many closed orbits} satisfying either \ref{cond:F} or \ref{cond:NF}. Indeed, to the best of our knowledge, all examples of contact forms $\alpha$ with finitely many closed orbits known so far are non-degenerate and lacunary and the corresponding contact manifolds are prequantizations of orbifolds admitting a Hamiltonian circle action with isolated fixed points \cite{AM5}.

More precisely, the following result computes $r_M$ in plenty of examples and furnishes a topological characterization of it when $M$ is a suitable prequantization of an orbifold. The classes listed below are not disjoint and, in fact, exhibit significant overlap. As mentioned above, to the best of our knowledge, they encompass all currently known examples of contact manifolds admitting Reeb flows with finitely many closed orbits under either condition \ref{cond:F} or \ref{cond:NF}; see Section \ref{sec:applications}.

As noted in Remark \ref{rmk:non-lacunary}, although Theorem \ref{thm:main} is stated for contact manifolds for which $\HC_*(M)$ is lacunary, it is in fact a consequence of a more general result (Theorem \ref{thm:main1}) that does not require this assumption. Accordingly, we include in the next theorem a case -- namely, case \ref{VSH} -- in which  $\HC_*(M)$ is not necessarily lacunary.

Some words are due before the statement of the next theorem. For a definition and brief discussion about toric contact manifolds we refer to Section \ref{sec:applications}. Let us just mention here that any toric contact manifold $M$ admits a toric contact form whose Reeb flow generates a locally free circle action and $M/S^1$ below refers to the quotient for any such action. We say that a (possibly degenerate) contact form has index-positivity if there are constants $c>0$ and $c'>0$ such that $\rs(\ga) > cT+c'$ for every closed Reeb orbit $\ga$, where $T$ is the period of $\ga$ and $\rs(\ga)$ is its Robbin-Salamon index \cite{RS}. In what follows, $\N$, as usual, is the set of positive integers and $\N_0=\N \cup \{0\}$.

\begin{theorem}
\label{thm:r_M}
Let $(M^{2n+1},\xi)$ be a closed contact manifold satisfying either \ref{cond:F} or \ref{cond:NF}. Suppose that it is one of the following manifolds:
\begin{itemize}
\item [\reflb{PM}{(PM)}] A prequantization $S^1$-bundle of a closed positive monotone symplectic manifold $(B,\om)$ such that $\H_{k}(B;\Q)=0$ for every odd $k$. Under \ref{cond:NF}, assume that the minimal Chern number of $B$ is bigger than one. Then $r_M = \dim \H_*(B;\Q)$.
\item [\reflb{T}{(T)}] A good toric contact manifold. Under \ref{cond:NF}, assume that it admits an index-admissible non-degenerate toric contact form. Then $r_M=\dim \H_*(M/S^1;\Q)$. The same holds using a combinatorial description of $\HC_*(M)$ given in \cite{AMM2} which does not require assumptions \ref{cond:F} or \ref{cond:NF}. In this case, $r_M \geq n+1$ with equality holding if and only if $M$ is the standard contact sphere, or more generally, a lens space.
\item [\reflb{U}{(U)}] An Ustilovsky sphere, that is, a Brieskorn manifold $\Sigma(a_0,a_1,\dots,a_{n+1})$ with $n$ even, $a_0 = \pm 1\,(\bmod\,8)$ and $a_1=a_2=\cdots=a_{n+1}=2$. Then $r_M=n+1$.
\item [\reflb{PO}{(PO)}] A prequantization $S^1$-bundle of a closed symplectic orbifold $B$ admitting a Hamiltonian circle action with isolated fixed points. Assume that the contact form $\beta$ that generates the $S^1$-action has index-positivity and, under \ref{cond:NF}, that every contractible orbit of $\beta$ has Robbin-Salamon index $\geq 4$. Then $r_M= \dim \H_*(B;\Q)$ and this number coincides with the number of fixed points of the Hamiltonian circle action on $B$.
\item [\reflb{VSH}{(VSH)}] $M$ satisfies \ref{cond:F} with the filling $W$ such that $\SH_*(W)=0$. Let $\chi(W)=\sum_{j=0}^{2n+2} (-1)^j\H^j(W;\Q)$ be the Euler characteristic of $W$ and $\b_j=\dim \H^j(W;\Q)$ be the Betti numbers. Then, given an integer $p\geq 0$ with parity different from that of $n$,
\begin{equation*}
r_M=
\begin{cases}
(n+1)\chi(W) + 2\sum_{j=-n+1}^{n+1} (-1)^j\sum_{i=1}^{\lfloor (j+n+1)/2 \rfloor} \b_{n-j+2i}\text{ if }n\text{ is odd and }p=0; \\
(n+1)\chi(W) + 2\sum_{j=-n+1}^{n+1} (-1)^j\sum_{i=1}^{\lfloor (j+n+1)/2 \rfloor} \b_{n-j+2i} + \sum_{j\in 2\N - 1} \b_j\text{ if }n\text{ is odd and }p\geq 2; \\
(n+2)\chi(W) - 2\sum_{j=-n+1}^{n+2} (-1)^j\sum_{i=1}^{\lfloor (j+n+1)/2 \rfloor} \b_{n-j+2i} + 2\sum_{j\in 2\N-1} \b_j + \sum_{j\in 2\N_0} \b_j\text{ if }n\text{ is even},
\end{cases}
\end{equation*}
where $r_M$ is given by \eqref{eq:r_M general}.
\end{itemize}
\end{theorem}

\begin{remark}[The possible values that $r_M$ can take]
As discussed above, a central question in Hamiltonian Dynamics is to determine the sharp lower bound for the number of periodic orbits of a Reeb flow on a given contact manifold. In the case of the standard contact sphere $S^{2n+1}$, this bound is $n+1$. However, as the examples above illustrate, this value is in general far from sharp: in many situations one has $r_M > n+1$. (In fact, for fixed dimension $\geq 5$, one can easily construct examples with arbitrarily large $r_M$; cf.\ \cite{AM3,AGZ}.) At present, we are not aware of any example where the bound is strictly less than $n+1$. It is therefore natural to ask whether $n+1$ is the minimal possible lower bound, that is, whether $r_M \geq n+1$ for every contact manifold $M$ of dimension $2n+1$, with equality if and only if $M$ is the sphere, or more generally, a lens space. The final assertion in case \ref{T} provides evidence in support of this expectation.
\end{remark}

\begin{remark}[The hypothesis on the map $\H_1(M;\Q) \to \H_1(W;\Q)$ in case \ref{VSH}]
\label{rmk:H^1}
In case \ref{VSH} (vanishing symplectic homology) we do not need to assume that the map induced by the inclusion $\H_1(M;\Q) \to \H_1(W;\Q)$ is trivial in \ref{cond:F}. As a matter of fact, as explained in Remark \ref{rmk:homotopy} we do not need this condition if we consider only orbits contractible in $W$. But, in this case, $\HC^0_*(M)=\HC_*(M)$ due to the following. As explained in Section \ref{sec:ESH}, under \ref{cond:F}, $\HC_*(M)$ has a decomposition in free homotopy classes $a$ of $W$. So the assertion $\HC^0_*(M)=\HC_*(M)$ is equivalent to the claim that $\HC^a_*(M)=0$ for every $a\neq 0$. We have that $\SH_*(W)=0$ if, and only if, $\SH^{S^1}_*(W)=0$; see \cite{Sei} and \cite[Theorem 4.1]{BO17}. As also explained in Section \ref{sec:ESH}, besides the notation, $\HC_*(M)$ is the positive part $\SH^{S^1,+}_*(W)$ of the equivariant symplectic homology $\SH^{S^1}_*(W)$. We have the tautological exact triangle
\[
\xymatrix{
\SH^{S^1,-}_*(W)\ar[rr]& & \SH^{S^1}_*(W)\ar[ld]\\
& \HC_{*}(M)\ar[ul]^{[-1]}&
},
\]
which respects the decomposition in free homotopy classes. But $\SH^{S^1,-,a}_*(W)=0$ for every $a\neq 0$. (In other words, the negative part of $\SH^{S^1}_*(W)$ only sees contractible orbits.) Thus, if $\SH^{S^1}_*(W)=0$ then $\HC^a_*(M)=0$ for every $a\neq 0$.
\end{remark}

\begin{remark}[The parity of $p$ in case \ref{VSH}]
The reason for requiring that $p$ have parity different from that of $n$ in case \ref{VSH} is that, as already mentioned, to the best of our knowledge all currently known examples of contact manifolds $M$ admitting Reeb flows with finitely many closed orbits have the property that $\HC_*(M)$ is lacunary with parity equal to that of $n$. Nevertheless, the computation carried out in Section \ref{sec:proof r_M} can be used to determine $r_M$ in case \ref{VSH} for arbitrary values of $p$.
\end{remark}

\begin{remark}[The long formula for $r_M$ in case \ref{VSH}]
The key reason why we have such a long expression in the formula of $r_M$ in case \ref{VSH} is that we do not assume in this situation that $\HC_*(M)$ is lacunary. As noticed above, we included this case to illustrate the fact that our results can be applied to this more general setting.
\end{remark}

\begin{remark}[The hypotheses on $B$ in case \ref{PM}]
In case \ref{PM}, the assumption that the minimal Chern number of $B$ is greater than one under condition \ref{cond:NF} ensures that the contact form $\beta$ generating the $S^1$-action admits a non-degenerate index-admissible perturbation; see \cite{AM5,GGM2}. The hypothesis that $B$ is positive monotone is used to guarantee that $\beta$ is index-positive. Moreover, the monotonicity of $B$ (regardless of the value of the monotonicity constant) implies that $c_1(\xi)|_{\H_2(M;\Q)}=0$, and hence, in particular, $c_1(\xi)|_{\pi_2(M)}=0$. Consequently, as observed in Remark \ref{rmk:homotopy}, when restricting to contractible orbits, the assumption $c_1(\xi)=0$ in \ref{cond:NF} is not needed.
\end{remark}

\begin{remark}[Computation of $\HC_*(M)$ in case \ref{PO}]
In case \ref{PO} we compute $\HC_*(M)$, given by \eqref{eq:HCorb}, which implies, in particular, that $\HC_*(M)$ is periodic. It is also lacunary because $M$ admits a non-degenerate lacunary contact form; see the proof in Section \ref{sec:proof r_M}. The computation of $\HC_*(M)$ should also follow from \cite{KvK} under some additional assumptions. Our proof here,  however, is quite different and works at the chain level.
\end{remark}

\begin{remark}[The combinatorial meaning of $r_M$ in case \ref{T}]
\label{rmk:toric edges}
In case \ref{T}, the number $r_M=\dim \H_*(M/S^1;\Q)$ has a clean toric description: it is equal to the number of edges of the moment cone associated to $M$ \cite{AM0}. Equivalently, it is the number of facets of the corresponding toric diagram \cite{AM3}.
\end{remark}

\begin{remark}[Computation of $r_M$ in case \ref{PM} without assuming that $\H_{\text{odd}}(B;\Q)=0$]
\label{rmk:PM general}
Similarly as in case \ref{VSH}, we compute $r_M$ in \ref{PM} in Section \ref{sec:proof r_M} without the hypothesis that $\H_{\text{odd}}(B;\Q)=0$ obtaining \eqref{eq:r_M-PM-general}. However, we rather state the result under this hypothesis because the outcome is much cleaner and, as already mentioned, it covers the currently known examples of prequantizations of smooth manifolds admitting contact forms with finitely many closed orbits. (Note that if a symplectic manifold $B$ admits a Hamiltonian circle action then $\H_{\text{odd}}(B;\Q)=0$.)
\end{remark}

The computation of $r_M$ in the previous theorem in cases \ref{T} and \ref{U} follows from Corollary \ref{cor:r_M lacunary} together with the well-known fact that such manifolds admit non-degenerate lacunary contact forms with precisely
\[
\dim \H_*(M/S^1;\Q)\text{ and }n+1
\]
closed orbits, respectively; see \cite{AM0,Ust}. In case \ref{PM}, the same conclusion holds provided that $B$ admits a Hamiltonian circle action with isolated fixed points; see, for instance, \cite{AM5}, as well as the proof of Theorem \ref{thm:r_M} in case \ref{PO} given in Section \ref{sec:proof r_M}, where we show that in case \ref{PO} there exist non-degenerate lacunary contact forms with $\dim \H_*(M/S^1;\Q)$ closed orbits. The general case of \ref{PM} requires a separate argument, which will be presented in Section \ref{sec:proof r_M}.

In the same section, we also provide an alternative proof in case \ref{T} that does not require the assumption that $(M,\xi)$ satisfies either \ref{cond:F} or \ref{cond:NF}, using a combinatorial description of $\HC_*(M)$ introduced in \cite{AMM2}. In a similar vein, we give a short alternative proof in the case \ref{PM} $\cap$ \ref{T} (i.e. a prequantization of a toric manifold) when $[\omega]=c_1(TB)$. Finally, the proof in case \ref{VSH}, also presented in Section \ref{sec:proof r_M}, is a consequence of the computation of $\HC_*(M)$ in this setting, due to \cite{BO17}.

By Theorem \ref{thm:r_M} one can see that Theorem \ref{thm:main} (see also Theorems \ref{thm:main1} and \ref{thm:main2}) is a powerful generalization and strengthening of the main results from \cite{AM5,GGM2}, where only prequantizations of symplectic manifolds are considered. Moreover, in \cite{GGM2} there is a restrictive assumption on the minimal Chern number of the basis.

Motivated by the results in this work, we conjecture that every closed contact manifold admitting contact forms with finitely many simple closed orbits is a prequantization of an orbifold, and that the minimal number of such orbits is given by $\dim \H_*(M/S^1;\Q)$; see Section \ref{sec:questions}. Therefore, Theorems \ref{thm:main} and \ref{thm:r_M} represent a step towards a positive answer to this conjecture and Theorem \ref{thm:r_M} also shows how this bound can be seen from $\HC_*(M)$, in terms of the deviation of the sum of the contact Betti numbers from the corresponding asymptotic average, as explained before.

\subsection{Multiplicity of non-hyperbolic closed orbits}
\label{sec:non-hyp}

Now we will show some results about the multiplicity of non-hyperbolic closed orbits when the contact form has finitely many geometrically distinct closed orbits which are byproducts of the proof of Theorem \ref{thm:main}. Recall that a closed orbit is hyperbolic if its linearized Poincar\'e map has no eigenvalues on the unit circle. In what follows, $\#\P^{\text{non-hyp}}_\alpha$ stands for the number of simple non-hyperbolic closed orbits of $\alpha$.

\begin{theorem}
\label{thm:non-hyp}
Let $(M^{2n+1},\xi)$ be a closed contact manifold satisfying either \ref{cond:F} or \ref{cond:NF}. Assume that $\HC_*(M)$ is periodic, lacunary with parity $\del$, $b_j \neq 0$ for some $j$, $b_j = 0$ for all $j$ sufficiently negative, and $b_j < \infty$ for all $j \in \Z$.

Let $\alpha$ be a non-degenerate contact form on $M$. Assume that there exists an integer $p \neq \del\,(\bmod\,2)$, which we assume to be equal to $1$ if $p$ is odd, such that $\alpha$ satisfies the following conditions:
\begin{itemize}
\item $\alpha$ has finitely many simple closed orbits;
\item $\alpha$ has no good periodic orbit $\ga$ such that $\cz(\ga)=\pm p$;
\item every closed orbit of $\alpha$ has non-vanishing mean index;
\item under \ref{cond:NF}, $\alpha$ is index-admissible. 
\end{itemize}
If $p>1$ assume that $M$ satisfies \ref{cond:PD}. Then
\[
\#\P^{\text{non-hyp}}_\alpha \geq r_M^{\text{non-hyp}}, 
\]
where
\begin{equation*}
r_M^{\text{non-hyp}}:=
\begin{cases}
2(-q\chi_+(M) - \sum_{j=-\infty}^q b_j) \text{ if } \del\text{ is odd} \\
2(q\chi_+(M) - \sum_{j=-\infty}^{q} b_j) + 2b_0 \text{ if } \del\text{ is even,}
\end{cases}
\end{equation*}
and $q=\min\{s\D;\,s\D \geq P,\,s \in \N\}$.
\end{theorem}

As in Theorem \ref{thm:main}, Theorem \ref{thm:non-hyp} follows from a more general result, Theorem \ref{thm:non-hyp-general}, which does not assume that $\HC_*(M)$ is lacunary, presented in Section \ref{sec:proofmain}.

Combining Theorems \ref{thm:main} and \ref{thm:non-hyp} we obtain the following corollary.

\begin{corollary}
\label{cor:non-hyp-lacunary}
Let $(M^{2n+1},\xi)$ be a closed contact manifold satisfying either \ref{cond:F} or \ref{cond:NF}. Assume that $\HC_*(M)$ is periodic, $b_j \neq 0$ for some $j$, $b_j = 0$ for all $j$ sufficiently negative, and $b_j < \infty$ for all $j \in \Z$.

Let $\alpha$ be a non-degenerate lacunary contact form on $M$ with parity $\del$. Under \ref{cond:NF}, suppose that $\alpha$ is index-admissible. Then the following assertions hold:
\begin{itemize}
\item If $\del$ is odd then every closed orbit of $\alpha$ is non-hyperbolic.
\item If $\del$ is even then $\alpha$ has at least $2(q\chi(M) - \sum_{j=-\infty}^{q} b_j) + 2b_0$ non-hyperbolic closed orbits.
\end{itemize}
\end{corollary}

Note that the first assertion is trivial since the index of a hyperbolic orbit is multiplicative under iterations which implies, in particular, that it is even for even iterates. So the real contribution of the previous corollary is in the case where $\del$ is even. To the best of our knowledge, in all currently known examples of contact forms with finitely many closed orbits every periodic orbit is elliptic.

\begin{remark}
\label{rmk:2b_0 leq b_q}
It follows from Theorem \ref{thm:main} and Corollary \ref{cor:non-hyp-lacunary} that if $M$ admits a non-degenerate lacunary contact form as in Corollary \ref{cor:non-hyp-lacunary} then $2b_0 \leq b_q$. It is obvious if $\del$ is odd but it is not if $\del$ is even. This inequality can be strict, with $\del$ even, as the example of the standard contact sphere shows.
\end{remark}

\subsection{Applications in examples}
\label{sec:applications}

In this section we present applications of Theorems \ref{thm:main} and \ref{thm:r_M} to several examples. Although these applications are immediate from these theorems, we will state them for the reader's convenience and to illustrate the wide range of contact manifolds that we can deal with. Moreover, we will focus mainly on examples that admit non-degenerate lacunary contact forms and therefore the lower bound provided by Theorem \ref{thm:main} is sharp. In these examples, we illustrate the fact, established in Theorem \ref{thm:r_M}, that the  homological bound $r_M$ coincides with the topological expected count, namely the sum of the Betti numbers of the quotient orbifold. We will consider only the multiplicity of closed orbits since the corresponding results about non-hyperbolic orbits are clear from Theorem \ref{thm:non-hyp} and left to the reader.

First, we consider (good) toric contact manifolds admitting fillings as in \ref{cond:F}. Toric contact manifolds are the odd-dimensional analogues of toric symplectic manifolds. They can be defined as contact manifolds of dimension $2n+1$ equipped with an effective Hamiltonian action of a torus of dimension $n+1$. In dimension three, the \emph{good} toric contact manifolds are $(S^3, \xi_{\rm st})$ and its finite quotients. In higher dimensions, good toric contact manifolds are precisely the compact toric contact manifolds for which the torus action is not free. These constitute the most important class of compact toric contact manifolds and admit a classification in terms of their associated moment cones, in direct analogy with Delzant's theorem, which classifies compact toric symplectic manifolds via their moment polytopes. We refer to \cite{Le,AM0} for further details. In what follows, we will omit sometimes the adjective ``good'', as it will be understood implicitly.

Toric contact manifolds $M^{2n+1}$ admit several examples with a symplectic filling $W$ as in \ref{cond:F}. These manifolds are prequantizations of (toric) symplectic orbifolds (i.e. they have toric Reeb flows that generate a locally free circle action) and admit non-degenerate lacunary contact forms with parity $n\,(\bmod\,2)$. We have that $\HC_*(M)$ is periodic, with period $\D=2$, lacunary, and satisfies all the properties of Theorem \ref{thm:main} \cite{AM0,AM3,AMM1,AMM2}. Concrete examples of such manifolds are the standard contact sphere, toric contact manifolds in dimension 3 and 5 \cite{AM3}, and the prequantization of any toric symplectic manifold $(B,\om)$ with $[\om]=c_1(TB)$, although there are plenty of other examples. Thus, we obtain the following result.

\begin{corollary}
Let $(M^{2n+1},\xi)$ be a closed good toric contact manifold admitting a filling $W$ as in \ref{cond:F}. For instance, $M$ can be the standard contact sphere, any good toric contact manifold with $c_1(\xi)=0$ in dimensions 3 and 5, and the prequantization of any toric symplectic manifold $(B,\om)$ with $[\om]=c_1(TB)$. Let $\alpha$ be a non-degenerate contact form on $M$. Assume that there exists an integer $p \neq n\,(\bmod\,2)$ such that $\alpha$ satisfies the following conditions:
\begin{itemize}
\item $\alpha$ has no good periodic orbit $\ga$ such that $\cz(\ga)=\pm p$;
\item every closed orbit of $\alpha$ has non-vanishing mean index.
\end{itemize}
Then
\[
\#\P_\alpha \geq \dim \H_*(M/S^1;\Q)
\]
where the $S^1$-action is the one induced by any periodic toric Reeb flow on $M$. Moreover, the equality holds if and only if $\alpha$ is lacunary.
\end{corollary}

\begin{remark}[Combinatorial description of $r_M$]
As noticed in Remark \ref{rmk:toric edges}, the number $r_M=\dim \H_*(M/S^1;\Q)$ in the previous corollary has a clean toric description: it is equal to the number of edges of the moment cone associated to $M$ \cite{AM0}, or, equivalently, the number of facets of the corresponding toric diagram \cite{AM3}.
\end{remark}

\begin{remark}[The hypotheses on the first homology groups]
Good toric contact manifolds $M$ satisfy $\H_1(M;\Q)=0$. Consequently, the hypotheses in \ref{cond:F} and \ref{cond:NF} involving $\H_1(W;\Q)$ and $\H_1(M;\Q)$ always hold in this case.
\end{remark}

\begin{remark}[Orbits with negative mean index]
It was recently proved in \cite{CGG2} that a non-degenerate contact form on the standard contact sphere $S^{2n+1}$ with finitely many closed orbits cannot possess periodic orbits with negative mean index. In particular, such a contact form cannot have periodic orbits with index less than or equal to $-n$. All the examples of contact forms with finitely many closed orbits that we are aware of satisfy the property that every periodic orbit has positive mean index.
\end{remark}

\begin{remark}[The period $\Delta$ in the toric case]
\label{rmk:period two}
The fact that $\D=2$ for toric contact manifolds $M$ is proved in \cite{AMM2} using a combinatorial definition of $\HC_*(M)$ that does not require conditions \ref{cond:F} or \ref{cond:NF} and that coincides with positive equivariant symplectic homology whenever one of these conditions is satisfied. However, in several examples toric contact manifolds admit (toric) symplectic fillings $W$ as in \ref{cond:F} such that $\SH_*(W)=0$; cf. \cite[Lemma 6.10]{AMM2}. In this case, the fact that $\D=2$ can be easily seen in the following way. As proved in \cite{BO17}, there is a Gysin exact triangle
\[
\xymatrix{
\SH^+_*(W)\ar[rr]& & \HC_*(M)\ar[ld]\\
& \HC_{*-2}(M)\ar[ul]^{[+1]}&
},
\]
where $\SH^+_*(W)$ stands for the positive non-equivariant symplectic homology of $W$ which is related to $\SH_*(W)$ via the tautological exact triangle
\[
\xymatrix{
\SH^-_*(W)\ar[rr]& & \SH_*(W)\ar[ld]\\
& \SH^+_{*}(M)\ar[ul]^{[-1]}&
},
\]
with $\SH^-_*(W)$ being the negative non-equivariant symplectic homology of $W$ which is isomorphic to $\H_{*+n+1}(W,M;\Q)$. Therefore, $\SH^-_k(W)=0$ for sufficiently large degrees $k$ and consequently, from the last triangle, we conclude that $\SH^+_{*}(M)=0$ for such degrees whenever $\SH_*(W)=0$. Hence, it follows from the former triangle that $\D=2$.
\end{remark}

We now turn to examples admitting fillings $W$ as in \ref{cond:F} for which symplectic homology $\SH_*(W)$ does not vanish. These examples are not toric in general. Before stating the result, we recall some definitions.

A connected Riemannian manifold $N$ is called a symmetric space if, for every $p \in N$, there exists an isometry $f_p: N \to N$ such that $f_p(p)=p$ and $f_p \circ \exp_p(v)=\exp_p(-v)$ for every $v \in T_pN$. The rank of a symmetric space $N$ is the maximal dimension of a flat, totally geodesic submanifold of $N$. By the classification of symmetric spaces, a compact rank one symmetric space (CROSS) is one of the following: $S^m$, $\RP^m$, $\CP^m$, $\HP^m$, or $\CaP^2$; see \cite{Bes} for details.

The unit cosphere bundle $S^*N$ of a CROSS $N$ admits a natural filling given by the unit codisk bundle $D^*N \subset T^*N$, which satisfies condition \ref{cond:F}. Moreover, it is well known that $\SH_*(D^*N)\neq 0$.

Every CROSS $N^{n+1}$ carries a Riemannian metric for which all geodesics are periodic with the same minimal period. Equivalently, the geodesic flow induces a free circle action on $S^*N$. To the best of our knowledge, CROSSes are the only known examples of closed manifolds admitting such a metric \cite{Bes}. It follows that the unit cosphere bundle $S^*N$ is a prequantization of a closed symplectic manifold $(B,\omega)$. Furthermore, a homological computation shows that $\H_k(B;\Q)=0$ for every odd $k$; see \cite[page 141]{Zil}. In this case, the total rank $r_M=\dim \H_*(S^*N/S^1;\Q) = \dim \H_*(B;\Q)$ is given by Table \ref{tab:cross}.

\begin{table}[tb]
\centering
\begin{tabular}{ | c | c | c | }
\hline
CROSS & $r_M = \dim \H_*(S^*N/S^1;\Q)$ \\
\hline
\rule{0pt}{0.4cm}
$S^{m}$ or $\RP^m$ \text{with} $m>2$ \text{even} & $m$ \\
$S^{m}$ or $\RP^m$ \text{with} $m$ \text{odd} & $m+1$ \\
$\CP^{m}$ & $m(m+1)$ \\
$\HP^m$ & $2m(m+1)$ \\
$\CaP^2$ & $24$ \\
\hline
\end{tabular}
\vskip .2cm
\caption{The total rank of $\dim \H_*(S^*N/S^1;\Q)$ for a CROSS $N$.}
\label{tab:cross}
\end{table}

From \eqref{eq:ESHpreq} we can easily see that $\HC_*(S^*N)$ satisfies all the assumptions of Theorem \ref{thm:main} and the parity of $\HC_*(S^*N)$ is $n\,(\bmod\,2)$. (Note that we take $N$ with dimension $n+1$ in order to have $S^*N$ with dimension $2n+1$.) Moreover, it is well known that $S^*N$ admits lacunary non-degenerate contact forms \cite{AM5,Zil}. Therefore, we obtain the following result which improves previous results from \cite{AM5,GGM2,DLLW2}.

\begin{corollary}
Let $N^{n+1}$ be a CROSS. Let $\alpha$ be a non-degenerate contact form on $S^*N$. Assume that there exists an integer $p \neq n\,(\bmod\,2)$ such that $\alpha$ satisfies the following conditions:
\begin{itemize}
\item $\alpha$ has no good periodic orbit $\ga$ such that $\cz(\ga)=\pm p$;
\item every closed orbit of $\alpha$ has non-vanishing mean index.
\end{itemize}
Then 
\[
\#\P_\alpha \geq \dim \H_*(S^*N/S^1;\Q)
\]
where $\dim \H_*(S^*N/S^1;\Q)$ is given by Table \ref{tab:cross}.  Moreover, the equality holds if and only if $\alpha$ is lacunary.
\end{corollary}

Another example arising from unit cosphere bundles is given when the Riemannian manifold $N$ is a lens space. More precisely, let $q \geq 1$ be an integer and consider the $\Z_q$-action on the odd-dimensional sphere $S^{2m+1} \subset \C^{m+1}$ generated by the map
\[
\psi(z_0, \dots, z_m) = \left(e^{\frac{2\pi i \ell_0}{q}}z_0, e^{\frac{2\pi i \ell_1}{q}}z_1, \dots, e^{\frac{2\pi i \ell_m}{q}}z_m \right),
\]
where $\ell_0, \ldots, \ell_m$ are integers, called the weights of the action. This action is free whenever the weights are coprime with $q$ (which we assume from now on), and in this case the quotient of $S^{2m+1}$ by $\Z_q$ is a lens space, denoted by $\L$.

If $S^{2m+1}$ is endowed with the round metric, it induces a Riemannian metric on $N=\L$. The corresponding geodesic flow is again periodic; however, in contrast to the case of a CROSS, it typically induces only a \emph{locally} free circle action on $S^*N$. The manifold $S^*N$ admits a filling satisfying \ref{cond:F}, namely the unit codisk bundle. In \cite{AM5}, we compute $\HC_*(S^*N)$ (more precisely, the positive equivariant symplectic homology for contractible orbits $\HC^0_*(S^*N)$; the computation readily extends to the general case). Alternatively, one may view $S^*N$ as a prequantization of an orbifold $B$, given by a finite quotient of the Grassmannian of oriented two-planes $G^+_2(\R^{2m+2})$, and apply the computation \eqref{eq:HCorb} from the proof of case \ref{PO} in Theorem \ref{thm:r_M}. In either approach, one finds that $S^*N$ satisfies the assumptions of Theorem \ref{thm:main}, with $\HC_*(S^*N)$ lacunary of even parity and
\[
r_M = \dim \H_*(S^*N/S^1;\Q) = 2m+2.
\]
Moreover, $M$ admits a non-degenerate lacunary contact form, as shown in \cite{AM5} or in Section \ref{sec:proof r_M}. (The equality $r_M=2m+2$ also follows from the fact that this contact form has precisely $2m+2$ closed orbits.) We thus obtain the following corollary, which improves earlier results from \cite{AM5}.

\begin{corollary}
Let $N^{2m+1}$ be a lens space $\L$, and let $\alpha$ be a non-degenerate contact form on $S^*N$. Assume that there exists an odd integer $p$ such that $\alpha$ satisfies the following conditions:
\begin{itemize}
\item $\alpha$ has no good periodic orbit $\ga$ with $\cz(\ga)=\pm p$;
\item every closed orbit of $\alpha$ has non-vanishing mean index.
\end{itemize}
Then
\[
\#\P_\alpha \geq \dim \H_*(S^*N/S^1;\Q) = 2m+2.
\]
Moreover, the equality holds if and only if $\alpha$ is lacunary.
\end{corollary}

Another example is provided by the Ustilovsky spheres $M$ discussed in Theorem \ref{thm:r_M}, which satisfy \ref{cond:F}; see \cite{KvK,Ust}. They admit non-degenerate lacunary contact forms with even parity. Their homology $\HC_*(M)$ is periodic with period $\D \neq 2$ in general \cite{Ust}. Therefore, in this case, $M$ does not admit a filling $W$ such that $\SH_*(W)=0$; see Remark \ref{rmk:period two}. (Note that $\HC_*(M)$ is a contact invariant since $M$ admits non-degenerate lacunary contact forms.) Hence, by Theorems \ref{thm:main} and \ref{thm:r_M}, we obtain the following corollary.

\begin{corollary}
Let $M^{2n+1}$ be an Ustilovsky sphere, that is, a Brieskorn manifold $\Sigma(a_0,a_1,\dots,\break a_{n+1})$ with $n$ even, $a_0 = \pm 1\,(\bmod\,8)$ and $a_1=a_2=\cdots=a_{n+1}=2$. Let $\alpha$ be a non-degenerate contact form on $M$. Assume that there exists an odd integer $p$ such that $\alpha$ satisfies the following conditions:
\begin{itemize}
\item $\alpha$ has no good periodic orbit $\ga$ such that $\cz(\ga)=\pm p$;
\item every closed orbit of $\alpha$ has non-vanishing mean index.
\end{itemize}
Then 
\[
\#\P_\alpha \geq n+1.
\]
Moreover, the equality holds if and only if $\alpha$ is lacunary.
\end{corollary}

Finally, we present a last example which, although not explicit, encompasses all the previous ones and, to the best of our knowledge, all currently known examples of contact manifolds admitting contact forms with finitely many closed orbits for which equivariant symplectic homology is a well-defined invariant with integer grading.

More precisely, we consider a prequantization $S^1$-bundle $M^{2n+1}$ over a closed symplectic \emph{orbifold} admitting a Hamiltonian circle action with isolated fixed points. As proved in Section \ref{sec:proof r_M}, $\HC_*(M)$ is periodic and lacunary, with parity $n \,(\bmod\,2)$. Moreover, $M$ admits a non-degenerate lacunary contact form whose number of closed orbits coincides with the number of fixed points of the Hamiltonian circle action, which, in turn, is equal to $\dim \H_*(M/S^1;\Q)$. Therefore, we obtain the following immediate consequence of Theorem \ref{thm:main} together with case \ref{PO} in Theorem \ref{thm:r_M}.

\begin{corollary}
Let $M^{2n+1}$ be a prequantization $S^1$-bundle of a closed symplectic orbifold $B$ admitting a Hamiltonian circle action with isolated fixed points. Assume that $M$ satisfies \ref{cond:F} or \ref{cond:NF}, that the contact form $\beta$ that generates the $S^1$-action has index-positivity and, under \ref{cond:NF}, that every contractible orbit of $\beta$ has Robbin-Salamon index $\geq 4$.

Let $\alpha$ be a non-degenerate contact form on $M$. Assume that there exists an integer $p \neq n\,(\bmod\,2)$ such that $\alpha$ satisfies the following conditions:
\begin{itemize}
\item $\alpha$ has no good periodic orbit $\ga$ such that $\cz(\ga)=\pm p$;
\item every closed orbit of $\alpha$ has non-vanishing mean index;
\item under \ref{cond:NF}, $\alpha$ is index-admissible. 
\end{itemize}
Then
\[
\#\P_\alpha \geq \dim \H_*(B;\Q),
\]
where $\dim \H_*(B;\Q)$ coincides with the number of fixed points of the Hamiltonian circle action on $B$. Moreover, the equality holds if and only if $\alpha$ is lacunary.
\end{corollary}

When $(B,\om)$ is a smooth manifold such that $[\om]=c_1(TB)$ and $\H_{\mathrm{odd}}(B;\Q)=0$, a prequantization $M$ of $(B,\om)$ satisfies \ref{cond:F} and all the conditions of Theorem \ref{thm:main} with $\del = n \,(\bmod\,2)$. (The filling is given by the disk bundle of the associated complex line bundle over $B$. As explained in Remark \ref{rmk:H^1 preqorb} below, we have that $\H_1(M;\Q)=0$.) Therefore, we have the following version of the previous corollary in the smooth case.

\begin{corollary}
Let $M^{2n+1}$ be a prequantization $S^1$-bundle of a closed symplectic manifold $(B,\om)$ such that $[\om]=c_1(TB)$ and $\H_{\mathrm{odd}}(B;\Q)=0$.

Let $\alpha$ be a non-degenerate contact form on $M$. Assume that there exists an integer $p \neq n\,(\bmod\,2)$ such that $\alpha$ satisfies the following conditions:
\begin{itemize}
\item $\alpha$ has no good periodic orbit $\ga$ such that $\cz(\ga)=\pm p$;
\item every closed orbit of $\alpha$ has non-vanishing mean index.
\end{itemize}
Then
\[
\#\P_\alpha \geq \dim \H_*(B;\Q).
\]
Moreover, the equality holds if and only if $\alpha$ is lacunary.
\end{corollary}

\begin{remark}[The hypotheses on the first homology groups in \ref{cond:F} and \ref{cond:NF}]
\label{rmk:H^1 preqorb}
To the best of our knowledge, all currently known examples of closed contact manifolds $M$ admitting contact forms with finitely many closed orbits satisfy $\H_1(M;\Q)=0$. In fact, as mentioned above, all examples we are aware of arise as prequantizations of orbifolds $B$ admitting a Hamiltonian circle action with isolated fixed points.  Such an action implies that $\H_{\mathrm{odd}}(B;\Q)=0$ (see Section \ref{sec:proof r_M}), and in particular $\H^1(B;\Q)=0$.

Applying the Gysin sequence to the natural $S^1$-bundle $M \times ES^1 \to M_{S^1} := (M \times ES^1)/S^1$, and using that $\H^*_{S^1}(M;\Q) = \H^*(M_{S^1};\Q) \cong \H^*(B;\Q)$ and $\H^*(M \times ES^1;\Q) \cong \H^*(M;\Q)$ we obtain
\[
\cdots \to 0=\H^1(B;\Q) \to \H^1(M;\Q) \to \H^0(B;\Q) \xrightarrow{\wedge\om} \H^2(B;\Q) \to \cdots\, .
\]
Since the map $\H^0(B;\Q) \xrightarrow{\wedge\om} \H^2(B;\Q)$ is injective, it follows that $\H^1(M;\Q)=0$.

Therefore, in particular, the hypotheses in \ref{cond:F} and \ref{cond:NF} involving $\H_1(W;\Q)$ and $\H_1(M;\Q)$ always hold in this case.
\end{remark}

\subsection{Organization of the paper}

The remainder of this paper is organized as follows. Section \ref{sec:prelim} provides the necessary background. More specifically, Section \ref{sec:ESH} contains a brief review of positive equivariant symplectic homology, while the corresponding Morse inequalities are discussed in Section \ref{sec:morseineq}. The resonance relations for equivariant symplectic homology are presented in Section \ref{sec:resonance}. In Section \ref{sec:IRT}, we state the index recurrence theorem, which serves as a key tool in the proofs of Theorems \ref{thm:main} and \ref{thm:non-hyp}.

In Section \ref{sec:proofmain}, we formulate three general results -- Theorems \ref{thm:main1}, \ref{thm:main2}, and \ref{thm:non-hyp-general} -- which, in particular, imply Theorems \ref{thm:main} and \ref{thm:non-hyp}. Theorems \ref{thm:main1} and \ref{thm:non-hyp-general} are proved in Section \ref{sec:proofmain1}, with Theorem \ref{thm:non-hyp-general} arising as a byproduct of the proof of Theorem \ref{thm:main1}, while Theorem \ref{thm:main2} is established in Section \ref{sec:proofmain2}. Section \ref{sec:proof r_M} is devoted to the proof of Theorem \ref{thm:r_M}. Finally, Section \ref{sec:questions} presents several questions concerning contact manifolds and contact forms with finitely many closed orbits.

\subsection{Acknowledgments}

The second author is grateful to Fr\'ed\'eric Bourgeois, Viktor Ginzburg, Alex Oancea and Baptiste Serraille for useful discussions and comments on a preliminary version of this work.

\section{Preliminaries}
\label{sec:prelim}

\subsection{Equivariant symplectic homology}
\label{sec:ESH}
In this section we briefly recall several facts about positive equivariant symplectic homology, presenting the subject from a slightly unconventional perspective, following \cite{AM5,GGM2}.

First, let $(M,\xi)$ be a closed contact manifold and $(W,\Omega)$ be a strong symplectic filling of $M$ with $c_1(TW)=0$. Usually, we also ask that $\Om$ is atoroidal, that is, $\int f^*\Omega = 0$ for any smooth map $f: T^2 \to W$, but this condition can be dropped using the universal Novikov field 
\[
\Nov = \bigg\{\sum_{i=1}^\infty n_iT^{a_i}\ ;\, a_i \in \R,\, a_i\to\infty,\, n_i\in \Q\bigg\},
\]
together with an action filtration introduced by McLean and Ritter \cite{MR}; cf. \cite[Section 2]{AGKM}.

Since $c_1(TW)=0$, the determinant line bundle $\Lambda_\C^{n+1}TW$ is trivial, and we can therefore choose a non-vanishing section $\sec$, which determines a trivialization of $\Lambda_\C^{n+1}TW$. Given a periodic orbit $\ga$ of a Hamiltonian $H \colon W \to \R$, this section can be used to symplectically trivialize $\ga^*TW$ and hence define its Conley-Zehnder index $\cz(\ga)$; see \cite[Section 3]{AM3}, \cite{AM5}, \cite{Es}, \cite{McL2}, and \cite[Appendix C]{MR}. The index is well-defined up to the homotopy class of $\sec$. If $\H^1(W;\Q)=0$, then it is in fact independent of the choice of $\sec$, since any two such sections are homotopic; see \cite[Lemma 4.3]{McL2}. Moreover, for contractible orbits, this trivialization agrees (up to homotopy) with the one induced by a capping disk. In particular, in this case the index does not depend on the choice of $\sec$, and the assumption $\H^1(W;\Q)=0$ can be dropped. More generally, the index does not depend on the choice of $\sec$ for orbits that are homologically trivial (taking rational coefficients) in $W$; see \cite[Lemma 4.3]{McL2} and \cite[Remark C.2]{MR}.

A crucial point throughout this work is that this trivialization is \emph{closed under iterations}: the trivialization induced on $\ga^k$ coincides, up to homotopy, with the $k$-th iterate of the trivialization along $\ga$. In general, this property fails if one instead fixes a trivialization over a reference loop in each free homotopy class.

Let $\alpha$ be a non-degenerate contact form on $M$ supporting the contact structure $\xi$. Since $c_1(\xi)=0$, as ensured by either condition \ref{cond:F} or \ref{cond:NF}, one can define the index of every closed orbit of $\alpha$ by choosing a non-vanishing section $\sec$ of $\Lambda_\C^{n}\xi$, as above. This index is well-defined up to the homotopy class of $\sec$, and it is independent of this choice whenever $\H^1(M;\Q)=0$, since any two such sections are then homotopic. As before, this assumption can be dropped when restricting to orbits that are homologically trivial (over $\Q$) in $M$.

Recall that a periodic orbit $\ga$ of $\alpha$ is \emph{good} if its index has the same parity of the index of the underlying simple closed orbit. Then the positive equivariant symplectic homology $\SH_*^{S^1,+}(W)$ with coefficients in $\Nov$ is the homology of a complex $\CC_*(\alpha)$ generated by the good closed Reeb orbits of $\alpha$; see \cite[Proposition 3.3]{GG}. (More precisely, \cite[Proposition 3.3]{GG} is proved for contractible orbits assuming that $\Om|_{\pi_2(W)}=0$ and using $\Q$-coefficients but its proof can be readily adapted to our context since it is purely algebraic; cf. \cite[Section 2]{AGKM}.) This complex is graded by the Conley-Zehnder index and filtered by the action introduced in \cite{MR}. It turns out that, although the filling in \ref{cond:F} is not exact in general, we can assume, without loss of generality, that the filtration in $\CC_*(\alpha)$ is given by the period of the Reeb orbits. Furthermore, once we fix a free homotopy class of loops in $W$, the part of $\CC_*(\alpha)$ generated by closed Reeb orbits in that class is a subcomplex. As a consequence, the entire complex $\CC_*(\alpha)$ decomposes into a direct sum of such subcomplexes indexed by free homotopy classes of loops in $W$.

The differential of the complex $\CC_*(\alpha)$, but not its homology, depends on several auxiliary choices, and the nature of the differential is not essential for our purposes. The complex $\CC_*(\alpha)$ is functorial in $\alpha$ in the sense that a symplectic cobordism equipped with a suitable extra structure gives rise to a map of complexes. For the sake of brevity and to emphasize the obvious analogy with contact homology, we denote the homology of $\CC_*(\alpha)$ by $\HC_*(M)$ rather than $\SH_*^{S^1,+}(W)$. Given a subset of free homotopy classes $\Gamma$ in $W$ and $T \in (0,\infty]$ we denote by $\HC_*^{\Gamma,T}(\alpha)$ the homology of the corresponding subcomplex $\CC_*^{\Gamma,T}(\alpha)$ with free homotopy classes in $\Gamma$ and action less than $T$. If $\Gamma$ is the set of all free homotopy classes we omit it in the notation. However, it is worth keeping in mind that $\CC_*^{\Gamma,T}(\alpha)$ and possibly even the homology may depend on the choice of the filling $W$. Note that if $T<\infty$ then $\HC_*^{\Gamma,T}(\alpha)$ depends on $\alpha$ but $\HC_*^{\Gamma,\infty}(\alpha)=\HC_*^{\Gamma}(M)$ does not. Moreover, if $M$ admits a non-degenerate lacunary contact form then $\HC_*(M)$ does not depend on the choice of filling (i.e. it is a contact invariant) since the differential vanishes due to index reasons.

This description of the positive equivariant symplectic homology as the homology of $\CC_*(\alpha)$ is not quite standard, but it is most suitable for our purposes. (We refer the reader to \cite{GG} for more details and further references and to \cite{BO09,BO10, BO13a, BO13b, BO17,Vit} for the original construction of the equivariant symplectic homology.) To see why $\HC_*(M):=\SH_*^{S^1,+}(W)$ can be obtained as the homology of a single complex generated by good closed Reeb orbits, let us first consider an admissible Hamiltonian $H$ on the symplectic completion of $W$ and focus on the orbits of $H$ with positive action. Such orbits are in a one-to-one correspondence with closed Reeb orbits $\gamma$ with action below a certain threshold $T$ depending on the slope of $H$. The $S^1$-equivariant Floer homology of $H$ is the homology of a Floer-type complex obtained from a non-degenerate parametrized perturbation of $H$; \cite{BO13b,Vit}. This complex is filtered by the action. (Here we are using the action filtration introduced by McLean-Ritter \cite{MR}.) The $E^1$-term of the resulting spectral sequence (over $\Nov$) is generated by the good Reeb orbits of $\alpha$ with action below $T$. Now we can (canonically, once the generators are fixed) reassemble the differentials $\p_r$ into a single differential $\p$ on $\CC_*(H):=E^1_{*,*}$ in such a way that the homology of the resulting complex is $E^\infty=\HF_*^{S^1,+}(H)$. Roughly speaking, $\p=\p_1+\p_2+\ldots$, where $\p_r$ is suitably ``extended'' from $E^r$ to $E^1$. Moreover, this procedure respects the action filtration and is functorial with respect to continuation maps. Passing to the limit in $H$, we obtain the complex $\CC_*(\alpha)$ as the limit of the complexes $\CC_*(H)$; see \cite[Sections 2.5 and 3]{GG} for further details.

A remarkable observation by Bourgeois and Oancea in \cite[Section 4.1.2]{BO17} is that under suitable additional assumptions on the indices of closed Reeb orbits the positive equivariant symplectic homology is defined even when $M$ does not have a symplectic filling. To be more precise, we assume that $c_1(\xi)=0$ and consider non-degenerate contact forms $\alpha$ on $M$ such that all of its closed contractible Reeb orbits have Conley-Zehnder index strictly greater than $3-n$. Furthermore, under this assumption the proof of \cite[Proposition 3.3]{GG} carries over essentially word-for-word, and hence again the positive equivariant symplectic homology of $M$ can be described as the homology of a complex $\CC_*(\alpha)$ generated by good closed Reeb orbits of $\alpha$, graded by the Conley-Zehnder index and filtered by the action. The complex breaks down into the direct sum of subcomplexes indexed by free homotopy classes of loops \emph{in} $M$. As in the fillable case, we will use the notation $\HC_*(M)$ and $\HC_*^{\Gamma,T}(\alpha)$. Note that in general, in spite of the notation, this homology has slightly different properties (and hypothetically could be different) from the homology defined via a filling. For instance, it has a decomposition by the free homotopy classes of loops in $M$, but not in $W$ as when $M$ is fillable. (Intuitively, one can think of the resulting homology as defined by using a non-compact filling of $M$ by the bounded part $M\times (0,\,1]$ of the symplectization of $(M,\xi)$.)

\subsection{Morse inequalities}
\label{sec:morseineq}

Given a non-degenerate contact form $\alpha$ and $m \in \Z$, define the Morse-type number $c_m$ as $\dim \CC_m(\alpha)$. Thus, $c_m$ equals the number of good closed orbits of $\alpha$ with index $m$, and, in general, one may have $c_m = \infty$. However, $c_m$ is finite whenever $\alpha$ has finitely many simple closed orbits, all of which have non-zero mean index. The following strong Morse inequalities will play an important role later on.

\begin{proposition}
\label{prop:morse_ineq}
Assume that $\alpha$ has finitely many simple closed orbits all of them with non-zero mean index. Let $k_1$ and $k_2$ be integers with the same parity such that $k_1 \leq k_2$. If $k_1$ and $k_2$ are even then
\[
 \sum_{m=k_1}^{k_2} (-1)^m b_m \leq  \sum_{m=k_1}^{k_2} (-1)^m c_m
 \]
 and if $k_1$ and $k_2$ are odd we have the reversed inequality
\[
 \sum_{m=k_1}^{k_2} (-1)^m b_m \geq  \sum_{m=k_1}^{k_2} (-1)^m c_m.
 \]
 \end{proposition}

\begin{proof}
Let $P(t) = \sum_{m=-\infty}^\infty b_m t^m$ and $M(t) = \sum_{m=-\infty}^\infty c_m t^m$ be the Poincar\'e and Morse series respectively. Consider the differential $\partial$ in $\CC_*(M)$ and let $q_m=\dim \Im\ \partial_{m+1}$. Define the series $Q(t) = \sum_{m=-\infty}^\infty q_m t^m$ and note that $q_m \geq 0$ for every $m$. A standard algebraic argument shows that
\[
M(t) - P(t) = (1+t)Q(t).
\]
Therefore, taking $t=-1$,
\begin{align*}
\sum_{m=k_1}^{k_2} (-1)^m c_m - \sum_{m=k_1}^{k_2} (-1)^m b_m & = \sum_{m=k_1}^{k_2} (-1)^m (q_m+q_{m-1}) \\
& = \sum_{m=k_1}^{k_2} (-1)^m q_m -  \sum_{m=k_1-1}^{k_2-1} (-1)^m q_m \\
& = (-1)^{k_2}q_{k_2} - (-1)^{k_1-1}q_{k_1-1} \\
& = (-1)^{k_2}(q_{k_2} + q_{k_1-1}),
\end{align*}
where in the last equality we use the fact that $k_1$ and $k_2$ have the same parity. Now, since $q_m \geq 0$ for every $m$, the last term is $\geq 0$ if $k_2$ is even and $\leq 0$ if $k_2$ is odd.
\end{proof}

\subsection{Resonance relations}
\label{sec:resonance}

Let $\ga$ be an isolated (possibly degenerate) closed Reeb orbit and denote by $\HC_*(\ga)$ its local equivariant symplectic homology; see \cite{GG,HM}.  For a non-degenerate orbit $\ga$, we have that
\begin{equation*}
\HC_*(\gamma) =
\begin{cases}
\Nov\ \ \text{if }*=\cz(\ga)\text{ and }\ga\text{ is good} \\
0\ \ \text{otherwise}.
\end{cases}
\end{equation*}
The Euler characteristic of $\ga$ is defined as
\[
\chi(\gamma) = \sum_{m \in \Z} (-1)^m\dim \HC_m(\gamma).
\]
This sum is finite. When $\gamma$ is non-degenerate,
\begin{equation*}
\chi(\gamma) =
\begin{cases}
(-1)^{\mu(\gamma)}\ \ \text{if }\ga\text{ is good} \\
0\ \ \text{otherwise}.
\end{cases}
\end{equation*}
The local {\it mean} Euler characteristic of $\gamma$ is
\[
\hat\chi(\gamma) = \lim_{j\to\infty} \frac{1}{j} \sum_{k=1}^j \chi(\gamma^k).
\]
The limit above exists and is rational; see \cite{GGo}. When $\ga$ is strongly non-degenerate, i.e. all iterates of $\ga$ are non-degenerate, we have
\begin{equation*}
\hat\chi(\gamma) =
\begin{cases}
(-1)^{\cz(\ga)}\ \text{ if }\ga^2\text{ is good} \\
(-1)^{\cz(\ga)}/2\ \text{ if }\ga^2\text{ is bad}.
\end{cases}
\end{equation*}

Recall from Section \ref{sec:results} that the mean index of a closed orbit $\ga$ is
\[
\mi(\ga) = \lim_{k\to\infty} \frac 1k \cz(\ga^k).
\]
Assume that $\alpha$ is has finitely many distinct simple closed orbits $\ga_1,\dots,\ga_r$ with non-vanishing mean index. This assumption ensures that the positive/negative mean Euler characteristic
\begin{equation*}
\label{eq:def_MEC}
\chi_\pm(M) = \lim_{j\to\infty} \frac{1}{j} \sum_{m=n+1}^j (-1)^m b_{\pm m}
\end{equation*}
is well-defined; see \cite{GGo}. 

The mean Euler characteristic is related to local equivariant symplectic homology via the resonance relation
\begin{equation}
\label{eq:resonance}
\sum_{\{1\leq i\leq r;\ \pm\mi(\ga_i)>0\}} \pm\frac{\hat\chi(\gamma_i)}{\mi (\gamma_i)} = \chi_\pm(M),
\end{equation}
proved in \cite{GK} in the non-degenerate case and in \cite{HM} in general.

\begin{remark}
Although it is not really important in this work, there is a missing minus sign in the resonance relation in \cite{GK,HM} for $\chi_-(M)$.
\end{remark}

\subsection{A useful property of periodic positive equivariant symplectic homology}

Assume that $\HC_*(M)$ is periodic as in Definition \ref{def:periodic}, $\dim \HC_j(M) < \infty$ for every $j$ and $\HC_j(M)=0$ for all $j$ sufficiently negative. Along this work the following elementary proposition will be very useful.

\begin{proposition}
\label{prop:periodicity}
Given an integer $p\geq 0$, let $f: \Z \to \Z$ be the function
\[
f(j) = j\chi_+(M) - \sum_{m=-\infty}^{j+p} (-1)^mb_m.
\]
Then $f(j) = f(j+\D)$ for every $j\geq P$.
\end{proposition}

\begin{proof}
It follows from the periodicity of the positive equivariant symplectic homology that
\[
\chi_+(M)=\r/\D,
\]
where
\[
\r = \sum_{m=k}^{k+\D-1} (-1)^mb_m
\]
for any $k\geq P$. Hence,
\begin{align*}
f(j+\D) & = j\chi_+(M) + \r - \sum_{m=-\infty}^{j+p} (-1)^mb_m - \sum_{m=j+p+1}^{j+p+\D} (-1)^mb_m \\
& = j\chi_+(M) - \sum_{m=-\infty}^{j+p} (-1)^mb_m \\
& = f(j),
\end{align*}
for any $j\geq P$.
\end{proof}

\section{Index recurrence}
\label{sec:IRT}

A crucial ingredient in the proof of Theorem \ref{thm:main} is the following combinatorial result addressing the index behavior under iterations which is an easy generalization of \cite[Theorem 4.1]{GGM2}. This result can also be deduced from the so-called generalized common index jump theorem due to Duan, Liu, Long and Wang \cite{DLLW2}; see also \cite{Lon02,LZ}.

\begin{theorem}
\label{thm:IRT}
Let $\Phi_1,\ldots,\Phi_r$ be a finite collection of strongly non-degenerate elements of $\TSp(2n)$ with $\mi(\Phi_i)\neq 0$ for all $1 \leq i \leq r$. Denote by $\s_i=\mi(\Phi_i)/|\mi(\Phi_i)|$ the sign of $\mi(\Phi_i)$. Then for any $\eta>0$ and any $\ell_0\in\N$, there exist two integer sequences $d_j^\pm\to\infty$ and two sequences of integer vectors $\vk^\pm_j=\big(k_{1j}^\pm,\ldots,k_{rj}^\pm\big)$ with all components going to infinity as $j\to\infty$, such that for all $i$ and $j$, and all $\ell\in\Z$ in the range $1\leq |\ell|\leq \ell_0$, we have
\begin{enumerate}
\item[\reflb{cond:i}{\rm{(i)}}]
  $\big|\mi\big(\Phi^{k_{ij}^\pm}_i\big)-\s_id_j^\pm\big|<\eta$ with the
  equality
  $\mi\big(\Phi^{k_{ij}^\pm}_i\big) =
  \cz\big(\Phi^{k_{ij}^\pm}_i\big)=\s_id_j^\pm$ whenever $\Phi_i(1)$ is
  hyperbolic,
\item[\reflb{cond:ii}{\rm{(ii)}}]
  $\cz\big(\Phi^{k_{ij}^\pm+\ell}_i\big)= \s_id_j^\pm + \cz(\Phi^\ell_i)$,
  and
\item[\reflb{cond:iii}{\rm{(iii)}}]
  $\cz\big(\Phi^{k_{ij}^-}_i\big)-\s_id_j^-=
  -\big(\cz\big(\Phi^{k_{ij}^+}_i\big)-\s_id_j^+\big)$.
\end{enumerate}
Furthermore, for any $N\in \N$ we can make all $d_j^\pm$ and $k_{ij}^\pm$ divisible by~$N$.
\end{theorem}

\begin{proof}
This theorem was proved in \cite[Theorem 4.1]{GGM2} under the assumption that $\mi(\Phi_i)>0$ for every $i$. We claim that the desired result can be reduced to this case. As a matter of fact, assume that at least one $\Phi_i$ has negative mean index and order $\Phi_1,\ldots,\Phi_r$ such that $\Phi_1,\ldots,\Phi_{r'}$ have negative mean index for some $r'>0$ and $\Phi_{r'+1},\ldots,\Phi_{r}$ have positive mean index (the set $\{\Phi_{r'+1},\ldots,\Phi_{r}\}$ can be empty). Take the paths $\Psi_1,\ldots,\Psi_r$ given by
\[
\Psi_i(t)=\Phi_i(t)^{-1} \text{ for every }1\leq i \leq r'
\]
and
\[
\Psi_i(t)=\Phi_i(t) \text{ for every } r'+1\leq i \leq r
\]
so that $\mi(\Psi_i)>0$ for every $i$.

Applying \cite[Theorem 4.1]{GGM2} to $\Psi_1,\ldots,\Psi_r$, given $\eta>0$ and $\ell_0\in\N$, we get two integer sequences $d_j^\pm\to\infty$ and two sequences of integer vectors $\vk^\pm_j=\big(k_{1j}^\pm,\ldots,k_{rj}^\pm\big)$ with all components going to infinity as $j\to\infty$, such that for all $i$ and $j$, and all $\ell\in\Z$ in the range $1\leq |\ell|\leq \ell_0$, we have
\begin{enumerate}
\item[\reflb{cond:i'}{\rm{(i')}}]
  $\big|\mi\big(\Psi^{k_{ij}^\pm}_i\big)-d_j^\pm\big|<\eta$ with the
  equality
  $\mi\big(\Psi^{k_{ij}^\pm}_i\big) =
  \cz\big(\Psi^{k_{ij}^\pm}_i\big)=d_j^\pm$ whenever $\Psi_i(1)$ is
  hyperbolic,
\item[\reflb{cond:ii'}{\rm{(ii')}}]
  $\cz\big(\Psi^{k_{ij}^\pm+\ell}_i\big)= d_j^\pm + \cz(\Psi^\ell_i)$,
  and
\item[\reflb{cond:iii'}{\rm{(iii')}}]
  $\cz\big(\Psi^{k_{ij}^-}_i\big)-d_j^-=
  -\big(\cz\big(\Psi^{k_{ij}^+}_i\big)-d_j^+\big)$.
\end{enumerate}
Now, \ref{cond:i}, \ref{cond:ii} and \ref{cond:iii} follow from \ref{cond:i'}, \ref{cond:ii'} and \ref{cond:iii'} using the facts that, for all $1\leq i\leq r'$ and $m \in \N$, $\mi(\Psi_i^m)=-\mi(\Phi_i^m)$, $\cz(\Psi_i^m)=-\cz(\Phi_i^m)$ (note that $\Phi_i^m$ is non-degenerate) and $\Psi_i(1)$ is hyperbolic if and only if so is $\Phi_i(1)$.
\end{proof}

\section{Proof of Theorems \ref{thm:main} and \ref{thm:non-hyp}}
\label{sec:proofmain}

Theorem \ref{thm:main} is an immediate consequence of the following two more general results, proved in the next two sections.

\begin{theorem}
\label{thm:main1}
Let $(M^{2n+1},\xi)$ be a closed contact manifold satisfying either \ref{cond:F} or \ref{cond:NF}. Assume that $\HC_*(M)$ is periodic, $b_j = 0$ for all $j$ sufficiently negative, and $b_j < \infty$ for all $j \in \Z$.

Let $\alpha$ be a non-degenerate contact form on $M$. Assume that there exists an integer $p\geq 0$ such that $\alpha$ satisfies the following conditions:
\begin{itemize}
\item $\alpha$ has no good periodic orbit $\ga$ such that $\cz(\ga)=\pm p$;
\item every closed orbit of $\alpha$ has non-vanishing mean index;
\item under \ref{cond:NF}, $\alpha$ is index-admissible. 
\end{itemize}
If $p>1$ assume that $M$ satisfies \ref{cond:PD}. Then, if $p$ is even (resp. $p$ is odd)
\[
\#\P_\alpha \geq r^o_M\ (\text{resp. }\#\P_\alpha \geq r^e_M),
\]
where
\begin{align*}
& r_M^o = 
\begin{cases}
2(-q\chi_+(M) + \sum_{j=-\infty}^q (-1)^jb_j)\ \text{if}\ p=0 \\
2(-q\chi_+(M) + \sum_{j=-\infty}^q (-1)^jb_j) + 2b_{q+p} - b_q\ \text{if}\ p\ \text{is even and}\ \geq 2, \\
\end{cases} \\
& r_M^e = 2(q\chi_+(M) - \sum_{j=-\infty}^{q} (-1)^jb_j) + 2b_{q+p} + b_q,
\end{align*}
with $q=\min\{s\D;\,s\D \geq P,\,s \in \N\}$. Moreover, if the equality holds then $\alpha$ is lacunary with odd (resp. even) parity.
\end{theorem}

In Section \ref{sec:hypotheses} we present examples showing that the hypothesis that $\alpha$ has no good periodic orbit $\ga$ such that $\cz(\ga)=\pm p$ is necessary to get the lower bounds stated above, although the corresponding assumption in Theorem \ref{thm:main} is probably just technical. The point of these examples is that they are lacunary contact forms with the ``wrong'' parity, that is, their parities are even for the bound $r^o_M$ and odd for the bound $r^e_M$; see the discussion below.

Recall that, given a lacunary non-degenerate contact form $\alpha$, its parity is defined as the parity of the indices of its closed orbits. (Here we tacitly assume that $\alpha$ has at least one closed orbit; this condition is automatically satisfied whenever $\HC_j(M)\neq 0$ for some $j \in \Z$.) The second main result shows that the inequality in Theorem \ref{thm:main1} becomes an equality for lacunary contact forms with the ``right'' parity, thereby explaining the notation in the superscript of $r^{o,e}_M$.

\begin{theorem}
\label{thm:main2}
Let $(M^{2n+1},\xi)$ be a closed contact manifold satisfying either \ref{cond:F} or \ref{cond:NF}. Assume that $\HC_*(M)$ is periodic, $b_j \neq 0$ for some $j$, $b_j = 0$ for all $j$ sufficiently negative, and $b_j < \infty$ for all $j \in \Z$.

Let $\alpha$ be a non-degenerate lacunary contact form on $M$ with parity $\del$. Under \ref{cond:NF}, suppose furthermore that $\alpha$ is index-admissible. If $\del$ is odd (resp. even) then 
\[
\#\P_\alpha = r^o_M\ (\text{resp. }\#\P_\alpha = r^e_M),
\]
with $r^{o}_M$ and $r^{e}_M$ as in Theorem \ref{thm:main1}, for any $p\neq \del\,(\bmod\,2)$, which, in this case, are equal to
\[
r_M^o = 2(-q\chi_+(M) - \sum_{j=-\infty}^q b_j)
\]
and
\[
r_M^e = 2(q\chi_+(M) - \sum_{j=-\infty}^{q} b_j) + b_q.
\]
In particular, $\#\P_\alpha$ does not depend on the choice of $\alpha$.
\end{theorem}

\begin{remark}
\label{rmk:unnecessary orbits 4}
Note that the hypothesis that there is no good closed orbit with index $\pm p$ for some $p$ even (resp. $p$ odd) in Theorem \ref{thm:main1} is automatically satisfied by a lacunary contact form with odd (resp. even) parity. Moreover, if $M$ admits a non-degenerate lacunary contact form $\alpha$ then $\HC_*(M)$ is lacunary with the same parity as $\alpha$; cf. Section \ref{sec:ESH}. Therefore, if $M$ admits a non-degenerate lacunary contact form with parity $\del$ then the hypothesis that, for a given non-degenerate contact form, there is no closed orbit with index $\pm p$ for some $p \neq \del\,(\bmod\,2)$ excludes periodic orbits which are \emph{unnecessary} from the homological point of view; cf. Remark \ref{rmk:unnecessary orbits 1}. Moreover, it follows from Theorem \ref{thm:main2} that the lower bound given by Theorem \ref{thm:main1} is sharp under  the hypothesis that there is no closed orbit with index $\pm p$ whenever $M$ admits a non-degenerate lacunary contact form with parity $\del \neq p\,(\bmod\,2)$. If we take the wrong parity the lower bound in Theorem \ref{thm:main1} can be negative: for instance, when $M=S^{2n+1}$ is the standard contact sphere with $n$ odd (resp. $n$ even) then it admits a non-degenerate lacunary contact form with odd (resp. even) parity and $r_M^e<0$ (resp. $r_M^o<0$) when $p$ is even (resp. $p$ is odd).
\end{remark}

\begin{remark}
\label{rmk:unnecessary orbits 5}
As already mentioned in Remark \ref{rmk:unnecessary orbits 2}, the assumption that $b_j<\infty$ for every $j \in \Z$ implies that closed orbits with zero mean index are \emph{unnecessary} from the homological point of view. In particular, if $\alpha$ is non-degenerate and lacunary then this assumption forces that every closed orbit of $\alpha$ has non-zero mean index.
\end{remark}

Theorem \ref{thm:main2} follows from the proof of Theorem \ref{thm:main1} and Lemma \ref{lemma:combinatorial} from \cite{Gu}, as explained in Section \ref{sec:proofmain2}; see Remarks \ref{rmk:lacunary1}, \ref{rmk:lacunary2} and \ref{rmk:lacunary3}. For the reader's convenience, we will provide an alternative easier proof in Section \ref{sec:proofmain2} following the lines of \cite{AM5}.

The proof of Theorem \ref{thm:main1} also yields the following multiplicity result concerning non-hyperbolic closed orbits when the contact form has finitely many geometrically distinct closed orbits. It clearly implies Theorem \ref{thm:non-hyp}. In what follows, we denote by $\#\P^{\text{non-hyp}}_\alpha$ the number of simple non-hyperbolic closed orbits of a contact form $\alpha$.

\begin{theorem}
\label{thm:non-hyp-general}
Let $(M^{2n+1},\xi)$ be a closed contact manifold satisfying either \ref{cond:F} or \ref{cond:NF}. Assume that $\HC_*(M)$ is periodic, $b_j = 0$ for all $j$ sufficiently negative, and $b_j < \infty$ for all $j \in \Z$.

Let $\alpha$ be a non-degenerate contact form on $M$. Assume that there exists an integer $p$, which is even or equal to $1$, such that $\alpha$ satisfies the following conditions:
\begin{itemize}
\item $\alpha$ has finitely many simple closed orbits;
\item $\alpha$ has no good periodic orbit $\ga$ such that $\cz(\ga)=\pm p$;
\item every closed orbit of $\alpha$ has non-vanishing mean index;
\item under \ref{cond:NF}, $\alpha$ is index-admissible. 
\end{itemize}
If $p>1$ assume that $M$ satisfies \ref{cond:PD}. Then, if $p$ is even (resp. $p=1$)
\[
\#\P^{\text{non-hyp}}_\alpha \geq r^{o,\text{non-hyp}}_M\ (\text{resp. }\#\P^{\text{non-hyp}}_\alpha \geq r^{e,\text{non-hyp}}_M),
\]
where
\begin{align*}
& r_M^{o,\text{non-hyp}} = 
\begin{cases}
2(-q\chi_+(M) + \sum_{j=-\infty}^q (-1)^jb_j) \ \text{if}\ p=0 \\
2(-q\chi_+(M) + \sum_{j=-\infty}^q (-1)^jb_j) + 2b_{q+p} - b_q \ \text{if}\ p\ \text{is even and}\ \geq 2, \\
\end{cases} \\
& r_M^{e,\text{non-hyp}} = 2(q\chi_+(M) - \sum_{j=-\infty}^{q} (-1)^jb_j) + b_{q-1} + 3b_{q+1} + 2b_0,
\end{align*}
and $q=\min\{s\D;\,s\D > P,\,s \in \N\}$.
\end{theorem}

This result follows from the proof of Theorem \ref{thm:main1}; see Remarks \ref{rmk:non-hyperbolic p even} and \ref{rmk:non-hyperbolic p odd}.

\begin{remark}
\label{rmk:strict ineq}
When $p$ is even, we can take $q=\min\{s\D;\,s\D \geq P,\,s \in \N\}$. The reason why we take the strict inequality when $p$ is odd is to ensure that $2(q\chi_+(M) - \sum_{j=-\infty}^{q} (-1)^jb_j) + b_{q-1} + 3b_{q+1} + 2b_0$ does not depend on the choice of $q$ (because $b_{q-1}=b_{q-1+\Delta}$). Note that, if we assume without loss of generality that $\D$ is even, as we do in this work, $q$ is even as well, and therefore the strict inequality $s\D > P$ is not necessary if $b_{\text{odd}}=0$. In particular, we can take $q=\min\{s\D;\,s\D \geq P,\,s \in \N\}$ in Theorem \ref{thm:non-hyp}.
\end{remark}

\begin{remark}
Note that, in principle, we may have $r_M^{e,\text{non-hyp}} > r_M^{e}$ although we are not aware of any example of such an $M$ carrying a contact form with finitely many closed orbits. Indeed, $r_M^{e,\text{non-hyp}} \leq r_M^{e}$ whenever $M$ admits an even lacunary contact form; see Remark \ref{rmk:2b_0 leq b_q}.
\end{remark}

\subsection{Necessity of the main hypothesis in Theorem \ref{thm:main1}}
\label{sec:hypotheses}

We show here that the hypothesis that there is no good periodic orbit with index $\pm p$ for some $p\geq 0$ is necessary for the results established in Theorem \ref{thm:main1}.

Let us consider first the case where $p$ is even. We constructed in \cite{AM0} a family of toric contact structures $\xi_k$, $k\in \N_0$, on $M=S^2 \times S^3$ such that, for every $k$,
\begin{itemize}
\item $(M,\xi_k)$ satisfies condition \ref{cond:F};
\item $\HC_*(M,\xi_k)$ satisfies all the requested properties of Theorem \ref{thm:main1};
\item $\xi_k$ admits a (toric) non-degenerate even lacunary contact form $\alpha_k$ with precisely 4 simple periodic orbits, all of which with positive mean index;
\item we have
\begin{equation}
\label{eq:CH-ex1}
\dim \HC_*(M,\xi_k) =
\begin{cases}
k & \text{if $* = 0$;} \\
2k+1  & \text{if $* = 2$;} \\
2k+2 & \text{if $* > 2$ and even;} \\
0 & \text{otherwise.}
\end{cases} 
\end{equation}
\end{itemize}
Let $r^{o}_{\xi_k}$ be $r^{o}_M$ corresponding to $(M,\xi_k)$ and a fixed even $p\geq 2$. We have that $P=4$ and $\D=2$ so that we can take $q=4$. Moreover, it is clear from \eqref{eq:CH-ex1} that $\chi_+(M,\xi_k)=k+1$. Thus,
\begin{align*}
4 & = 2(q\chi_+(M) - \sum_{j=-\infty}^{q} (-1)^jb_j) + b_q \\
& = 2(q\chi_+(M) - \sum_{j=-\infty}^{q} (-1)^jb_j) + 2k+2
\end{align*}
which implies that
\begin{align*}
& 2(q\chi_+(M) - \sum_{j=-\infty}^{q} (-1)^jb_j) = 0\quad\text{if}\ k=1, \\
& 2(q\chi_+(M) - \sum_{j=-\infty}^{q} (-1)^jb_j) < 0\quad\text{if}\ k>1.
\end{align*}

Assume now that the hypothesis in Theorem \ref{thm:main1} that there is no good periodic orbit $\ga$ with $\cz(\ga)=\pm p$ for some even $p \geq 2$ can be dropped, so that we could apply Theorem \ref{thm:main1} to $\alpha_k$. When $k=1$ we have that
\[
r^o_{\xi_k} = 2k + 2
\]
and therefore the equality
\[
\#\P_{\alpha_k} = r^o_{\xi_k} = 4.
\]
But $\alpha_k$ is not odd lacunary (it is even lacunary), which contradicts the assertion in Theorem \ref{thm:main1} that if $\#\P_{\alpha_k} = r^o_{\xi_k}$ then $\alpha_k$ must be odd lacunary.

Furthermore, when $k>1$ the lower bound $r^o_{\xi_k}$ is given by
\[
r^o_{\xi_k}= -2(q\chi_+(M) - \sum_{j=-\infty}^{q} (-1)^jb_j) + 2k+2 > 4
\]
contradicting the fact that $\alpha_k$ has precisely 4 simple periodic orbits.

Notice that, in both situations, the point here is that if $k\geq 1$ then $\HC_*(M,\xi_k) \neq 0$ for every even $*\geq 0$ and therefore we must have periodic orbits with index $p$ for every even $p\geq 0$.

\begin{remark}
If we take the bound $r^{o}_{\xi_k}$ with $p=0$, arguing as above we also obtain a contradiction for $k\geq 3$.
\end{remark}

Concerning the hypothesis that there is no good periodic orbit with index $\pm p$ for some $p$ odd, we proceed in a similar fashion. One can show that there is a family of contact structures $\xi_k$, $k\in \N_0$, on $M_k:=L^3_{k+1}(1,k)$ such that, for every $k$,
\begin{itemize}
\item $(M_k,\xi_k)$ satisfies condition \ref{cond:F};
\item $\HC_*(M_k,\xi_k)$ satisfies all the required properties of Theorem \ref{thm:main1};
\item $\xi_k$ admits a (toric) non-degenerate odd lacunary contact form $\alpha_k$ with precisely 2 simple periodic orbits, all of which with positive mean index;
\item we have
\begin{equation}
\label{eq:CH-ex2}
\dim \HC_*(M_k,\xi_k) =
\begin{cases}
k & \text{if $* = 1$;} \\
1 + k  & \text{if $* \geq 3$ and odd;} \\
0 & \text{otherwise.}
\end{cases} 
\end{equation}
\end{itemize}
As above, let $r^{e}_{\xi_k}$ be $r^{e}_{M_k}$ corresponding to $(M_k,\xi_k)$ for any odd $p$. We have that $P=3$ and $\D=2$ so that we can take $q=4$. Moreover, it is clear from \eqref{eq:CH-ex2} that $\chi_+(M,\xi_k)=-(k+1)/2$. Thus,
\[
2 = 2(-q\chi_+(M) + \sum_{j=-\infty}^{q} (-1)^jb_j).
\]

Assume now that the hypothesis in Theorem \ref{thm:main1} that there is no good periodic $\ga$ with $\cz(\ga)=\pm p$ for some odd $p>0$ can be dropped, so that we could apply Theorem \ref{thm:main1} to $\alpha_k$. The lower bound $r^e_{\xi_k}$ is given by
\begin{align*}
r^e_{\xi_k} & = -2(-q\chi_+(M) + \sum_{j=-\infty}^{q} (-1)^jb_j) + b_{q} + 2b_{q+p} \\
& = -2(-q\chi_+(M) + \sum_{j=-\infty}^{q} (-1)^jb_j) + 2(k + 1) \\
& = -2 + 2k + 2.
\end{align*}

Thus, when $k=1$ we have the equality
\[
\#\P_{\alpha_k} = r^e_{\xi_k} = 2
\]
but $\alpha_k$ is not even lacunary (it is odd lacunary), contradicting the assertion in Theorem \ref{thm:main1} that if $\#\P_{\alpha_k} = r^e_{\xi_k}$ then $\alpha_k$ must be even lacunary.

Moreover, when $k>1$ the lower bound $r^e_{\xi_k}$ satisfies
\[
r^e_{\xi_k} = 2k > 2
\]
contradicting the fact that $\alpha_k$ has precisely 2 simple periodic orbits.

Notice that the point here, as in the previous case where $p$ is even, is that if $k\geq 1$ then $\HC_*(M,\xi_k) \neq 0$ for every odd $*\geq 1$ and therefore we must have periodic orbits with index $p$ for every odd $p\geq 1$.

\section{Proof of Theorem \ref{thm:main1}}
\label{sec:proofmain1}

The main tool used in the proof is the positive equivariant symplectic homology. Recall from Section \ref{sec:ESH} that both conditions \ref{cond:F} and \ref{cond:NF} ensure that this homology (for $M$ or the filling) with integer grading is well-defined. Then the proof is the same in both cases of the theorem, \ref{cond:F} and \ref{cond:NF}, and relies only on the condition shared by these cases that $\alpha$ has no good periodic orbits $\ga$ such that $\cz(\ga)=\pm p$ for some integer $p$ and that the mean index of every closed orbit is not zero. Recall that under \ref{cond:NF} we assume that $\alpha$ is index-admissible.

\subsection{Outline of the proof}
\label{sec:proofmain1-outline}
First of all, let us give a brief idea of the proof of Theorem \ref{thm:main1}. Assume that $\alpha$ has finitely many distinct simple closed orbits $\ga_1,\dots,\ga_r$. Using the resonance relation \eqref{eq:resonance} we obtain in Lemma \ref{lemma:resonance} an expression for the truncated mean Euler characteristic of $M$. From this expression we obtain, after a careful analysis of the contributions of each orbit to the alternating sum $\sum_{m=-d-p}^{d+p} (-1)^m c_m$, the relation
\begin{equation}
\label{eq:truncate1}
\sum_{m=-d-p}^{d+p} (-1)^m c_m 
= d\chi_+(M) + \sum_{j=-p}^p (-1)^jc_{j}
- \sum_{\substack{i=1\\ |\cz(\ga_i^{k_i})|>d+p}}^r \chi(\ga_i^{k_i}),
\end{equation}
where the even numbers $d,k_1,\dots,k_r$ are given by Theorem \ref{thm:IRT} (with $\eta$, $\ell_0$ and $N$ suitably chosen) and $c_m$ is the $m$-th Morse type number.

In essence, Lemma \ref{lemma:resonance} asserts that the mean Euler characteristic can be exactly evaluated as a finite sum once the summation for every orbit $\ga_i$ is truncated at a suitably chosen order of iteration $k_i$.  Truncating the sum between degrees $-d-p$ and $d+p$ for all orbits, as in the above formula, introduces an error given by the second and third terms on the right. Note that the third error term involves only the orbits $\ga_i$ with $|\cz(\ga_i^{k_i})|>d+p$. We denote the number of such orbits by $r_+$. Using \eqref{eq:truncate1} and the periodicity of $\HC_*(M)$ we can show that
\begin{align*}
r_+ \geq & (-1)^{p+1}q\chi_+(M) + (-1)^p\sum_{m=-\infty}^{q+p} (-1)^m b_m + (-1)^{p+1} \sum_{j=-p}^p (-1)^jc_{j} \\
& + \#\{1\leq i\leq r;\,\cz(\ga_i^{k_i})=p\,(\bmod\,2),\ |\cz(\ga_i^{k_i})|>d+p\ \text{and}\ \ga_i^{k_i}\ \text{is good}\}.
\end{align*}

Repeating the above argument and applying again Theorem \ref{thm:IRT}, we can find even numbers $d',k'_1,\dots,k'_r$ such that
\begin{equation}
\label{eq:truncate2}
\sum_{m=-d'-p}^{d'+p} (-1)^m c_m 
= d'\chi_+(M) + \sum_{j=-p}^p (-1)^jc_{j}
- \sum_{\substack{i=1\\ |\cz(\ga_i^{k'_i})|>d'+p}}^r \chi(\ga_i^{k'_i}) 
\end{equation}
Here the last error term involves only the orbits with $|\cz(\ga_i^{k'_i})|>d'+p$. Denoting by $r_-$ the number of such orbits, we obtain the same lower bound for $r_-$:
\begin{align*}
r_- \geq & (-1)^{p+1}q\chi_+(M) + (-1)^p\sum_{m=-\infty}^{q+p} (-1)^m b_m + (-1)^{p+1} \sum_{j=-p}^p (-1)^jc_{j} \\
& + \#\{1\leq i\leq r;\,\cz(\ga_i^{k'_i})=p\,(\bmod\,2),\ |\cz(\ga_i^{k'_i})|>d+p\ \text{and}\ \ga_i^{k'_i}\ \text{is good}\}.
\end{align*}
By item \ref{cond:iii} of Theorem \ref{thm:IRT},
\[
|\cz(\ga^{k'_i})|>d'+p\ \text{if and only if}\ |\cz(\ga_i^{k_i})|<d-p.
\]

Thus $r_+$ and $r_-$ count disjoint sets of simple closed orbits and therefore $r \geq r_-+r_+$. Now, one can show that there are more
\[
r' := \sum_{j=-p}^p (c_{d+j}+c_{-d+j}-c_j-c_{-j})
\]
new simple closed orbits. This term counts simple orbits $\ga_i$ such that $|\cz(\ga_i^{k_i})| \in [d-p,d+p]$.

Set
\[
r^e = \#\{1\leq i\leq r;\,\cz(\ga_i^{k_i})\ \text{is even}\ \text{and}\ \ga_i^{k_i}\ \text{is good}\}
\]
and
\[
r^o = \#\{1\leq i\leq r;\,\cz(\ga_i^{k_i})\ \text{is odd}\ \text{and}\ \ga_i^{k_i}\ \text{is good}\}
\]
Using the second error term on the right of \eqref{eq:truncate1} and \eqref{eq:truncate2}, the Morse inequalities, the Poincar\'e duality property \ref{cond:PD}, and our assumption that there is no good periodic orbit with index $\pm p$, a (non-trivial) computation shows that, if $p=0$,
\begin{align*}
r & \geq r_-+r_++r' \\
& \geq 2(-q\chi_+(M) + \sum_{m=-\infty}^{q} (-1)^m b_m) + r^e,
\end{align*}
if $p \geq 2$ is even,
\begin{align*}
r & \geq r_-+r_++r' \\
& \geq 2(-q\chi_+(M) + \sum_{m=-\infty}^{q} (-1)^m b_m) + 2b_{q+p} - b_q + r^e,
\end{align*}
and if $p$ is odd,
\begin{align*}
r & \geq r_-+r_++r' \\
&\geq 2(q\chi_+(M) - \sum_{m=-\infty}^{q} (-1)^m b_m) + 2b_{q+p} + b_q + r^o.
\end{align*}
Using these inequalities, one can show that if the lower bound of the theorem is an equality then $\alpha$ is lacunary.

\subsection{Proof of the theorem}
We begin the proof by assuming that $\alpha$ has finitely many distinct simple closed orbits $\ga_1,\dots,\ga_r$. Our goal is to establish the lower bound on $r$ asserted by the theorem and to show that, if this bound is attained, then $\alpha$ is lacunary. (In fact, we prove the converse as well, namely that this bound is attained if and only if $\alpha$ is lacunary, under the assumption that $\alpha$ has finitely many closed orbits; see Remarks \ref{rmk:lacunary1}, \ref{rmk:lacunary2}, and \ref{rmk:lacunary3}. Under our hypotheses, the fact that a non-degenerate lacunary contact form has finitely many periodic orbits follows from Lemma \ref{lemma:combinatorial} from \cite{Gu}; see Section \ref{sec:proofmain2}, where we also provide a simpler independent proof of this converse (Theorem \ref{thm:main2}).)

By assumption, the mean index of each $\ga_i$ is non-zero. Set
\begin{align*}
\lo = \max\bigg\{p,\,\max_{1 \leq i \leq r} \min\big\{k_0 \in \N \ \big|\ & \cz(\ga_i^{k+\l}) \geq |\cz(\ga_i^k)| + 2n + p + 1\ \forall k \geq 1,\ \forall \l \geq k_0 \ \text{if } \mi(\ga_i) > 0, \\
& \cz(\ga_i^{k+\l}) \leq -|\cz(\ga_i^k)| - 2n - p - 1\ \forall k \geq 1,\ \forall \l \geq k_0 \ \text{if } \mi(\ga_i) < 0 \big\}\bigg\}.
\end{align*}
Note that $\lo$ is well-defined, i.e., the minima above are finite, by the definition of the mean index.

By Theorem \ref{thm:IRT}, given $N \in \N$, $\eta>0$, and $\lo$ as above, there exist two sequences of integer vectors $(d^\pm_j,k^\pm_{1j},\dots,k^\pm_{rj})$ satisfying conditions \ref{cond:i}, \ref{cond:ii}, and \ref{cond:iii}, such that all entries $d^\pm_j,k^\pm_{1j},\dots,k^\pm_{rj}$ are divisible by $N$. We will only need one vector from each sequence, and we set
\begin{equation}
\label{eq:integers}
(d,k_1,\dots,k_r):=(d^+_1,k^+_{11},\dots,k^+_{r1}), \qquad (d',k'_1,\dots,k'_r):=(d^-_1,k^-_{11},\dots,k^-_{r1}).
\end{equation}

The following lemma, which provides an expression for the truncated mean Euler characteristic (see Section \ref{sec:resonance}), is one of the key steps in the proof; cf.\ \cite[Sublemma 5.2]{AM2}, \cite[Lemma 5.1]{GGM2}, and \cite[Lemma 4.5]{AM5}. Recall that, by the periodicity of the equivariant symplectic homology,
\[
\chi_+(M) = \frac{\r}{\D},
\]
where $\D$ is the period of $\HC_*(M)$ and $\r=\sum_{m=P}^{P+\D-1} (-1)^m b_m$.

\begin{lemma}
\label{lemma:resonance}
The numbers $N$ and $\eta$ can be chosen such that $d=s\D$ for some integer $s$ and
\[
\sum_{i=1}^r \sum_{k=1}^{k_{i}} \chi(\ga_i^k) = \sum_{i=1}^r k_{i} \hat{\chi}(\gamma_i) = d \chi_+(M) = s\r.
\]
The same holds for $d',k'_1,\dots,k'_r$.
\end{lemma}

\begin{proof}
Let $N$ be any (positive) integer multiple of $\D$. Since $\D$ is even so is $N$. The first equality follows from the (strong) non-degeneracy of $\ga_1,\dots,\ga_r$ since the numbers $k_i$ are even. The third equality follows just from the facts that $d=s\D$ and $\chi_+(M)=\r/\D$. To prove the second equality, take $\eta$ sufficiently small such that $\eta \sum_{i=1}^r \big| \frac{\hat\chi (\gamma_i)}{\mi(\ga_i)}\big|<1$. Using the resonance relation \eqref{eq:resonance} we conclude that
\begin{align*}
d \chi_\pm(M)
& = \sum_{\{1\leq i\leq r;\ \pm\mi(\ga_i)>0\}} \pm\frac{d\hat\chi (\gamma_i)}{\mi (\gamma_i)} \\
& = \sum_{\{1\leq i\leq r;\ \pm\mi(\ga_i)>0\}}  k_i\hat\chi (\gamma_i) + \sum_{\{1\leq i\leq r;\ \pm\mi(\ga_i)>0\}}  \frac{(\pm d-k_i\mi(\ga_i))\hat\chi(\gamma_i)}{\mi(\ga_i)} \\
& = \sum_{\{1\leq i\leq r;\ \pm\mi(\ga_i)>0\}}  k_i\hat\chi (\gamma_i) +  \sum_{\{1\leq i\leq r;\ \pm\mi(\ga_i)>0\}}  \frac{(\pm d-\mi(\ga_i^{k_i}))\hat\chi (\gamma_i)}{\mi(\ga_i)}. 
\end{align*}
By property \ref{cond:i} of Theorem \ref{thm:IRT} and the condition on $\eta$,
\[
\bigg| \sum_{\{1\leq i\leq r;\ \pm\mi(\ga_i)>0\}} \frac{(\pm d-\mi(\ga_i^{k_i}))\hat\chi(\gamma_i)}{\mi(\ga_i)}\bigg| = \bigg| \sum_{i=1}^r \frac{(\s_i d-\mi(\ga_i^{k_i}))\hat\chi(\gamma_i)}{\mi(\ga_i)}\bigg| < \eta \sum_{i=1}^r\bigg| \frac{\hat\chi (\gamma_i)}{\mi(\ga_i)}\bigg| <1.
\]
Note that by our choice of $N$ the numbers $d\chi_\pm(M)$ and $k_i\hat\chi (\gamma_i)$ for all $i$ are integers. Therefore,
\[
d \chi_\pm(M) =  \sum_{\{1\leq i\leq r;\ \pm\mi(\ga_i)>0\}}  k_i\hat\chi (\gamma_i).
\]
Now, observe that, by our hypotheses, $\sum_{\{1\leq i\leq r;\,\mi(\ga_i)<0\}} k_i\hat\chi (\gamma_i)=d\chi_-(M)=0$. Consequently, since every periodic orbit has non-zero mean index,
\[
\sum_{i=1}^r  k_i\hat\chi (\gamma_i)  = \sum_{\{1\leq i\leq r;\,\mi(\ga_i)<0\text{ or }\mi(\ga_i)>0\}} k_i\hat\chi (\gamma_i)  =  \sum_{\{1\leq i\leq r;\,\mi(\ga_i)>0\}} k_i\hat\chi (\gamma_i) = d\chi_+(M).
\]
Obviously, the same argument works for $d',k'_1,\dots,k'_r$.
\end{proof}

The next lemma will be useful when $p>1$ in Theorem \ref{thm:main1}. It establishes that suitable prequantizations $M$ of symplectic manifolds $B$ have a symmetry in $\HC_*(M)$ for big enough degrees, which follows easily from the Poincar\'e duality on $B$; cf. Propositions \ref{prop:PD finite} and \ref{prop:PD preq_orb} and Remark \ref{rmk:PD}.

\begin{proposition}
\label{prop:PD preq}
Let $(M^{2n+1},\xi)$ be a closed contact manifold satisfying either \ref{cond:F} or \ref{cond:NF}. Assume that $M$ is a prequantization of a positive monotone symplectic manifold $(B,\om)$ such that $\H_{\text{odd}}(B;\Q)=0$. Under condition \ref{cond:NF} assume furthermore that the minimal Chern number $c_B$ of $B$ satisfies $c_B>1$. Given $p\geq 0$ we can choose $d$ large enough such that
\[
b_{d+j}=b_{d-j}
\]
for every $j \in [-p,p]$.
\end{proposition}

\begin{proof}
From the assumptions on $M$ we have that $\HC_*(M)$ is well-defined and given by
\begin{equation}
\label{eq:ESHpreq0}
\HC_\ast(M) \cong \bigoplus_{m\in\N} \H_{\ast -m\bD+n} (B; \Nov). 
\end{equation}
where $\bD$ is the Robbin-Salamon index of the simple orbit of the $S^1$-action on $M$ which is an even positive integer since $B$ is positive monotone; see Section \ref{sec:proof r_M} for details.

Choose $N$ such that
\begin{equation}
\label{eq;condition d preq}
d=2s\bD\quad\text{and}\quad d>n+p.
\end{equation}
Given $j \in \Z$ we have from \eqref{eq:ESHpreq0} that
\[
b_{d+j} = \dim \oplus_{m=1}^\infty \H_{n+2(s-m)\bD+j}(B;\Q).
\]
Since $2s\bD=d>n+p$ we clearly have that
\[
b_{d+j} = \dim \oplus_{m=-\infty}^\infty \H_{n+2(s-m)\bD+j}(B;\Q)
\]
for every $j \in [-p,p]$. Thus, by Poincar\'e duality,
\begin{align*}
b_{d+j} & = \dim \oplus_{m=-\infty}^\infty \H_{n-2(s-m)\bD-j}(B;\Q) \\
& = \dim \oplus_{m=-\infty}^\infty \H_{n+2(s-m)\bD-j}(B;\Q) \\
& = b_{d-j}
\end{align*}
for every $j \in [-p,p]$.
\end{proof}

Now, for the sake of simplicity, we will split the proof of Theorem \ref{thm:main1} in three main steps.

\vskip .3cm
\noindent
{\bf Step 1. $\alpha$ has at least }$(-1)^{p+1}q\chi_+(M) + (-1)^p\sum_{m=-\infty}^{q+p} b_m + (-1)^{p+1} \sum_{j=-p}^p (-1)^jc_{j} + \#\{1\leq i\leq r;\,\cz(\ga_i^{k_i})=p\,(\bmod\,2),\ |\cz(\ga_i^{k_i})|>d+p\ \text{and}\ \ga_i^{k_i}\ \text{is good}\}$ {\bf simple closed orbits}
\vskip .2cm

\vskip .2cm

Fix $N$ and $\eta$ as in Lemma \ref{lemma:resonance}. (In particular, $N$ is even.)  Taking $N$ big enough, we can assume that $d$ and $k_i$ are big enough such that
\begin{equation}
\label{eq:d1}
d>\max\{p+P,p+n,q\},
\end{equation}
\begin{equation}
\label{eq:d2}
-d + p < \bmin := \min\{m \in \Z;\, b_m \neq 0\},
\end{equation}
and
\begin{equation}
\label{eq:k_i}
\lo + 2 \leq \min_{1\leq i \leq r} k_i.
\end{equation}
Condition \eqref{eq:d2} can be achieved because, by assumption, $\bmin > -\infty$. We also take $N$ big enough such that $d$ satisfies Propositions \ref{prop:PD preq} and \ref{prop:PD finite} whenever $M$ meets \ref{cond:PD}.

Clearly, $\eta$ can be chosen so small that the vector $(d,k_{1},\dots,k_{r})=(d^+_1,k^+_{11},\dots,k^+_{r1})$ given by Theorem \ref{thm:IRT} satisfies
\begin{equation}
\label{eq:IRT1}
\cz(\ga_i^{k_i-\l}) = \s_id - \cz(\ga_i^\l),
\end{equation}
\begin{equation}
\label{eq:IRT2}
\cz(\ga_i^{k_i+\l}) = \s_id + \cz(\ga_i^\l),
\end{equation}
and
\begin{equation}
\label{eq:IRT3}
|\cz(\ga_i^{k_i}) - \s_id| \leq n
\end{equation}
for every $1\leq i \leq r$ and $1 \leq \l \leq \lo$.  (Here, since $N$ is even, the integers $d$ and $k_i$ are even.) For each periodic orbit $\ga_i$, we will consider five types of iterates outside $\ga_i^{k_i}$:
\begin{itemize}
\item[(A)] $\ga_i^{k_i-\l}$ with $\l>\lo$;
\item[(B)] $\ga_i^{k_i-\l}$ with $1\leq \l \leq \lo$ and $\cz(\ga_i^\l) \notin [-p,p]$;
\item[(C)] $\ga_i^{k_i+\l}$ with $1\leq \l \leq \lo$ and $\cz(\ga_i^\l) \notin [-p,p]$;
\item[(D)] $\ga_i^{k_i+\l}$ with $\l>\lo$;
\item[(E)] $\ga_i^{k_i\pm\l}$ with $1\leq \l \leq \lo$ and $\cz(\ga_i^\l) \in [-p,p]$.
\end{itemize}
Let us analyze the contributions of these iterates to the Morse type numbers appearing in the alternating sum
\begin{equation}
\label{eq:c_m}
\sum_{m=-d-p}^{d+p} (-1)^m c_m,
\end{equation}
where, as in Section \ref{sec:morseineq}, $c_m$ is the number of good closed orbits of index $m$.  First, class (A) iterates have index with absolute value $<d+p$ and hence all good orbits here contribute to \eqref{eq:c_m}. Indeed, by definition of $\lo$, we have $|\cz(\ga_i^{k_i})| \geq |\cz(\ga_i^{k_i-\l})| + 2n+p+1$ for every $\l \in (\lo,k_i)$ which, combined with \eqref{eq:IRT3}, implies that $|\cz(\ga_i^{k_i-\l})| \leq d-n-p-1$ for all $\l \in (\lo,k_i)$.  Class (D) orbits do not contribute to \eqref{eq:c_m} since $|\cz(\ga_i^{k_i+\l})| \geq |\cz(\ga_i^{k_i})| + 2n+p+1 \geq d + n+p+1$ for every $\l>\lo$, where the last inequality again follows from \eqref{eq:IRT3}. In turn, class (E) does contribute to \eqref{eq:c_m} by equations \eqref{eq:IRT1} and \eqref{eq:IRT2}.

In order to understand the contributions from classes (B) and (C), let us further divide each of them into four subclasses:

\begin{itemize}
\item[(B1+)] $\ga_i^{k_i-\l}$ with $1\leq \l \leq \lo$ if $\cz(\ga_i^\l) \geq -p$, $\mi(\ga_i)>0$ and $\cz(\ga_i^\l) \notin [-p,p]$,
\item[(B2+)] $\ga_i^{k_i-\l}$ with $1\leq \l \leq \lo$ if $\cz(\ga_i^\l) < -p$ and $\mi(\ga_i)>0$,
\item[(B1--)] $\ga_i^{k_i-\l}$ with $1\leq \l \leq \lo$ if $\cz(\ga_i^\l) > -p$, $\mi(\ga_i)<0$ and $\cz(\ga_i^\l) \notin [-p,p]$,
\item[(B2--)] $\ga_i^{k_i-\l}$ with $1\leq \l \leq \lo$ if $\cz(\ga_i^\l) \leq -p$, $\mi(\ga_i)<0$ and $\cz(\ga_i^\l) \notin [-p,p]$,
\end{itemize}
and
\begin{itemize}
\item[(C1+)] $\ga_i^{k_i+\l}$ with $1\leq \l \leq \lo$ if $\cz(\ga_i^\l) \geq -p$, $\mi(\ga_i)>0$ and $\cz(\ga_i^\l) \notin [-p,p]$,
\item[(C2+)] $\ga_i^{k_i+\l}$ with $1\leq \l \leq \lo$ if $\cz(\ga_i^\l) < -p$ and $\mi(\ga_i)>0$.
\item[(C1--)] $\ga_i^{k_i+\l}$ with $1\leq \l \leq \lo$ if $\cz(\ga_i^\l) > -p$, $\mi(\ga_i)<0$ and $\cz(\ga_i^\l) \notin [-p,p]$,
\item[(C2--)] $\ga_i^{k_i+\l}$ with $1\leq \l \leq \lo$ if $\cz(\ga_i^\l) \leq -p$, $\mi(\ga_i)<0$ and $\cz(\ga_i^\l) \notin [-p,p]$.
\end{itemize}

Now all of the good orbits in class (B1+) contribute to \eqref{eq:c_m}, while class (B2+) makes no contribution to \eqref{eq:c_m}. Indeed, by \eqref{eq:IRT1}, $\cz(\ga_i^{k_i-\l}) = d - \cz(\ga_i^\l)$ which is $\leq d+p$ whenever $\cz(\ga_i^\l) \geq -p$, $\cz(\ga_i^{k_i-\l}) \geq -d-p$ by \eqref{eq:d1} and the fact that $\mi(\ga_i)>0 \implies \cz(\ga_i^j)>-n$ for all $j$, and $\cz(\ga_i^{k_i-\l})>d+p$ whenever $\cz(\ga_i^\l) < -p$. In a similar vein, (B2--) contributes to \eqref{eq:c_m} and (B1--) does not: by \eqref{eq:IRT1}, $\cz(\ga_i^{k_i-\l}) = -d - \cz(\ga_i^\l)$ which is $\geq -d-p$ whenever $\cz(\ga_i^\l) \leq -p$, $\cz(\ga_i^{k_i-\l}) \leq d$ by \eqref{eq:d1} and the fact that $\mi(\ga_i)<0 \implies \cz(\ga_i^j)<n$ for all $j$, and $\cz(\ga_i^{k_i-\l})<-d-p$ whenever $\cz(\ga_i^\l) > -p \implies \cz(\ga_i^\l) > p$, since we are excluding good orbits $\ga_i^{k_i\pm\l}$ such that $\cz(\ga_i^\l) \in [-p,p]$.

Now, since we are excluding good orbits $\ga_i^{k_i\pm\l}$ such that $\cz(\ga_i^\l) \notin [-p,p]$, we have that for all good iterates $\ga_i^{k_i+\l}$ in (C1+) $\cz(\ga_i^\l) > p$, and for all good iterates $\ga_i^{k_i+\l}$ in (C2--) $\cz(\ga_i^\l) < -p$. Hence, if $\ga_i^{k_i+\l}$ belongs to (C1+), by \eqref{eq:IRT2}, $\cz(\ga_i^{k_i+\l}) = d + \cz(\ga_i^\l) > d+p$ whenever $\ga_i^{k_i+\l}$ is good. Thus, the orbits from (C1+) do not contribute to \eqref{eq:c_m}. The orbits of (C2--) do not contribute as well: by \eqref{eq:IRT2}, $\cz(\ga_i^{k_i+\l}) = -d + \cz(\ga_i^\l) < -d-p$ whenever $\ga_i^\l$ is good. On the other hand, the good orbits of (C2+) and (C1--) contribute: in the first case, $\cz(\ga_i^{k_i+\l}) = d + \cz(\ga_i^\l) < d+p$ because $\cz(\ga_i^\l) < -p$ and $\cz(\ga_i^{k_i+\l}) \geq -d-p$ by \eqref{eq:d1} and the fact that $\mi(\ga_i)>0 \implies \cz(\ga_i^j)>-n$ for all $j$; in the second case, $\cz(\ga_i^{k_i+\l}) = -d + \cz(\ga_i^\l) > -d-p$ because $\cz(\ga_i^\l) > -p$ and $\cz(\ga_i^{k_i+\l}) \leq d+p$ again by \eqref{eq:d1} and the fact that $\mi(\ga_i)<0 \implies \cz(\ga_i^j)<n$ for all $j$. (Above $\cz(\ga_i^{k_i+\l})$ and $\cz(\ga_i^\l)$ have the same parity since the integers $k_i$ are even.)

To summarize, all good orbits from classes (A), (E), (B1+), (B2--), (C2+) and (C1--) contribute to \eqref{eq:c_m}, and orbits from classes (D), (B1--), (B2+), (C1+) and (C2--)  make no contribution to \eqref{eq:c_m}.  Define
\begin{equation}
\label{eq:c^e}
c^e_\pm = \sum_{i=1}^r \#\{1 \leq \l \leq \lo \mid \s_i\cz(\ga_i^\l) < 0,\ \ga_i^{k_i\pm \l}\text{ is good, }\cz(\ga_i^\l) \notin [-p,p]\text{ and }\cz(\ga_i^\l)\text{ is even}\},
\end{equation}
\begin{equation}
\label{eq:c^o}
c^o_\pm = \sum_{i=1}^r \#\{1 \leq \l \leq \lo \mid \s_i\cz(\ga_i^\l) < 0,\ \ga_i^{k_i\pm \l}\text{ is good, }\cz(\ga_i^\l) \notin [-p,p]\text{ and }\cz(\ga_i^\l)\text{ is odd}\},
\end{equation}
and
\[
c_{j,i} = \#\{k \in \N;\,\ga_i^k\text{ is good and }\cz(\ga_i^k)=j\}.
\]
Notice that
\begin{equation}
\label{eq:c_ji}
c_{j,i} = \#\{1\leq\l\leq\lo;\,\ga_i^\l\text{ is good and }\cz(\ga_i^\l)=j\},
\end{equation}
whenever $j \in [-p,p]$. Indeed, from the definition of $\lo$, we have, for every $k\in \N$, that $\cz(\ga_i^{\lo+k}) \geq 2n+p+1 > p$ if $\s_i>0$ and $\cz(\ga_i^{\lo+k}) \leq -2n-p-1 < -p$ if $\s_i<0$.

Consider now $ \sum_{m=-\infty}^\infty (-1)^m c_m $ to which all good orbits from classes (A)--(E) and the collection $\{\gamma_i^{k_i}\}$ contribute. In particular, the contributions of classes (B1--) and (B2+) and classes (C1--) and (C2+) iterates are respectively $c_-^e-c_-^o$ and $c_+^e-c_+^o$. Viewing \eqref{eq:c_m} as
$$ 
\sum_{m=-d-p}^{d+p} (-1)^m c_m=\sum_{m=-\infty}^\infty (-1)^m c_m -
\sum_{|m|>d+p} (-1)^m c_m,
$$ 
with the above discussion in mind, we have
\begin{align}
\label{eq:contributions1}
\sum_{m=-d-p}^{d+p} (-1)^m c_m 
= & \sum_{i=1}^r\bigg(\underbrace{\sum_{\l=1}^{k_i}
  \chi(\ga_i^\l)  - \sum_{j=-p}^p (-1)^jc_{j,i}}_{(A)+(B)+\chi(\ga_i^{k_i})}\bigg)
+ \underbrace{c_+^e-c_+^o}_{(C1-)+(C2+)}
- \underbrace{(c_-^e-c_-^o)}_{(B1-)+(B2+)}  \nonumber \\
& - \sum_{\substack{i=1\\ |\cz(\ga_i^{k_i})|>d+p}}^r \chi(\ga_i^{k_i})
+ \underbrace{2\sum_{j=-p}^p (-1)^jc_{j}}_{(E)}.
\end{align}
Here, as indicated by the underbraces, the first term on the right-hand side comes from the iterates in classes (A) and (B) and the iterate $\ga_i^{k_i}$ (note that we are subtracting the term $\sum_{j=-p}^p (-1)^jc_{j,i}$ because we exclude in (B) the orbits $\ga_i^{k_i-\l}$ with $1\leq \l \leq \lo$ and $\cz(\ga_i^\l) \in [-p,p]$), the second term comes from classes (C1--) and (C2+) iterates, and the third term cancels out the contribution of classes (B1--) and (B2+) orbits to the first term. The second to last term eliminates the contribution to the first term of the orbits $\ga_i^{k_i}$ with index with modulus greater than $d+p$ and the last term is the contribution from (E).

Note that, since $\cz(\ga^{k_i-\l})$ and $\cz(\ga^{k_i+\l})$ have the same parity, $c_-^e = c_+^e$ and $c_-^o = c_+^o$.  Thus the second and third underbraced terms on the right-hand side of equation \eqref{eq:contributions1} cancel each other out and we arrive at
\begin{equation}
\label{eq:contributions2}
\sum_{m=-d-p}^{d+p} (-1)^m c_m 
= \sum_{i=1}^r\sum_{\l=1}^{k_i} \chi(\ga_i^\l) 
- \sum_{\substack{i=1\\ |\cz(\ga_i^{k_i})|>d+p}}^r \chi(\ga_i^{k_i}) + \sum_{j=-p}^p (-1)^jc_{j}.
\end{equation}
Define
\begin{equation}
\label{eq:r^{e,p}}
r^{e,p}_\pm = \#\{1 \leq i \leq r \mid 
\pm(\sigma_i\cz(\ga_i^{k_i}) - d) > p,\ 
\ga_i^{k_i}\text{ is good and }\cz(\ga_i^{k_i})\text{ is even}\}
\end{equation}
and
\begin{equation}
\label{eq:r^o}
r^{o,p}_\pm = \#\{1 \leq i \leq r \mid 
\pm(\sigma_i\cz(\ga_i^{k_i}) - d) > p,\ 
\ga_i^{k_i}\text{ is good and }\cz(\ga_i^{k_i})\text{ is odd}\}.
\end{equation}
Essentially, the terms $r^{e,p}_\pm$ and $r^{o,p}_\pm$ count the good iterates $\ga_i^{k_i}$ whose indices lie outside the interval $[\sigma_i d-p,\sigma_i d+p]$, distinguished according to parity of $\cz(\ga_i^{k_i})$ and the sign of $\cz(\ga_i^{k_i}) - \sigma_i d$.

Notice that the second term in \eqref{eq:contributions2} is given by $r^{e,p}_+ - r^{o,p}_+$ because, by \eqref{eq:d1} and \eqref{eq:IRT3}, $\cz(\ga_i^{k_i})$ and $\mi(\ga_i^{k_i})$ have the same sign $\s_i$ (it also follows from \eqref{eq:k_i} and the definition of $\lo$). Then, by Lemma \ref{lemma:resonance},
\begin{align}
\label{eq:morse_ineq}
d\chi_+(M) - r^{e,p}_+ + r^{o,p}_+ & = \sum_{m=-d-p}^{d+p} (-1)^m c_m - \sum_{j=-p}^p (-1)^jc_{j}.
\end{align}

Now, let us break down the argument according to the parity of $p$.

\vskip .3cm
\noindent
{\bf Case 1. $p$ is even}
\vskip .2cm

In this case, since both $d$ and $p$ are even, we conclude from \eqref{eq:morse_ineq} and the Morse inequalities that
\begin{align}
\label{eq:lower estimate r^{o,p}_+}
r^{o,p}_+ & \geq -d\chi_+(M) + \sum_{m=-d-p}^{d+p} (-1)^m b_m - \sum_{j=-p}^p (-1)^jc_{j} + r^{e,p}_+ \nonumber \\
& = -d\chi_+(M) + \sum_{m=-\infty}^{d+p} (-1)^m b_m - \sum_{j=-p}^p (-1)^jc_{j} + r^{e,p}_+ \nonumber \\
& = -q\chi_+(M) + \sum_{m=-\infty}^{q+p} (-1)^m b_m - \sum_{j=-p}^p (-1)^jc_{j} + r^{e,p}_+,
\end{align}
where the second equation uses \eqref{eq:d2} and the last equality is a consequence of Proposition \ref{prop:periodicity} and the facts that $q$ and $d$ are both multiples of $\D$ and $d\geq q\geq P$ by \eqref{eq:d1}.

\vskip .3cm
\noindent
{\bf Case 2. $p$ is odd}
\vskip .2cm

Similarly to the previous case, since $d$ is even and $p$ is odd, from \eqref{eq:morse_ineq} and the Morse inequalities we obtain
\begin{align}
\label{eq:lower estimate r^{e,p}_+}
r^{e,p}_+ & \geq d\chi_+(M) - \sum_{m=-d-p}^{d+p} (-1)^m b_m + \sum_{j=-p}^p (-1)^jc_{j} + r^{o,p}_+ \nonumber \\
& = d\chi_+(M) - \sum_{m=-\infty}^{d+p} (-1)^m b_m + \sum_{j=-p}^p (-1)^jc_{j} + r^{o,p}_+ \nonumber \\
& = q\chi_+(M) - \sum_{m=-\infty}^{q+p} (-1)^m b_m + \sum_{j=-p}^p (-1)^jc_{j} + r^{o,p}_+,
\end{align}
where, as before, the second equation follows from \eqref{eq:d2} and the last equality is a consequence of Proposition \ref{prop:periodicity}.

\begin{remark}
\label{rmk:lacunary1}
Assume that $\alpha$ is lacunary as in Theorem \ref{thm:main2}. Then, since the Morse inequalities are equalities, in place of \eqref{eq:lower estimate r^{o,p}_+} and \eqref{eq:lower estimate r^{e,p}_+} we have
\[
r^{e,p}_+ = 0\quad\text{and}\quad r^{o,p}_+ = -q\chi_+(M) + \sum_{m=-\infty}^{q+p} (-1)^m b_m - \sum_{j=-p}^p (-1)^jb_{j}
\]
if $\alpha$ is odd and
\[
r^{o,p}_+ = 0\quad\text{and}\quad r^{e,p}_+ = q\chi_+(M) - \sum_{m=-\infty}^{q+p} (-1)^m b_m + \sum_{j=-p}^p (-1)^jb_{j}
\]
if $\alpha$ is even.
\end{remark}

\vskip .3cm
\noindent
{\bf Step 2. Existence of other} $(-1)^{p+1}q\chi_+(M) + (-1)^p\sum_{m=-\infty}^{q+p} b_m + (-1)^{p+1} \sum_{j=-p}^p (-1)^jc_{j} + \#\{1\leq i\leq r;\,\cz(\ga_i^{k_i})=p\,(\bmod\,2),\ |\cz(\ga_i^{k_i})|<d-p\ \text{and}\ \ga_i^{k_i}\ \text{is good}\}$ {\bf simple closed orbits}
\vskip .2cm

Now, we claim that if $p$ is even
\begin{equation}
\label{eq:lower estimate r^{o,p}_-}
r^{o,p}_- \geq -q\chi_+(M) + \sum_{m=-\infty}^{q+p} (-1)^m b_m - \sum_{j=-p}^p (-1)^jc_{j} + r^{e,p}_-
\end{equation}
and if $p$ is odd,
\begin{equation}
\label{eq:lower estimate r^{e,p}_-}
r^{e,p}_- \geq q\chi_+(M) - \sum_{m=-\infty}^{q+p} (-1)^m b_m + \sum_{j=-p}^p (-1)^jc_{j} + r^{o,p}_-.
\end{equation}
In order to prove this, observe that applying Theorem \ref{thm:IRT} to $N$, $\eta$ and $\lo$ as above, we obtain positive integers $(d',k'_1,\dots,k'_r)=(d^-_1,k^-_{11},\dots,k^-_{r1})$ such that
\begin{equation*}
\label{eq:IRT1'}
\cz(\ga_i^{k'_i-\l}) = \s_id' - \cz(\ga_i^\l),
\end{equation*}
\begin{equation*}
\label{eq:IRT2'}
\cz(\ga_i^{k'_i+\l}) = \s_id' + \cz(\ga_i^\l),
\end{equation*}
and
\begin{equation}
\label{eq:IRT3'}
\s_i\cz(\ga_i^{k'_i}) - d' = -(\s_i\cz(\ga_i^{k_i}) - d),
\end{equation}
for every $1\leq i \leq r$ and $1 \leq \l \leq \lo$. Arguing as before, we arrive at the equation
\begin{equation}
\label{eq:contributions3}
\sum_{m=-d'+p}^{d'-p} (-1)^m c_m 
= \sum_{i=1}^r\sum_{\l=1}^{k'_i} \chi(\ga_i^\l) - r'^{e,p}_+ + r'^{o,p}_+ + \sum_{j=-p}^p (-1)^jc_{j}\,,
\end{equation}
where, similarly to \eqref{eq:r^{e,p}} and \eqref{eq:r^o},
\begin{equation*}
r'^{e,p}_\pm = \#\{1 \leq i \leq r \mid 
\pm(\s_i\cz(\ga_i^{k'_i}) - d') > p,\ 
\ga_i^{k'_i}\text{ is good and }\cz(\ga_i^{k'_i})\text{ is even}\}
\end{equation*}
and
\begin{equation*}
\label{eq:r'^{o,p}}
r'^{o,p}_\pm = \#\{1 \leq i \leq r \mid 
\pm(\s_i\cz(\ga_i^{k'_i}) - d') > p,\ 
\ga_i^{k'_i}\text{ is good and }\cz(\ga_i^{k'_i})\text{ is odd}\}.
\end{equation*}
Notice that, due to \eqref{eq:IRT3'}, $r^{e,p}_\pm = r'^{e,p}_\mp$ and $r^{o,p}_\pm = r'^{o,p}_\mp$.  Therefore, equation \eqref{eq:contributions3}, together with Lemma \ref{lemma:resonance}, give rise to the relation
\begin{align}
\label{eq:morse_ineq2}
-d'\chi_+(M) - r^{e,p}_- + r^{o,p}_- 
& = \sum_{m=-d'}^{d'} (-1)^m c_m - \sum_{j=-p}^p (-1)^jc_{j}.
\end{align}
Therefore, \eqref{eq:lower estimate r^{o,p}_-} and \eqref{eq:lower estimate r^{e,p}_-} hold just as before.

\begin{remark}
\label{rmk:lacunary2}
As in Remark \ref{rmk:lacunary1}, if $\alpha$ is lacunary, in place of \eqref{eq:morse_ineq2} we have the equality
\[
-d'\chi_+(M) - r^{e,p}_- + r^{o,p}_- = \sum_{m=-d'}^{d'} (-1)^m b_m - \sum_{j=-p}^p (-1)^jb_j
\]
and therefore, if  $\alpha$ is odd,
\[
r^{e,p}_- = 0\quad\text{and}\quad r^{o,p}_- = -q\chi_+(M) + \sum_{m=-\infty}^{q} (-1)^m b_m - \sum_{j=-p}^p (-1)^jb_j,
\]
and, if $\alpha$ is even,
\[
r^{o,p}_- = 0\quad\text{and}\quad r^{e,p}_- = q\chi_+(M) - \sum_{m=-\infty}^{q} (-1)^m b_m + \sum_{j=-p}^p (-1)^jb_j.
\]
\end{remark}

\vskip .3cm

Now, we split the outcome of Steps 1 and 2 according to the parity of $p$:

\vskip .3cm
\noindent
{\bf Case 1. $p$ is even}
\vskip .2cm

Since $r^{o,p}_+$ and $r^{o,p}_-$ count two disjoint sets of orbits, in view of \eqref{eq:lower estimate r^{o,p}_+} and \eqref{eq:lower estimate r^{o,p}_-}, we must have at least $\br_p+r^{e,p}$ distinct simple closed orbits, say, $\ga_1,\dots,\ga_{\br_p},\ga_{\br_p+1},\dots,\ga_{\br+r^{e,p}}$, where
\[
\br_p := \max\{0,2(-q\chi_+(M) + \sum_{m=-\infty}^{q+p} (-1)^m b_m - \sum_{j=-p}^p (-1)^jc_{j})\}
\]
and
\[
r^{e,p} := r^{e,p}_- + r^{e,p}_+ = \#\{1\leq i\leq r;\ |\cz(\ga_i^{k_i})| \notin [d-p,d+p],\ \cz(\ga_i^{k_i})\text{ is even and }\ga_i^{k_i}\text{ is good}\}.
\]
Moreover, from the definitions of $r^{o,p}_+$ and $r^{o,p}_-$, we have that $|\cz(\ga_i^{k_i})| \notin [d-p,d+p]$ for every $1\leq i \leq \br_p+r^{e,p}$. (Note that, as already mentioned, $\cz(\ga_i^{k_i})$ and $\mi(\ga_i)$ have the same sign.)

\vskip .3cm
\noindent
{\bf Case 2. $p$ is odd}
\vskip .2cm

Similarly as in the previous case, since $r^{e,p}_+$ and $r^{e,p}_-$ count two disjoint sets of orbits, in view of \eqref{eq:lower estimate r^{e,p}_+} and \eqref{eq:lower estimate r^{e,p}_-}, we must have at least $\br_p+r^{o,p}$ distinct simple closed orbits, say, $\ga_1,\dots,\ga_{\br_p},\ga_{\br_p+1},\dots,\ga_{\br+r^{o,p}}$, where
\[
\br_p := \max\{0,2(q\chi_+(M) - \sum_{m=-\infty}^{q+p} (-1)^m b_m + \sum_{j=-p}^p (-1)^jc_{j})\}
\]
and
\[
r^{o,p} := r^{o,p}_- + r^{o,p}_+ = \#\{1\leq i\leq r;\ |\cz(\ga_i^{k_i})| \notin [d-p,d+p],\ \cz(\ga_i^{k_i})\text{ is odd and }\ga_i^{k_i}\text{ is good}\}.
\]
Moreover, from the definitions of $r^{e,p}_+$ and $r^{e,p}_-$, we have that $|\cz(\ga_i^{k_i})| \notin [d-p,d+p]$ for every $1\leq i \leq \br_p+r^{o,p}$.

Note that, in both cases, the set $\ga_1,\dots,\ga_{\br},\ga_{\br+1},\dots,\ga_{\br+r^{e,o}}$ might be empty. (Indeed, it is if $p>n$ and $\eta$ is chosen sufficiently small in Theorem \ref{thm:IRT}.)

\begin{remark}
\label{rmk:good1}
Observe that both $r^{o,p}_\pm$ and $r^{e,p}_\pm$ count simple periodic orbits $\ga_i$ such that $\ga_i^{k_i}$ is good.
\end{remark}

\vskip .3cm
\noindent
{\bf Step 3. Existence of more $\sum_{j=-p}^p (c_{d+j}+c_{-d+j}-c_j-c_{-j})$ simple closed orbits and proof of the desired lower bounds}
\vskip .2cm

From the definitions of $r^{o,p}_\pm$ and $r^{e,p}_\pm$, we have that the (good) iterates $\ga_i^{k_i}$ of the orbits $\ga_i$ obtained in Steps 1 and 2 have index outside the interval $[\s_id-p,\s_id+p]$. By our conditions on $d$, we have that $d-p>n$. Thus, if $\cz(\ga_i^{k_i}) \in [-d-p,-d+p] \implies \cz(\ga_i^{k_i}) < -n \implies \mi(\ga_i^{k_i}) < 0 \iff \mi(\ga_i) < 0$. Similarly, if $\cz(\ga_i^{k_i}) \in [d-p,d+p] \implies \cz(\ga_i^{k_i}) > n \implies \mi(\ga_i^{k_i}) > 0 \iff \mi(\ga_i) > 0$. Therefore, the orbits $\ga_i$ obtained in Steps 1 and 2 satisfy $\cz(\ga_i^{k_i}) \notin [-d-p,-d+p] \cup [d-p,d+p]$. Hence, from \eqref{eq:IRT1} and \eqref{eq:IRT2} we conclude that we have at least more
\[
\sum_{j=-p}^p (c_{d+j}+c_{-d+j}-c_j-c_{-j})
\]
new periodic orbits. As a matter of fact, $c_{d+j}+c_{-d+j}-c_j-c_{-j}$, with $j \in [-p,p]$, counts the number of periodic orbits $\ga_i$ such that $\ga_i^{k_i}$ is good and $\cz(\ga_i^{k_i}) = \pm d+j \in [-d-p,-d+p] \cup [d-p,d+p]$ which therefore must be different from the orbits $\ga_1,\dots,\ga_{\br_p},\ga_{\br_p+1}.\dots,\ga_{\br_p+r^{e,p}}$ when $p$ is even and  different from the orbits $\ga_1,\dots,\ga_{\br_p},\ga_{\br_p+1}.\dots,\ga_{\br_p+r^{o,p}}$ when $p$ is odd.

\begin{remark}
\label{rmk:good2}
As in Remark \ref{rmk:good1}, we have that $\sum_{j=-p}^p (c_{d+j}+c_{-d+j}-c_j-c_{-j})$ counts simple periodic orbits $\ga_i$ such that $\ga_i^{k_i}$ is good.
\end{remark}

Let us break down the argument again according to the parity of $p$:

\vskip .3cm
\noindent
{\bf Case 1. $p$ is even}
\vskip .2cm

From the previous discussions,

\begin{align*}
\#\P_\alpha & \geq  2(-q\chi_+(M) + \sum_{m=-\infty}^{q+p} (-1)^m b_m - \sum_{j=-p}^p (-1)^jc_{j}) + \sum_{j=-p}^p (c_{d+j}+c_{-d+j}-c_j-c_{-j}) + r^{e,p}.
\end{align*}

Given $m\in [-p,p]$, set
\begin{align*}
r^e_{m} = & \#\{1\leq i\leq r;\ |\cz(\ga_i^{k_i})| \in [d-m,d+m],\ \cz(\ga_i^{k_i})\text{ is even and }\ga_i^{k_i}\text{ is good}\},
\end{align*}
and note that
\[
r^e_{m} = \sum_{j=-m,\,j\,\text{even}}^{m} (c_{d+j}+c_{-d+j}-c_j-c_{-j}).
\]

When $p=0$ we obtain, since, by hypothesis, $c_0=0$,
\begin{align*}
\#\P_\alpha & \geq  2(-q\chi_+(M) + \sum_{m=-\infty}^{q} (-1)^m b_m) + c_{d} + c_{-d} + r^{e,p} \\
& = 2(-q\chi_+(M) + \sum_{m=-\infty}^{q} (-1)^m b_m) + r^{e,0} + r^e_0.
\end{align*}
Notice here that, since $c_0=0$,
\begin{equation}
\label{eq:even orbits1}
r^e_0 = c_d + c_{-d} = \#\{1\leq i\leq r;\ \cz(\ga_i^{k_i}) \in \{-d,d\},\ \text{and}\ \ga_i^{k_i}\text{ is good}\}.
\end{equation}

When $p>0$ we argue as follows. First, note that
\begin{align*}
\#\P_\alpha & \geq  2(-q\chi_+(M) + \sum_{m=-\infty}^{q+p} (-1)^m b_m  - \sum_{j=-p}^p (-1)^jc_{j}) + \sum_{j=-p}^p (c_{d+j}+c_{-d+j}-c_j-c_{-j}) + r^{e,p} \\
& = 2(-q\chi_+(M) + \sum_{m=-\infty}^{q} (-1)^m b_m) + r^{e,p} \\
& + \underbrace{\sum_{j=-p}^p (c_{d+j}+c_{-d+j}-c_j-c_{-j}) + 2\sum_{j=1}^p (-1)^jb_{d+j} - 2\sum_{j=-p}^p (-1)^jc_{j}}_{(*)},
\end{align*}
where in the equality we use the periodicity of $\HC_*(M)$.

Since, by \eqref{eq:IRT1} and \eqref{eq:IRT2}, $(c_{d+j}+c_{-d+j}-c_j-c_{-j}) \geq 0$ for every $j \in [-p,p]$ and, by assumption, $c_p=c_{-p}=0$ we arrive at
\begin{align*}
(*) & \geq \sum_{j=-p}^p (-1)^{j+1}(c_{d+j}+c_{-d+j}-c_j-c_{-j}) + 2(c_{-d-p}+c_{-d+p}+c_{d-p}+c_{d+p}) + 2r^e_{p-1}\\
& + 2\sum_{j=1}^p (-1)^jb_{d+j} - 2\sum_{j=-p}^p (-1)^jc_{j} \\
& = -\sum_{j=-p}^p (-1)^{j}(c_{d+j}+c_{-d+j}) + \sum_{j=-p}^p (-1)^{j}b_{d+j} - b_d + 2\sum_{j=-p}^p (-1)^jc_{j} - 2\sum_{j=-p}^p (-1)^jc_{j} \\
& + 2(c_{d-p}+c_{d+p}+c_{-d-p}+c_{-d+p}) + 2r^e_{p-1} \\
& = \underbrace{\sum_{j=-p+1}^{p-1} (-1)^{j}b_{d+j} - \sum_{j=-p+1}^{p-1} (-1)^{j}c_{d+j}}_{\geq 0\text{ by Morse inequalities and the fact that }p\text{ is even}} + b_{d-p} + b_{d+p} - c_{d-p} - c_{d+p} - c_{-d-p} - c_{-d+p} \\
&+ 2(c_{d-p}+c_{d+p}+c_{-d-p}+c_{-d+p}) - b_d - \underbrace{\sum_{j=-p+1}^{p-1} (-1)^{j}c_{-d+j}}_{\leq 0} + 2r^e_{p-1}\\
& \geq b_{d-p} + b_{d+p} - b_d + c_{d-p} + c_{d+p} + c_{-d-p} + c_{-d+p} + 2r^e_{p-1} \\
& = 2b_{d+p} - b_d + c_{d-p} + c_{d+p} + c_{-d-p} + c_{-d+p} + 2r^e_{p-1} \\
& = 2b_{q+p} - b_q + c_{d-p} + c_{d+p} + c_{-d-p} + c_{-d+p} + 2r^e_{p-1}, 
\end{align*}
where the first inequality uses the facts that $p$ is even and $c_p=c_{-p}=0$, the first equality holds because
\[
2\sum_{j=1}^p (-1)^jb_{d+j} = \sum_{j=-p}^p (-1)^{j}b_{d+j} - b_d
\]
by our assumption that $\D=2$ or using Propositions \ref{prop:PD preq} or \ref{prop:PD finite} (note that Proposition \ref{prop:PD finite} is proved in Section \ref{sec:proofmain2}) when $p>1$ and the fact that $d-p\geq P$ (by \eqref{eq:d1}) which is also used in the second to last equation, the last inequality holds because, since $-d+p<\bmin$ and $p$ is even,
\[
\sum_{j=-p+1}^{p-1} (-1)^{j}c_{-d+j} \leq \sum_{j=-p+1}^{p-1} (-1)^{j}b_{-d+j} = 0
\]
and the last equality uses the periodicity of $\HC_*(M)$.

Similarly as in \eqref{eq:even orbits1}, since $c_p=c_{-p}=0$,
\begin{equation*}
c_{d-p} + c_{d+p} + c_{-d-p} + c_{-d+p} = \#\{1\leq i\leq r;\ \cz(\ga_i^{k_i}) \in \{-d-p,-d+p,d-p,d+p\},\ \text{and}\ \ga_i^{k_i}\text{ is good}\}.
\end{equation*}
Consequently, from the definitions of $r^e_{p-1}$ and $r^e_{p}$,
\[
c_{d-p} + c_{d+p} + c_{-d-p} + c_{-d+p} + 2r^e_{p-1} \geq r^e_{p}.
\]

Hence, summing up the results, we have, for any even $p$,
\begin{equation}
\label{eq:bound even}
\#\P_\alpha \geq 
\begin{cases}
2(-q\chi_+(M) + \sum_{m=-\infty}^{q} (-1)^m b_m) + r^{e,p} + r^e_{p} \text{ if } p=0 \\
2(-q\chi_+(M) + \sum_{m=-\infty}^{q} (-1)^m b_m) + 2b_{q+p} - b_q + r^{e,p} + r^e_{p} \text{ if } p\ \text{is even and}\,\geq 2.
\end{cases}
\end{equation}
Since $r^{e,p} + r^e_{p} \geq 0$ it proves the desired lower bound.

Let us prove that if the equality in the theorem holds then $\alpha$ is lacunary with odd parity. We will prove it when $p\geq 2$ since the argument for $p=0$ is analogous. From the discussion above, we actually have that
\[
\#\P^{\text{good}}_\alpha \geq 2(-q\chi_+(M) + \sum_{m=-\infty}^{q} (-1)^m b_m) + 2b_{q+p} - b_q + r^{e,p} + r^e_{p}
\]
where $\#\P^{\text{good}}_\alpha$ is the number of simple closed orbits $\ga_i$ of $\alpha$ such that
\[
\ga_i^{2}\ \text{is good}\iff \ga_i^{k_i}\ \text{is good}.
\]
(Recall that $k_i$ is even.) See Remarks \ref{rmk:good1} and \ref{rmk:good2}.

Suppose now that
\begin{equation}
\label{eq:equality bound1}
\#\P_\alpha = 2(-q\chi_+(M) + \sum_{m=-\infty}^{q} (-1)^m b_m) + 2b_{q+p} - b_q.
\end{equation}
So that, by \eqref{eq:bound even} and the definitions of $r^{e,p}$ and $r^e_{p}$,
\begin{align*}
r^e & := r^{e,p} + r^e_{p} \\
& = \#\{1\leq i\leq r;\ \cz(\ga_i^{k_i})\text{ is even and }\ga_i^{k_i}\text{ is good}\} \\
& = \#\{1\leq i\leq r;\ \cz(\ga_i^j)\text{ is even for every }j\} \\
& = 0,
\end{align*}
where the third equation holds because the parities of the indices of the iterates $\ga_i^j$ depend only on the parity of $j$, that is, $\cz(\ga_i^j)=\cz(\ga_i^{j'})\,(\bmod\,2)$ whenever $j=j'\,(\bmod\,2)$.

We claim that there is no periodic orbit $\ga_i$ such that $\cz(\ga_i^j)$ is even for some $j$. As a matter of fact, suppose, by contradiction, that there exists at least one such an orbit. If $\ga_i^2$ is good $\iff$ $\ga_i^{k_i}$ is good, then $\cz(\ga_i^{k_i})$ is even, contradicting the fact that $r^e=0$. Thus, we necessarily have that $\ga_i^{k_i}$ is bad. Then,
\[
\#\P_\alpha \geq \#\P^{\text{good}}_\alpha + 1 \geq 2(-q\chi_+(M) + \sum_{m=-\infty}^{q} (-1)^m b_m) + 2b_{q+p} - b_q + 1
\]
contradicting \eqref{eq:equality bound1}. It proves the desired result.

\begin{remark}
\label{rmk:non-hyperbolic p even}
From the discussion above, the periodic orbits $\ga_1,\dots,\ga_{\br_p},\ga_{\br_p+1},\dots,\ga_{\br_p+r^{e,p}}$ satisfy $\cz(\ga_i^{k_i}) \neq d$ if $\s_i=1$ and $\cz(\ga_i^{k_i}) \neq -d$ if $\s_i=-1$. Therefore, by item \ref{cond:i} of Theorem \ref{thm:IRT}, they cannot be hyperbolic. By the same reason, we have at least more $\sum_{j=-p}^p (c_{d+j}+c_{-d+j}-c_j-c_{-j})-(c_d+c_{-d}-2c_0)$ non-hyperbolic orbits. Thus, when $p=0$, it gives us
\[
\#\P^{\text{non-hyp}}_\alpha \geq 2(-q\chi_+(M) + \sum_{m=-\infty}^{q} (-1)^m b_m),
\]
where $\#\P^{\text{non-hyp}}_\alpha$ stands for the number of simple non-hyperbolic closed orbits. When $p$ is even and positive, we have, arguing as before,
\begin{align*}
\#\P^{\text{non-hyp}}_\alpha & \geq  2(-q\chi_+(M) + \sum_{m=-\infty}^{q+p} (-1)^m b_m - \sum_{j=-p}^p (-1)^jc_{j}) \\
&+ \sum_{j=-p}^p (c_{d+j}+c_{-d+j}-c_j-c_{-j}) -(c_d+c_{-d}-2c_0)\\
& \geq 2(-q\chi_+(M) + \sum_{m=-\infty}^{q} (-1)^m b_m) + \\
& \underbrace{\sum_{j=-p}^p (c_{d+j}+c_{-d+j}-c_j-c_{-j}) + 2\sum_{j=1}^p (-1)^jb_{d+j} - 2\sum_{j=-p}^p (-1)^jc_{j}-(c_d+c_{-d}-2c_0)}_{(**)}.
\end{align*}
But, since $(c_{d+j}+c_{-d+j}-c_j-c_{-j}) \geq 0$ for every $j \in [-p,p]$, we also have that
\begin{align*}
(**) & \geq \sum_{j=-p}^p (-1)^{j+1}(c_{d+j}+c_{-d+j}-c_j-c_{-j}) + 2(c_{d-p}+c_{d+p}+c_{-d-p}+c_{-d+p}) \\
&+ 2\sum_{j=1}^p (-1)^jb_{d+j} - 2\sum_{j=-p}^p (-1)^jc_{j} \\
& \geq 2b_{q+p} - b_q, 
\end{align*}
just as before. (Note that the exponent at the first alternating sum is odd when $j=0$.) Consequently,
\[
\#\P^{\text{non-hyp}}_\alpha \geq 
\begin{cases}
2(-q\chi_+(M) + \sum_{j=-\infty}^q (-1)^jb_j)\ \text{if}\ p=0 \\
2(-q\chi_+(M) + \sum_{j=-\infty}^q (-1)^jb_j) + 2b_{q+p} - b_q\ \text{if}\ p\ \text{is even and}\ \geq 2. \\
\end{cases}
\]
Here we do not count periodic orbits $\ga_i$ such that $\cz(\ga_i^{k_i})$ is even because it is not needed for the proof of Theorem \ref{thm:non-hyp-general}.
\end{remark}

\vskip .3cm
\noindent
{\bf Case 2. $p$ is odd}
\vskip .2cm

This case is similar to the previous one. As before,

\begin{align*}
\#\P_\alpha & \geq  2(q\chi_+(M) - \sum_{m=-\infty}^{q+p} (-1)^m b_m + \sum_{j=-p}^p (-1)^jc_{j}) + \sum_{j=-p}^p (c_{d+j}+c_{-d+j}-c_j-c_{-j}) + r^{o,p}.
\end{align*}

Given $m\in [-p,p]$, set
\begin{align*}
r^o_{m} = & \#\{1\leq i\leq r;\ |\cz(\ga_i^{k_i})| \in [d-m,d+m],\ \cz(\ga_i^{k_i})\text{ is odd and }\ga_i^{k_i}\text{ is good}\},
\end{align*}
and note that
\[
r^o_{m} = \sum_{j=-m,\,j\,\text{odd}}^{m} (c_{d+j}+c_{-d+j}-c_j-c_{-j}).
\]

When $p=1$ we obtain, since, by hypothesis, $c_1=c_{-1}=0$,
\begin{align*}
\#\P_\alpha & \geq  2(q\chi_+(M) - \sum_{m=-\infty}^{q+1} (-1)^m b_m + c_0) + c_{d} - 2c_0 + c_{d-1} + c_{d+1} + c_{-d} + c_{-d-1} + c_{-d+1} + r^{o,1} \\
& = 2(q\chi_+(M) - \sum_{m=-\infty}^{q} (-1)^m b_m) + c_{d} + c_{d-1} + c_{d+1} + 2b_{q+1} + c_{-d} + c_{-d-1} + c_{-d+1} + r^{o,1} \\
& \geq 2(q\chi_+(M) - \sum_{m=-\infty}^{q} (-1)^m b_m) + b_{d} + 2b_{q+1} + r^{o,1} + r^o_1 \\
& = 2(q\chi_+(M) - \sum_{m=-\infty}^{q} (-1)^m b_m) + b_{q} + 2b_{q+1} + r^{o,1} + r^o_1,
\end{align*}
where in the last equality we used the periodicity of $\HC_*(M)$ and the fact that $q\geq P$. Note here that, since $c_{-1}=c_1=0$,
\begin{align*}
\label{eq:odd orbits1}
r^o_1 & = c_{d-1} + c_{d+1} + c_{-d-1} + c_{-d+1} \\
& = \#\{1\leq i\leq r;\ \cz(\ga_i^{k_i}) \in \{-d-1,-d+1,d-1,d+1\},\,\text{and}\ \ga_i^{k_i}\text{ is good}\}. \numberthis
\end{align*}

When $p>1$ we argue analogously as in the previous case. First,
\begin{align*}
\#\P_\alpha & \geq  2(q\chi_+(M) - \sum_{m=-\infty}^{q+p} (-1)^m b_m + \sum_{j=-p}^p (-1)^jc_{j}) + \sum_{j=-p}^p (c_{d+j}+c_{-d+j}-c_j-c_{-j}) + r^{o,p} \\
& = 2(q\chi_+(M) - \sum_{m=-\infty}^{q} (-1)^m b_m) + r^{o,p} \\
& +  \underbrace{\sum_{j=-p}^p (c_{d+j}+c_{-d+j}-c_j-c_{-j}) - 2\sum_{j=1}^p (-1)^jb_{d+j} + 2\sum_{j=-p}^p (-1)^jc_{j}}_{(*)}.
\end{align*}

Since, by \eqref{eq:IRT1} and \eqref{eq:IRT2}, $(c_{d+j}+c_{-d+j}-c_j-c_{-j}) \geq 0$ for every $j \in [-p,p]$ and, by assumption, $c_p=c_{-p}=0$ we have that
\begin{align*}
(*) & \geq \sum_{j=-p}^p (-1)^{j}(c_{d+j}+c_{-d+j}-c_j-c_{-j}) + 2(c_{d-p}+c_{d+p}+c_{-d-p}+c_{-d+p}) + 2r^o_{p-1} \\
& - 2\sum_{j=1}^p (-1)^jb_{d+j} + 2\sum_{j=-p}^p (-1)^jc_{j} \\
& = \sum_{j=-p}^p (-1)^{j}(c_{d+j}+c_{-d+j}) - \sum_{j=-p}^p (-1)^{j}b_{d+j} + b_d - 2\sum_{j=-p}^p (-1)^jc_{j} + 2\sum_{j=-p}^p (-1)^jc_{j} \\
& + 2(c_{d-p}+c_{d+p}+c_{-d-p}+c_{-d+p}) + 2r^o_{p-1} \\
& = \underbrace{\sum_{j=-p+1}^{p-1} (-1)^{j}c_{d+j} - \sum_{j=-p+1}^{p-1} (-1)^{j}b_{d+j}}_{\geq 0\text{ by Morse inequalities and the fact that }p\text{ is odd}} + b_{d-p} + b_{d+p} - c_{d-p} - c_{d+p} - c_{-d-p} - c_{-d+p} \\
& + 2(c_{d-p}+c_{d+p}+c_{-d-p}+c_{-d+p}) + b_d + \underbrace{\sum_{j=-p+1}^{p-1} (-1)^{j}c_{-d+j}}_{\geq 0} + 2r^o_{p-1} \\
& \geq b_{d-p} + b_{d+p} + b_d + c_{d-p} + c_{d+p} + c_{-d-p} + c_{-d+p} + 2r^o_{p-1} \\
& = 2b_{d+p} + b_d + c_{d-p} + c_{d+p} + c_{-d-p} + c_{-d+p} + 2r^o_{p-1} \\
& = 2b_{q+p} + b_q + c_{d-p} + c_{d+p} + c_{-d-p} + c_{-d+p} + 2r^o_{p-1},
\end{align*}
where the first inequality uses the facts that $p$ is odd and $c_p=c_{-p}=0$, the first equality holds because
\[
2\sum_{j=1}^p (-1)^jb_{d+j} = \sum_{j=-p}^p (-1)^{j}b_{d+j} - b_d
\]
by our assumption that $\D=2$ or using Propositions \ref{prop:PD preq} or \ref{prop:PD finite} when $p>1$ and the fact that $d-p\geq P$ (by \eqref{eq:d1}) which is also used in the second to last equation, the last inequality holds because, since $-d+p<\bmin$ and $p$ is odd,
\[
\sum_{j=-p+1}^{p-1} (-1)^{j}c_{-d+j} \geq \sum_{j=-p+1}^{p-1} (-1)^{j}b_{-d+j} = 0
\]
and the last equality uses the periodicity of $\HC_*(M)$. 

Similarly as in \eqref{eq:even orbits1}, since $c_p=c_{-p}=0$,
\begin{equation*}
c_{d-p} + c_{d+p} + c_{-d-p} + c_{-d+p} = \#\{1\leq i\leq r;\ \cz(\ga_i^{k_i}) \in \{-d-p,-d+p,d-p,d+p\},\ \text{and}\ \ga_i^{k_i}\text{ is odd}\}.
\end{equation*}
Consequently, from the definitions of $r^o_{p-1}$ and $r^o_{p}$,
\[
c_{d-p} + c_{d+p} + c_{-d-p} + c_{-d+p} + 2r^o_{p-1} \geq r^o_{p}.
\]

Therefore, for any odd $p$,
\[
\#\P_\alpha \geq 2(q\chi_+(M) - \sum_{m=-\infty}^{q} (-1)^m b_m) + 2b_{q+p} + b_q + r^{o,p} + r^o_{p}.
\]
Since $r^{o,p} + r^o_{p} \geq 0$ it proves the desired lower bound.

To prove that if the equality in the theorem holds then $\alpha$ is lacunary with even parity, we argue as in the previous case. From the discussion above, we have that
\begin{equation}
\label{eq:bound odd}
\#\P^{\text{good}}_\alpha \geq 2(q\chi_+(M) - \sum_{m=-\infty}^{q} (-1)^m b_m) + 2b_{q+p} + b_q + r^{o,p} + r^o_{p}
\end{equation}
where $\#\P^{\text{good}}_\alpha$ is the number of simple closed orbits $\ga_i$ of $\alpha$ such that
\[
\ga_i^{2}\ \text{is good}\iff \ga_i^{k_i}\ \text{is good}.
\]
See Remarks \ref{rmk:good1} and \ref{rmk:good2}.

Suppose now that
\begin{equation}
\label{eq:equality bound2}
\#\P_\alpha = 2(q\chi_+(M) - \sum_{m=-\infty}^{q} (-1)^m b_m) + 2b_{q+p} + b_q.
\end{equation}
So, by \eqref{eq:bound odd} and the definitions of $r^{o,p}$ and $r^o_{p}$,
\begin{align*}
r^o & := r^{o,p} + r^o_{p} \\
& = \#\{1\leq i\leq r;\ \cz(\ga_i^{k_i})\text{ is odd and }\ga_i^{k_i}\text{ is good}\} \\
& = \#\{1\leq i\leq r;\ \cz(\ga_i^j)\text{ is odd for every }j\} \\
& = 0,
\end{align*}
where, as before, the third equation holds because the parities of the indices of the iterates $\ga_i^j$ depend only on the parity of $j$.

We claim that there is no periodic orbit $\ga_i$ such that $\cz(\ga_i^j)$ is odd for some $j$. As a matter of fact, suppose, by contradiction, that there exists at least one such an orbit. If $\ga_i^2$ is good $\iff$ $\ga_i^{k_i}$ is good, then $\cz(\ga_i^{k_i})$ is odd, contradicting the fact that $r^o=0$. Thus, we necessarily have that $\ga_i^{k_i}$ is bad. Then,
\[
\#\P_\alpha \geq \#\P^{\text{good}}_\alpha + 1 \geq 2(q\chi_+(M) - \sum_{m=-\infty}^{q} (-1)^m b_m) + 2b_{q+p} + b_q + 1
\]
contradicting \eqref{eq:equality bound2}. It proves the desired result.

\begin{remark}
\label{rmk:non-hyperbolic p odd}
As in Remark \ref{rmk:non-hyperbolic p even}, the periodic orbits $\ga_1,\dots,\ga_{\br_p},\ga_{\br_p+1},\dots,\ga_{\br_p+r^{o,p}}$ satisfy $\cz(\ga_i^{k_i}) \neq d$ if $\s_i=1$ and $\cz(\ga_i^{k_i}) \neq -d$ if $\s_i=-1$. Therefore, by item \ref{cond:i} of Theorem \ref{thm:IRT}, they cannot be hyperbolic. By the same reason, we have at least more $\sum_{j=-p}^p (c_{d+j}+c_{-d+j}-c_j-c_{-j})-(c_d+c_{-d}-2c_0)$ non-hyperbolic orbits. Suppose that $p=1$. Then, assuming that $q>P$,
\begin{align*}
\#\P^{\text{non-hyp}}_\alpha & \geq  2(d\chi_+(M) - \sum_{m=-\infty}^{d+1} (-1)^m b_m + c_0) + c_{d-1} + c_{d+1} + c_{-d-1} + c_{-d+1} \\
& \geq 2(d\chi_+(M) - \sum_{m=-\infty}^{d} (-1)^m b_m) + b _{d-1} + 3b_{d+1} + 2b_0 \\
& = 2(q\chi_+(M) - \sum_{m=-\infty}^{q} (-1)^m b_m) + b _{q-1} + 3b_{q+1} + 2b_0,
\end{align*}
where in the last equality we used the periodicity of $\HC_*(M)$.
\end{remark}

\begin{remark}
\label{rmk:lacunary3}
Assume that $\alpha$ is lacunary with odd parity. Then, for all $i$, $\cz(\ga_i^j)\neq 0$ for every $j$ and, since $d$ is even, $\cz(\ga_i^{k_i})\neq d$. Therefore, from the previous discussion (applied to the case where $p=0$) and Remarks \ref{rmk:lacunary1} and \ref{rmk:lacunary2},
\[
r = \#\{1\leq i\leq r;\ \cz(\ga_i^{k_i})<d\text{ or }\cz(\ga_i^{k_i})>d\} = r^{o,0}_- + r^{o,0}_+ = 2(-q\chi_+(M) + \sum_{m=-\infty}^{q} (-1)^mb_m).
\]
Note here that every periodic orbit of $\alpha$ is good. Similarly, if the parity of $\alpha$ is even, there is no closed orbit with index $\pm 1$ and consequently, from the previous argument when $p=1$, we have that the simple orbits $\ga_i$ counted by $r^{e,1}_\pm$ are those for which $\cz(\ga_i^{k_i})\neq d$. As explained above (see Step 3), we have $c_d - 2c_0 = b_d - 2b_0 = b_q - 2b_0$ new simple periodic orbits $\ga_i$ such that $\cz(\ga_i^{k_i})=d$. Thus, we have, by Remarks \ref{rmk:lacunary1} and \ref{rmk:lacunary2},
\begin{align*}
r &= \#\{1\leq i\leq r;\ \cz(\ga_i^{k_i})<d\text{ or }\cz(\ga_i^{k_i})>d\} + \#\{1\leq i\leq r;\ \cz(\ga_i^{k_i})=d\} \\
& = r^{e,1}_- + r^{e,1}_+ + b_q - 2b_0 \\
& = 2(q\chi_+(M) - \sum_{m=-\infty}^{q+1} (-1)^m b_m + \sum_{j=-1}^1 (-1)^jb_{j}) + b_q - 2b_0 \\
& = 2(q\chi_+(M) - \sum_{m=-\infty}^{q} (-1)^m b_m + b_0) + b_q - 2b_0 \\
& = 2(q\chi_+(M) - \sum_{m=-\infty}^{q} (-1)^mb_m) + b_q,
\end{align*}
where in the fourth equality we use that $q$ is even. Note again that every periodic orbit of $\alpha$ is good.
\end{remark}

\section{Proof of Theorem \ref{thm:main2}}
\label{sec:proofmain2}

As explained in Section \ref{sec:ESH}, under the assumptions of the theorem, we have the positive equivariant symplectic homology $\HC_*(M)$. If $M$ does not admit a ``nice'' filling, in the sense of \ref{cond:F}, we assume that the contact form $\alpha$ is index-admissible.

By the discussion in the same section, $\HC_*(M)$ is the homology of a chain complex $\CC_*(\alpha)$ generated by the good closed orbits of $\alpha$ graded by the index. Since $\alpha$ is lacunary, every periodic orbit is good and the differential in $\CC_*(\alpha)$ vanishes. Therefore,
\begin{equation}
\label{eq:perorbs}
\HC_k(M) \cong \CC_k(\alpha) \cong \oplus_{\ga \in \PP(\alpha);\, \cz(\ga)=k} \Nov
\end{equation}
for every $k \in \Z$, where $\PP(\alpha)$ is the set of (not necessarily simple) closed orbits of $\alpha$. In other words, every closed orbit of $\alpha$ contributes to the positive equivariant symplectic homology.

Define
\[
\cmin:=\min\{m \in \Z;\ \CC_m(M)\neq 0\}.
\]
From \eqref{eq:perorbs} and our hypotheses,
\[
\cmin = \bmin > -\infty.
\]
From this fact and the assumption that $b_j < \infty$ for every $j \in \Z$ we conclude from  \eqref{eq:perorbs} that every closed orbit of $\alpha$ has positive mean index.

\vskip .3cm
\noindent
{\bf Claim 1. $\alpha$ has finitely many simple closed orbits}
\vskip .2cm

This part of the proof follows from the work of G\"urel \cite[Theorem 1.5]{Gu}. For the sake of completeness, we will reproduce her argument. Define
\[
b = \limsup_{k\to\infty} \sum_{i=0}^{2n} \dim \HC_{k+i}(M).
\]
By the periodicity of $\HC_*(M)$ and the assumption that the contact Betti numbers are finite, $b$ is a finite integer.

The key ingredient is the following combinatorial lemma proved in \cite{Gu} which uses the aforementioned fact that every closed orbit of $\alpha$ has positive mean index:

\begin{lemma} \cite[Lemma 3.2]{Gu}
\label{lemma:combinatorial}
Assume that $\alpha$ has a collection of $m$ geometrically distinct periodic orbits $\ga_1,\dotsc, \ga_m$ with positive mean index. Then for every sufficiently small $\ep>0$, there exist infinitely many disjoint intervals $I$ of length $2n+\ep$ such that for some positive integers $k_1\geq 1,\dots, k_m \geq 1$ (depending on the interval), the iterated orbits $\ga_1^{k_1},\dotsc, \ga_m^{k_m}$ all have indices in the interval $I$.
\end{lemma}

Since every closed orbit of $\alpha$ has positive mean index, it follows from \eqref{eq:perorbs} and the previous lemma that $\alpha$ has at most $b$ simple closed orbits.

\begin{remark}
Note that the dimension of $M$ in \cite{Gu} is $2n-1$ and here it is $2n+1$.
\end{remark}

As mentioned in Section \ref{sec:proofmain}, once we have that $\alpha$ has finitely many closed orbits, Theorem \ref{thm:main2} follows from the proof of Theorem \ref{thm:main1}; see Remarks \ref{rmk:lacunary1}, \ref{rmk:lacunary2} and \ref{rmk:lacunary3}. However, for the reader's convenience, we will present an easier and shorter proof following the lines of \cite{AM5}. The proof here is actually easier than the proof in \cite{AM5} since we do not have to split it according to the parity of $\alpha$. Denote by $\delta$ the parity of $\alpha$.

\vskip .3cm
\noindent
{\bf Claim 2. $\alpha$ has precisely $2((-1)^\del q\chi_+(M) - \sum_{m=-\infty}^{q} b_m) + b_q$ simple closed orbits}
\vskip .2cm

From the first claim, $\alpha$ has finitely many simple closed orbits $\{\ga_1,\dots,\ga_r\}$. Moreover, $\mi(\ga_i)>0$ for every $1\leq i \leq r$. Take
\[
\ell_0:=\max_{1\leq i \leq r}\{\min\{k_0 \in \N;\ \cz(\ga_i^{k+\ell}) \geq  \cz(\ga_i^{k}) + n + 3\ \forall k\geq 1\text{ and }\forall \ell \geq k_0\}\},
\]
$\eta$ sufficiently small satisfying $\eta<1$ and Lemma \ref{lemma:resonance}, and $N$ in Theorem \ref{thm:IRT} big enough such that $\ell_0+2 \leq \min_{1\leq i \leq r} k_i$.

As in the previous section, by Theorem \ref{thm:IRT}, given $N \in \N$, $\eta>0$, and $\lo$ as above, there exist two sequences of integer vectors $(d^\pm_j,k^\pm_{1j},\dots,k^\pm_{rj})$ satisfying conditions \ref{cond:i}, \ref{cond:ii}, and \ref{cond:iii}, such that all entries $d^\pm_j,k^\pm_{1j},\dots,k^\pm_{rj}$ are divisible by $N$. We will only need one vector from each sequence, and we set
\begin{equation}
\label{eq:integers-lacunary}
(d,k_1,\dots,k_r):=(d^+_1,k^+_{11},\dots,k^+_{r1}), \qquad (d',k'_1,\dots,k'_r):=(d^-_1,k^-_{11},\dots,k^-_{r1}).
\end{equation}

\vskip .3cm
\noindent
{\bf Step 1. $\alpha$ has at least $(-1)^\del q\chi_+(M) - \sum_{m=-\infty}^{q} b_m + b_0$ simple closed orbits}
\vskip .2cm

By property \ref{cond:ii} of Theorem \ref{thm:IRT},
\[
\cz(\ga_i^{k_i -\ell}) = d - \cz(\ga_i^\ell)
\]
\[
\cz(\ga_i^{k_i +\ell}) = d + \cz(\ga_i^\ell)
\]
for every $1 \leq \ell \leq \ell_0$. Therefore,
\begin{equation}
\label{eq:1}
\#\{\ell \in [-\ell_0,\ell_0]\setminus\{0\};\, \cz(\ga_i^{k_i +\ell}) < d\} = \#\{\ell \in [-\ell_0,\ell_0]\setminus\{0\};\, \cz(\ga_i^{k_i +\ell}) > d\}.
\end{equation}
(Note here that $k_i > \ell_0$.) From these two equations we also have that an orbit of index $d$ can happen in two different ways: either it is an orbit of the form $\ga_i^{k_i}$ or an orbit of the form $\ga_i^{k_i+\ell}$ for some $\ell \in [-\ell_0,\ell_0]\setminus\{0\}$. (Indeed, by the definition of $\ell_0$, $\cz(\ga_i^{k_i\pm\ell})\neq d$ for every $\ell>\ell_0$.) In the last case, $\cz(\ga_i^\ell)=0$ and we have the pair of orbits $(\ga_i^{k_i-\ell},\ga_i^{k_i+\ell})$ with index $d$.

Define $c_{0,i}=\#\{j \in \N;\, \cz(\ga_i^j)=0\}$ and ${\bar c}_{d,i}=\#\{j \in [-\ell_0,\ell_0]\setminus\{0\};\, \cz(\ga_i^{k_i+j})=d\}$. By the previous discussion,
\begin{equation}
\label{eq:5}
{\bar c}_{d,i} = 2c_{0,i}.
\end{equation}
Now, let $a_i=\#\{\ell \in [-\ell_0,\ell_0]\setminus\{0\};\, \cz(\ga_i^{k_i +\ell}) < d\}$ and $b_i=\#\{\ell \in [-\ell_0,\ell_0]\setminus\{0\};\, \cz(\ga_i^{k_i +\ell}) > d\}$. Clearly,
\begin{equation}
\label{eq:6}
a_i + b_i + {\bar c}_{d,i} = 2\ell_0.
\end{equation}
By the definition of $\ell_0$, our choice of $\eta$, and the fact that $\cz(\ga) < \mi(\ga)+n$ for every non-degenerate closed orbit $\ga$,
\begin{equation}
\label{eq:3}
\cz(\ga_i^{k_i-\ell}) \leq \cz(\ga_i^{k_i}) - (n+3) \leq \mi(\ga_i^{k_i}) + n - (n+3) \leq d+1 + n - (n+3) = d-2,
\end{equation}
for every $\ell_0+1 \leq \ell < k_i$, where the third inequality holds by property \ref{cond:i} of Theorem \ref{thm:IRT}. Similarly,
\begin{equation}
\label{eq:4}
\cz(\ga_i^{k_i+\ell}) \geq d+2,
\end{equation}
for all $\ell \geq \ell_0$. Using \eqref{eq:1}, \eqref{eq:5} and \eqref{eq:6} we arrive at
\begin{equation}
\label{eq:7}
a_i + c_{0,i} = \ell_0.
\end{equation}

Thus, we conclude that
\begin{align*}
\#\{j \in \N,\, j\neq k_i;\, \cz(\ga_i^j) \leq d\} & = \#\{1\leq j \leq k_i+\ell_0,\, j\neq k_i;\, \cz(\ga_i^j) \leq d\} \\
& = k_i - \ell_0 - 1 + \#\{k_i - \ell_0 \leq j \leq k_i+\ell_0,\, j\neq k_i;\, \cz(\ga_i^j) \leq d\} \\
& = k_i - \ell_0 - 1 + a_i +  {\bar c}_{d,i} \\
& = k_i - \ell_0 - 1 +  a_i + 2c_{0,i}\\
& = k_i - \ell_0 - 1 + \ell_0 + c_{0,i} \\
& = k_i - 1 + c_{0,i},
\end{align*}
where the first equation follows \eqref{eq:4}, the second equation follows from \eqref{eq:3}, the third identity holds by the definitions of $a_i$ and ${\bar c}_{d,i}$, the fourth uses \eqref{eq:5} and the fifth is a consequence of \eqref{eq:7}. Since $\alpha$ is lacunary, $\sum_{i=1}^r c_{0,i} = b_0$. Hence, we conclude that
\begin{align*}
\sum_{m=\cmin}^{d} c_m & = \sum_{i=1}^r \#\{j \in \N;\ \cz(\ga_i^j)\leq d\} \\
& = \sum_{i=1}^r \#\{j \in \N,\, j\neq k_i;\, \cz(\ga_i^j) \leq d\} + \#\{1\leq i \leq r;\, \cz(\ga_i^{k_i})\leq d\} \\
& = \sum_{i=1}^r k_i  - r + (r-r_+) + \sum_{i=1}^r c_{0,i} \\
& = \sum_{i=1}^r k_i - r_+ + b_0,
\end{align*}
where $r_+:=\#\{1\leq i\leq r;\ \cz(\ga_i^{k_i})>d\}$ . Therefore, by Lemma \ref{lemma:resonance},
\begin{equation}
\label{eq:sumc_m-odd-2}
\sum_{m=\cmin}^{d} c_m = (-1)^\del d\chi_+(M) - r_+ + b_0,
\end{equation}
because $\hat{\chi}(\gamma_i)=(-1)^\del$ for every $1\leq i\leq r$ since $\alpha$ is lacunary with parity $\del$.

Thus,
\begin{align*}
r_+ & = (-1)^\del d\chi_+(M) - \sum_{m=\cmin}^{d} c_m + b_0 \nonumber \\
& =  (-1)^\del d\chi_+(M) - \sum_{m=-\infty}^{d} b_m + b_0 \nonumber \\
& = (-1)^\del q\chi_+(M) - \sum_{m=-\infty}^{q} b_m + b_0,
\end{align*}
where the last equality follows from Proposition \ref{prop:periodicity}.

\vskip .3cm
\noindent
{\bf Step 2. Existence of other $(-1)^\del q\chi_+(M) - \sum_{m=-\infty}^{q} b_m + b_0$ simple closed orbits}
\vskip .2cm

Applying the argument from the previous step for the integer vector $(d',k'_1,\dots,k'_r)$ in \eqref{eq:integers-lacunary} provided by Theorem \ref{thm:IRT}, we get
\begin{align*}
r_- & :=\#\{1\leq i\leq r;\ \cz(\ga_i^{k'_i})>d'\} \nonumber \\
& = (-1)^\del d'\chi_+(M) - \sum_{m=\cmin}^{d'} c_m + b_0 \nonumber \\
& =  (-1)^\del d'\chi_+(M) - \sum_{m=-\infty}^{d'} b_m + b_0 \nonumber \\
& = (-1)^\del q\chi_+(M) - \sum_{m=-\infty}^{q} b_m + b_0,
\end{align*}
where, again, the last equality follows from Proposition \ref{prop:periodicity}.

But, by property \ref{cond:iii} of Theorem \ref{thm:IRT}, we have that
\[
\#\{1\leq i\leq r;\ \cz(\ga_i^{k'_i})>d'\} = \#\{1\leq i\leq r;\ \cz(\ga_i^{k_i})<d\}.
\]

\vskip .3cm
\noindent
{\bf Step 3. Existence of more $b_q - 2b_0$ simple closed orbits}
\vskip .2cm

From Steps 1 and 2 we get $2((-1)^\del q\chi_+(M) + \sum_{m=-\infty}^{q} b_m + b_0)$ closed orbits given by the $\ga_i's$ such that $\cz(\ga_i^{k_i})<d$ and $\cz(\ga_i^{k_i})>d$. As mentioned earlier, the set of closed orbits with index $d$ splits into two sets: orbits of the form $\ga_i^{k_i}$ such that $\cz(\ga_i^{k_i})=d$ and pairs of orbits $(\ga_i^{k_i-\ell},\ga_i^{k_i+\ell})$ with $\cz(\ga_i^\ell)=0$. The latter was already counted in Steps 1 and 2 and has cardinality
\[
\sum_{i=1}^r {\bar c}_{d,i} = 2b_0.
\]
So we get more
\begin{align*}
\#\{1\leq i\leq r;\ \cz(\ga_i^{k_i})=d\} & = b_d - 2b_0 \\
& = b_q - 2b_0
\end{align*}
new orbits given by the former set, where the last equality holds by the periodicity of $\HC_*(M)$.

\vskip .3cm
\noindent
{\bf Step 4. Existence of precisely $2((-1)^\del q\chi_+(M) - \sum_{m=-\infty}^{q} b_m) + b_q$ simple closed orbits}
\vskip .2cm

By Steps 1 and 2,
\[
\#\{1\leq i\leq r;\ \cz(\ga_i^{k_i})<d\} + \#\{1\leq i\leq r;\ \cz(\ga_i^{k_i})>d\} = 2((-1)^\del q\chi_+(M) - \sum_{m=-\infty}^{q} b_m) + 2b_0. 
\]
This, together with Step 3, yields
\begin{align*}
r & = \#\{1\leq i\leq r;\ \cz(\ga_i^{k_i})<d\} + \#\{1\leq i\leq r;\ \cz(\ga_i^{k_i})>d\} + \#\{1\leq i\leq r;\ \cz(\ga_i^{k_i})=d\} \\
& = 2((-1)^\del q\chi_+(M) - \sum_{m=-\infty}^{q} b_m) + 2b_0 + b_q - 2b_0 \\
& = 2((-1)^\del q\chi_+(M) - \sum_{m=-\infty}^{q} b_m) + b_q,
\end{align*}
finishing Step 4.

To finish the proof Theorem \ref{thm:main2}, just note that if $\alpha$ is odd then so is $\HC_*(M)$ and consequently
\begin{equation*}
r_M^o = 2(-q\chi_+(M) - \sum_{j=-\infty}^q b_j) = r,
\end{equation*}
since $q$ is even. Similarly, if $\alpha$ is even,
\begin{equation*}
r_M^e = 2(q\chi_+(M) - \sum_{j=-\infty}^{q} b_j) + b_{q} = r.
\end{equation*}

We end up this section with the following proposition, which plays an important role in the proof of Theorem \ref{thm:main1}. As in Proposition \ref{prop:PD preq}, it shows that $\HC_*(M)$ exhibits a symmetry in sufficiently large degrees whenever $M$ satisfies the assumptions of Theorem \ref{thm:main2}; cf. Remark \ref{rmk:PD}. Its proof is a simple consequence of the argument used above in the proof of Theorem \ref{thm:main2}. A particular case, when $M$ is a prequantization of an orbifold admitting a Hamiltonian circle action, is established in Proposition \ref{prop:PD preq_orb} by completely different methods.

\begin{proposition}
\label{prop:PD finite}
Let $(M^{2n+1},\xi)$ be a closed contact manifold satisfying either \ref{cond:F} or \ref{cond:NF}. Assume that $\HC_*(M)$ is periodic, $b_j \neq 0$ for some $j$, $b_j = 0$ for all $j$ sufficiently negative, and $b_j < \infty$ for all $j \in \Z$. Assume that $M$ admits a non-degenerate lacunary contact form. Given an integer $p\geq 0$ there exists $C>0$ such that if $d=s\D$ is bigger than $C$ then
\[
b_{d-j} = b_{d+j}
\]
for every $j \in [-p,p]$.
\end{proposition}

\begin{proof}
Fix a non-degenerate lacunary contact form $\alpha$ on $M$. As in the proof above of Theorem \ref{thm:main2}, we have $b_j=c_j$ for every $j \in \Z$ since the differential in $\CC_*(\alpha)$ vanishes. Take $\lo$ big enough such that all the orbits that contribute to $c_{d\pm j}$, with $j \in [-p,p]$, are of the form $\ga_i^{k_i\pm \l}$ with $\l \in [-\lo,\lo]$. By property \ref{cond:ii} of Theorem \ref{thm:IRT}, we have that an orbit with index $d\pm j$ can happen in two different ways: either it is an orbit of the form $\ga_i^{k_i}$ or an orbit of the form $\ga_i^{k_i+\ell}$ for some $\ell \in [-\ell_0,\ell_0]\setminus\{0\}$ such that $\cz(\ga_i^\l)=j$ or $\cz(\ga_i^\l)=-j$. In this way, we arrive at
\begin{equation}
\label{eq:b_{d+j}}
b_{d+j} = c_{d+j} = c_{-j} + c_j + \#\{1\leq i \leq r;\,\cz(\ga_i^{k_i})=d+j\},
\end{equation}
for every $j \in [-p,p]$.

Now, arguing as in Step 2 above, we have the integer vector $(d',k'_1,\dots,k'_r)$ such that
\[
(d' - \cz(\ga_i^{k'_i})) = -(d - \cz(\ga_i^{k_i})).
\]
Hence, by \eqref{eq:b_{d+j}},
\begin{align*}
\label{eq:b_{d'+j}}
b_{d'+j} & = c_{d'+j} \\
& = c_{-j} + c_j + \#\{1\leq i \leq r;\,\cz(\ga_i^{k'_i})=d'+j\} \\
& = c_{-j} + c_j + \#\{1\leq i \leq r;\,\cz(\ga_i^{k_i})=d-j\} \\
& = b_{d-j}, \numberthis
\end{align*}
for all $j \in [-p,p]$.

Taking $N$ big enough in Theorem \ref{thm:IRT} we can ensure that $d-p>P$ and $d'-p>P$. Since both $d$ and $d'$ are multiples of $\D$, $d'-d$ is a multiple of $\D$ as well and consequently, by the periodicity of $\HC_*(M)$ and \eqref{eq:b_{d'+j}},
\[
b_{d+j} = b_{d'+j} = b_{d-j}
\]
for every $j \in [-p,p]$, proving the desired result.
\end{proof}

\section{Proof of Theorem \ref{thm:r_M}}
\label{sec:proof r_M}

Let us consider each case separately.

\vskip .3cm
\noindent
{\bf Case \ref{PM}}
\vskip .2cm

Suppose that $\H_*(B;\Q)$ is lacunary or that the minimal Chern number $c_B$ is greater than $n$. Moreover, under assumption \ref{cond:NF} assume that $c_B>1$. Under these conditions, in \cite[Proposition 2.1]{AM5} and \cite[Proposition 3.1]{GGM2}, it is computed the positive equivariant symplectic homology for contractible orbits of $M$ given by
\[
\HC^0_\ast(M) \cong \bigoplus_{m\in\N} \H_{\ast -2mc_B+n} (B; \Nov), 
\]
where $c_B$ is the minimal Chern number of $B$.

Let $\bD$ be the Robbin-Salamon index of the simple orbit of the $S^1$-action on $M$ which is positive since $B$ is positive monotone and it is an even number since the linearized Reeb flow on the contact structure generates a loop based at the identity in $\Sp(2n)$ and in this case the Robbin-Salamon is twice the Maslov index of such loop; see, for instance, \cite{RS}. When we do not restrict ourselves to contractible orbits, the proof of \cite[Proposition 2.1]{AM5} and \cite[Proposition 3.1]{GGM2}, works verbatim to show that if $\H_*(B;\Q)$ is lacunary or $\bD>n$ then
\begin{equation}
\label{eq:ESHpreq}
\HC_\ast(M) \cong \bigoplus_{m\in\N} \H_{\ast -m\bD+n} (B; \Nov). 
\end{equation}
As before, under assumption \ref{cond:NF} we have to further assume that the minimal Chern number $c_B$ is bigger than one.

\begin{remark}
The similarity between the notations of the period $\D$ of $\HC_*(M)$ and the Robbin-Salamon index $\bD$ hinges on the fact that it follows easily from \eqref{eq:ESHpreq} that $\bD$ is a period for $\HC_*(M)$. However, since the period is not unique, we rather use different notations.
\end{remark}

The isomorphism \eqref{eq:ESHpreq} should be true without any assumption on $B$, but it is not known so far. Motivated by this, we will prove a more general version of Case \ref{PM} where we do not assume that the rational homology of $B$ is lacunary (cf. Remark \ref{rmk:PM general}) but assuming that \eqref{eq:ESHpreq} holds (which is true if $\bD>n$ as mentioned above). More precisely, we will prove that, in this general case, given an integer $p\geq 0$ with parity different from that of $n$,
\begin{equation}
\label{eq:r_M-PM-general}
r_M=
\begin{cases}
\chi(B) + \dim \H_n(B;\Q) + 2\sum_{m=1}^{\lfloor n/\bD \rfloor} \dim \H_{n-m\bD}(B;\Q) \text{ if }n\text{ is odd and }p=0, \\
\chi(B) + 2\dim \H_{n+p}(B;\Q) + 2\sum_{m=1}^{\lfloor (n+p)/\bD \rfloor} \dim \H_{n-m\bD+p}(B;\Q) \\
+ 2\sum_{m=1}^{\lfloor (n-p)/\bD \rfloor} \dim \H_{n-m\bD-p}(B;\Q), \text{ if }n\text{ is odd and }p\geq 2 \\
\chi(B) + 2\dim \H_{n+p}(B;\Q) + 2\sum_{m=1}^{\lfloor (n-p)/\bD \rfloor} \dim \H_{n-m\bD-p}(B;\Q) \\
+ 2\sum_{m=1}^{\lfloor (n+p)/\bD \rfloor} \dim \H_{n-m\bD+p}(B;\Q) \text{ if }n\text{ is even},
\end{cases}
\end{equation}
where $r_M$ is given by \eqref{eq:r_M general}. Note that it clearly implies that $r_M=\dim \H_*(B;\Q)$ whenever $\H_{\text{odd}}(B;\Q)=0$ since $\bD$ is even and $p\neq n\,(\bmod\,2)$.

We shall need the following lemma essentially taken from \cite[Lemma A.1]{AM5}. Set $d=s\bD$ for some $s\in \N$. In what follows recall that the dimension of $M$ is $2n+1$.

\begin{lemma}
\label{lemma:sumb_m}
We can choose $s$ sufficiently large such that
\[
2\sum_{m=-\infty}^{d} (-1)^mb_m =  (-1)^n(2s-1)\chi(B) + \dim \H_n(B;\Q) + 2\sum_{m=s+1}^{2s-1} \dim \H_{n+(s-m)\bD}(B;\Q).
\]
\end{lemma}

\begin{proof}
Take $s$ big enough such that $s\bD > n$. Change the grading of $\HC_*(M)$ to $*' = 2s\bD - *$ so that
\[
\HC_{*'}(M) \cong \bigoplus_{m \in \N} \H_{2s\bD-*'-m\bD+n}(B;\Nov).
\]
Thus, since $\bmin=\min\{m \in \Z,\,b_m\neq 0\}=\bD-n$,
\begin{align*}
2\sum_{m=-\infty}^{d} (-1)^mb_m & = 2\sum_{*=\bD-n}^{s\bD} (-1)^*b_* \\
& = 2\sum_{*'=s\bD}^{n+(2s-1)\bD} \sum_{m=1}^\infty (-1)^{*'}\dim \H_{(2s-m)\bD - *' + n}(B;\Q) \\
& = 2\sum_{*'=s\bD}^{n+(2s-1)\bD} \sum_{m=1}^{2s-1} (-1)^{*'}\dim \H_{(2s-m)\bD - *' + n}(B;\Q) \\
& = \sum_{m=1}^{2s-1} \overbrace{\sum_{*'=s\bD}^{n+(2s-1)\bD} (-1)^{*'}(\dim \H_{(2s-m)\bD - *' + n}(B;\Q) + \dim \H_{n + *' - (2s-m)\bD }(B;\Q))}^{d_m},
\end{align*}
where the third equality follows from the fact that $s\bD > n$, the last equation holds by Poincar\'e duality and $d_m$ is defined as in the equation.

We have to prove that
\begin{equation}
\label{eq:claim}
\sum_{m=1}^{2s-1} d_m = (-1)^n(2s-1)\chi(B) + \dim \H_n(B;\Q) + 2\sum_{m=s+1}^{2s-1} \dim \H_{n+(s-m)\bD}(B;\Q).
\end{equation}
In order to do it, let us look at the first term of $d_m$. When $*'$ runs from $s\bD$ to $(2s-1)\bD+n$, it sums up the ranks of the homology of $B$ from the degree $-m\bD+\bD\leq 0$ to $n+(s-m)\bD$. Similarly, the second term of $d_m$ sums up the ranks of the homology of $B$ from the degree $n-(s-m)\bD$ to $2n+(m-1)\bD \geq 2n$.

Hence, we have the following computations of $d_m$ according to the relation between $m$ and $s$:
\begin{itemize}
\item $m<s\ \iff\ (s-m)\bD>0$
\begin{align*}
(-1)^nd_m & = \sum_{j=0}^{n+(s-m)\bD} (-1)^j\dim \H_j(B;\Q) + \sum_{j=n-(s-m)\bD}^{2n} (-1)^j\dim \H_j(B;\Q) \\
& = \chi(B) + \sum_{j=n-(s-m)\bD}^{n+(s-m)\bD}  (-1)^j\dim \H_j(B;\Q). \numberthis \label{eq:d_m1}
\end{align*}

\item $m=s\ \iff\ (s-m)\bD=0$
\begin{equation}
\label{eq:d_m2}
(-1)^nd_m = \sum_{j=0}^{n} (-1)^j\dim \H_j(B;\Q) + \sum_{j=n}^{2n} (-1)^j\dim \H_j(B;\Q) = \chi(B) + (-1)^n\dim \H_n(B;\Q).
\end{equation}

\item $m>s\ \iff\ (s-m)\bD<0$
\begin{align*}
(-1)^nd_m & = \sum_{j=0}^{n+(s-m)\bD} (-1)^j\dim \H_j(B;\Q) + \sum_{j=n-(s-m)\bD}^{2n} (-1)^j\dim \H_j(B;\Q) \\
& = \chi(B) - \sum_{j=n+(s-m)\bD}^{n-(s-m)\bD}  (-1)^j\dim\H_j(B;\Q) + (-1)^n\dim \H_{n+(s-m)\bD}(B;\Q) + \\
& + (-1)^n\dim\H_{n-(s-m)\bD}(B;\Q) \\
& = \chi(B) - \sum_{j=n+(s-m)\bD}^{n-(s-m)\bD}  (-1)^j\dim \H_j(B;\Q) + (-1)^n2\dim \H_{n+(s-m)\bD}(B;\Q), \numberthis \label{eq:d_m3}
\end{align*}
where the last equality holds by Poincar\'e duality.
\end{itemize}
Clearly,
\begin{equation}
\label{eq:d_m4}
\sum_{m=1}^{s-1} \sum_{j=n-(s-m)\bD}^{n+(s-m)\bD} (-1)^j\dim \H_j(B;\Q) = \sum_{m=s+1}^{2s-1} \sum_{j=n+(s-m)\bD}^{n-(s-m)\bD} (-1)^j\dim \H_j(B;\Q).
\end{equation}
Hence, \eqref{eq:claim} follows from \eqref{eq:d_m1}, \eqref{eq:d_m2}, \eqref{eq:d_m3} and \eqref{eq:d_m4}.
\end{proof}

Take an integer $p\geq 0$. Note that taking $s$ big enough such that $-s\bD < -n-p-1$ we have from \eqref{eq:ESHpreq} the equalities
\begin{equation}
\label{eq:b_0}
b_i = \sum_{m=s+1}^{2s-1} \dim \H_{n+(s-m)\bD+i}(B;\Q),
\end{equation}
for every $i \in \{0,\pm p\}$;
\begin{align*}
\label{eq:b_d}
b_d & = \dim \bigoplus_{m=1}^\infty \H_{n+(s-m)\bD}(B;\Q) \\
& = \sum_{m=s+1}^{2s-1} \dim \H_{n+(s-m)\bD}(B;\Q) + \sum_{m=1}^{s-1} \dim \H_{n+(s-m)\bD}(B;\Q) + \dim \H_n(B;\Q) \\
& = \sum_{m=s+1}^{2s-1} \dim \H_{n+(s-m)\bD}(B;\Q) + \sum_{m=s+1}^{2s-1} \dim \H_{n+(s-m)\bD}(B;\Q) + \dim \H_n(B;\Q) \\
& = 2b_0 + \dim \H_n(B;\Q); \numberthis
\end{align*}
\begin{align*}
\label{eq:b_d+p}
b_{d+p} & = \dim \bigoplus_{m=1}^\infty \H_{n+(s-m)\bD+p}(B;\Q) \\
& = \sum_{m=s+1}^{2s-1} \dim \H_{n+(s-m)\bD+p}(B;\Q) + \sum_{m=1}^{s-1} \dim \H_{n+(s-m)\bD+p}(B;\Q) + \dim \H_{n+p}(B;\Q) \\
& = \sum_{m=s+1}^{2s-1} \dim \H_{n+(s-m)\bD+p}(B;\Q) + \sum_{m=s+1}^{2s-1} \dim \H_{n+(s-m)\bD-p}(B;\Q) + \dim \H_{n+p}(B;\Q) \\
& = b_p + b_{-p} + \dim \H_{n+p}(B;\Q); \numberthis
\end{align*}
where, in \eqref{eq:b_d} and \eqref{eq:b_d+p}, the third equation holds by Poincar\'e duality and the fourth equality follows from \eqref{eq:b_0}.

Assume, without loss of generality, that $d>P$. Now, we will split the argument according to the parity of $n$.

\vskip .3cm
\noindent
{\bf Case 1. $n$ is odd}
\vskip .2cm

Let us compute $r_M$ for a given even $p\geq 0$. Since $n$ is odd, from \eqref{eq:ESHpreq},
\[
\chi_+(M) = -\frac{\chi(B)}{\bD}\ \implies\ d\chi_+(M)=-s\chi(B).
\]
Hence, for $p=0$,
\begin{align*}
r_M & = 2(-q\chi_+(M) + \sum_{m=-\infty}^q (-1)^mb_m) \\
& = 2(-d\chi_+(M) + \sum_{m=-\infty}^d (-1)^mb_m) \\
& = 2s\chi(B) - 2s\chi(B) + \chi(B) + 2b_0 + \dim \H_n(B;\Q) \\
& = \chi(B) + 2b_0 + \dim \H_n(B;\Q) \\
& = \chi(B) + \dim \H_n(B;\Q) + 2\sum_{m=1}^{\lfloor n/\bD \rfloor} \dim \H_{n-m\bD}(B;\Q),
\end{align*}
where the second equality follows from Proposition \ref{prop:periodicity} and the periodicity of $\HC_*(M)$, the third holds by Lemma \ref{lemma:sumb_m} and \eqref{eq:b_0} and the last is a consequence of \eqref{eq:ESHpreq}.

Similarly, when $p\geq 2$,
\begin{align*}
r_M & = 2(-q\chi_+(M) + \sum_{m=-\infty}^q (-1)^mb_m) + 2b_{q+p} - b_q \\
& = 2(-d\chi_+(M) + \sum_{m=-\infty}^d (-1)^mb_m) + 2b_{d+p} - b_d \\
& = 2s\chi(B) - 2s\chi(B) + \chi(B) + 2b_0 + \dim \H_n(B;\Q) + 2b_{d+p} - b_d \\
& = \chi(B) + 2b_0 + \dim \H_n(B;\Q) + 2b_p + 2b_{-p} + 2\dim\H_{n+p}(B;\Q) - 2b_0 - \dim \H_n(B;\Q) \\
& = \chi(B) + 2b_p + 2b_{-p} + 2\dim\H_{n+p}(B;\Q) \\
& = \chi(B) + 2\dim \H_{n+p}(B;\Q) + 2\sum_{m=1}^{\lfloor (n+p)/\bD \rfloor} \dim \H_{n-m\bD+p}(B;\Q) \\
& + 2\sum_{m=1}^{\lfloor (n-p)/\bD \rfloor} \dim \H_{n-m\bD-p}(B;\Q),
\end{align*}
where the second equality follows from Proposition \ref{prop:periodicity} and the periodicity of $\HC_*(M)$, the third is a consequence of Lemma \ref{lemma:sumb_m} and \eqref{eq:b_0}, the fourth follows from \eqref{eq:b_d} and \eqref{eq:b_d+p}, and the last holds by \eqref{eq:ESHpreq}.

\vskip .3cm
\noindent
{\bf Case 2. $n$ is even}
\vskip .2cm

The argument is similar to the previous case. Take an odd $p\geq 1$. Since $n$ is even, from \eqref{eq:ESHpreq},
\[
\chi_+(M) = \frac{\chi(B)}{\bD}\ \implies\ d\chi_+(M)=s\chi(B).
\]
Therefore,
\begin{align*}
r_M & = 2(-q\chi_+(M) + \sum_{m=-\infty}^q (-1)^mb_m) + 2b_{q+p} + b_q \\
& = 2(d\chi_+(M) - \sum_{m=-\infty}^{d} (-1)^mb_m) + 2b_{d+p} + b_d \\
& = 2s\chi(B) - 2s\chi(B) + \chi(B) - 2b_0 - \dim \H_n(B;\Q) + 2b_{d+p} + b_d \\
& = \chi(B) - 2b_0 - \dim \H_n(B;\Q) + 2b_p + 2b_{-p} + 2\dim\H_{n+p}(B;\Q) + 2b_0 + \dim\H_n(B;\Q) \\
& = \chi(B) + 2b_p + 2b_{-p} + 2\dim\H_{n+p}(B;\Q) \\
& = \chi(B) + 2\dim \H_{n+p}(B;\Q) + 2\sum_{m=1}^{\lfloor (n-p)/\bD \rfloor} \dim \H_{n-m\bD-p}(B;\Q) \\
& + 2\sum_{m=1}^{\lfloor (n+p)/\bD \rfloor} \dim \H_{n-m\bD+p}(B;\Q),
\end{align*}
where, as before, the second equality follows from Proposition \ref{prop:periodicity} and the periodicity of $\HC_*(M)$, the third holds by Lemma \ref{lemma:sumb_m} and \eqref{eq:b_0}, the fourth follows from  \eqref{eq:b_d} and \eqref{eq:b_d+p}, and the last one goes from \eqref{eq:ESHpreq}. This completes the proof in case \ref{PM}.

\begin{remark}
In \cite[Theorem 2.1]{GGM2}, a lower bound for the number of periodic orbits of non-degenerate contact forms $\alpha$ on suitable prequantizations of symplectic manifolds $B$ was obtained, namely
\begin{equation}
\label{eq:r_M old}
r_B :=
\begin{cases}
\chi(B) + 2\dim H_n(B;\Q) & \text{if } n \text{ is odd}, \\
\chi(B) + 4\dim H_{n+1}(B;\Q) & \text{if } n \text{ is even},
\end{cases}
\end{equation}
under the assumption that $\alpha$ has no periodic orbit $\ga$ with $\cz(\ga)=0$ (resp.\ $\cz(\ga)\in\{0,\pm1\}$) when $n$ is odd (resp. even).

Under these assumptions, if $n$ is odd, we have
\begin{align*}
b_0 = \sum_{m=1}^{\lfloor n/\bD \rfloor} \dim \H_{n-m\bD}(B;\Q) = 0,
\end{align*}
which implies that our bound $r_M$ for $p=0$ is given by
\begin{equation}
\label{eq:r_M new-odd}
r_M = \chi(B) + \dim H_n(B;\Q).
\end{equation}
Similarly, if $n$ is even, then
\begin{align*}
b_{-1} &= \sum_{m=1}^{\lfloor (n-1)/\bD \rfloor} \dim \H_{n-m\bD-1}(B;\Q) = 0, \\
b_{1}  &= \sum_{m=1}^{\lfloor (n+1)/\bD \rfloor} \dim \H_{n-m\bD+1}(B;\Q) = 0,
\end{align*}
which implies that our bound $r_M$ for $p=1$ is
\begin{equation}
\label{eq:r_M new-even}
r_M = \chi(B) + 2\dim H_{n+1}(B;\Q).
\end{equation}

Observe that the bounds \eqref{eq:r_M old}, \eqref{eq:r_M new-odd}, and \eqref{eq:r_M new-even} coincide when $\H_{\mathrm{odd}}(B;\Q)=0$, but differ in general. This discrepancy arises from the fact that the bound $r_M$ in the present work is defined so that, if $\alpha$ has more than $r_M$ periodic orbits, then it cannot be lacunary.

However, we are not aware of any example of a prequantization of a symplectic manifold $B$ admitting Reeb flows with finitely many closed orbits for which the condition $\H_{\mathrm{odd}}(B;\Q)=0$ fails.
\end{remark}

\vskip .3cm
\noindent
{\bf Case \ref{T}}
\vskip .2cm

It is well known that good toric contact manifolds $M$ are prequantizations of orbifolds and admit (toric) contact forms $\alpha$, which are non-degenerate and lacunary, such that $\#\P_\alpha = \dim \H_*(M/S^1;\Q)$ \cite{AM0,AMM1}. If $M$ satisfies condition \ref{cond:F} or \ref{cond:NF}, one can define $\HC_*(M)$ and it is also known that  $\HC_*(M)$ is periodic, there exists an integer $D$ such that $\HC_j(M)=0$ for every $j < D$, $b_j < \infty$ for every $j \in \Z$ and $\HC_*(M)$ is lacunary with the same parity of $n$ \cite{AM0,AMM1,AMM2}. Then, if $M$ satisfies condition \ref{cond:F} or \ref{cond:NF} the result follows from Corollary \ref{cor:r_M lacunary}. 

When $M$ does not satisfy conditions \ref{cond:F} and \ref{cond:NF}, one can define $\HC_*(M)$ combinatorially and it also satisfies all the properties listed above \cite{AMM2}. This definition coincides with the one given in Section \ref{sec:ESH} whenever $M$ meets \ref{cond:F} or \ref{cond:NF}. Therefore we can define $r_M$ as in Theorem \ref{thm:main} and we show below that in this case still holds the equality $r_M= \dim \H_*(M/S^1;\Q)$. 

In what follows, we will omit several details and definitions and refer the reader to \cite{AMM2}. Let $D\subset\R^n$ be the toric diagram of a Gorenstein toric contact manifold $(M^{2n+1}, \xi)$ and $\delta_k \in \N_0$, $k=0,\ldots,n$, the non-negative integers determined by the Ehrhart polynomial of $D$. We have that the SFT contact Betti numbers of $(M^{2n+1}, \xi)$ are given by
\[
cb_{0} = \delta_n\,,\ cb_{2} = \delta_n + \delta_{n-1}\,,\ cb_4 = \delta_n + \delta_{n-1} + \delta_{n-2}\,,\ 
\ldots\,,\ cb_{2n} = \delta_n + \delta_{n-1} + \cdots + \delta_{0}\,,
\]
\[
cb_{2k} = cb_{2n}\,,\ \forall k>n\,,\ \text{and all other SFT contact Betti numbers are zero.}
\]
In particular,
\[
\sum_{k=0}^n cb_{2k} = \sum_{k=0}^n (k+1) \delta_k\,.
\]
Here the SFT contact Betti number is the contact Betti number used throughout this paper but with the SFT degree. This means that their relation to the $b_j$'s in the grading used here is
\[
cb_{j} = b_{j-n+2} \quad\text{or equivalently}\quad b_j = cb_{j+n-2}\,,\ \forall j\in\Z\,.
\]
In this Gorenstein toric context, we can choose $P= 2n -n +2 = n+2$, $\Delta = 2$ and we have that
\[
\chi_+(M) = (-1)^n \frac{\delta_n + \delta_{n-1} + \cdots + \delta_{0}}{2}\,.
\]

When $n$ is even we can take $q= n+2$ and therefore
\begin{align}
r_M & = 2 \left(q \chi (M) - \sum_{j=-\infty}^q b_j \right) + b_q \notag \\
& = 2 \left((n+2) \frac{\delta_n + \delta_{n-1} + \cdots + \delta_{0}}{2} - \sum_{k=0}^n cb_{2k}\right) + cb_{2n}\notag \\
& = (n+3) \sum_{k=0}^n \delta_k - 2 \sum_{k=0}^n (k+1) \delta_k = \sum_{k=0}^n (n+1-2k) \delta_k \notag \\
& = \# \ \text{of facets of $D$.} \notag
\end{align}

When $n$ is odd we can take $q= \lcm\{P,\Delta\} = \lcm\{n+2,2\} = 2n+4$ and consequently
\begin{align}
r_M & = 2 \left(- q \chi (M) - \sum_{j=-\infty}^q b_j \right) \notag \\
& = 2 \left(2(n+2) \frac{\delta_n + \delta_{n-1} + \cdots + \delta_{0}}{2} - \sum_{k=0}^{3n+2} cb_{k}\right) \notag \\
& = 2 \left((n+2) (\delta_n + \delta_{n-1} + \cdots + \delta_{0}) - \sum_{k=0}^{n} cb_{2k} - \frac{n+1}{2} cb_{2n}\right) \notag \\
& = 2 \left(\frac{n+3}{2} (\delta_n + \delta_{n-1} + \cdots + \delta_{0}) - \sum_{k=0}^{n} cb_{2k} \right) \notag \\
& = (n+3) \sum_{k=0}^n \delta_k - 2 \sum_{k=0}^n (k+1) \delta_k = \sum_{k=0}^n (n+1-2k) \delta_k \notag \\
& = \# \ \text{of facets of $D$,} \notag
\end{align}
proving the assertion in \ref{T} concerning the computation of $r_M$.

To prove the last assertion in \ref{T} we proceed as follows. The number of simple periodic Reeb orbits of a non-degenerate toric contact form on a good toric contact manifold of dimension $2n+1$ is given by the number of edges of its moment cone $C\subset \R^{n+1}$. This is a strictly convex rational good polyhedral cone of dimension $n+1$ and so, in particular, the primitive integral  vectors that generate its edges span the whole $\R^{n+1}$, which implies that $C$  has at least $n+1$ edges.

If $C\subset \R^{n+1}$ has exactly $n+1$ edges, then it also has exactly $n+1$ facets,  each of them being determined by $n$ of the $n+1$ edges, i.e.
\[
| \text{facets} | = \binom{n+1}{n} = n+1\,.
\]
Let $\nu_1,\ldots\nu_{n+1} \in \Z^{n+1}$ be the primitive interior integral normals to those facets and consider the map $\beta : \T^{n+1} \to \T^{n+1} $ represented by the matrix $[\nu_1 | \cdots | \nu_{n+1}]$. Then $\ker (\beta)$ is isomorphic to the finite cyclic group $\Z_p$, with $p = |\det ([\nu_1 | \cdots | \nu_{n+1}])|$, and the toric contact manifold determined by the cone $C\subset \R^{n+1}$ is the lens space
\[
S^{2n+1} / \Z_p
\]
where $\Z_p$ acts linearly on $S^{2n+1} \subset \C^{n+1}$ as the subgroup $\ker (\beta) \subset \T^{n+1}$.

\vskip .3cm
\noindent
{\bf Case \ref{PM} $\cap$ \ref{T} when $[\omega]=c_1(TB)$}
\vskip .2cm

Following the argument in case \ref{T}, we give a quick proof of case \ref{PM} in the toric case (i.e. a prequantization of a toric symplectic manifold) when $[\omega]=c_1(TB)$, using the combinatorial description of $\HC_*(M)$ described in \cite{AMM2}. Note that taking into account \cite[Remark 7.12]{AMM2}, when $M$ is the prequantization of a smooth basis $B$ with $r=1$, i.e. when the base polytope is reflexive or equivalently when $[\omega] = c_1(TB)$, we have that 
\[
\delta_k =  \dim \H^{2k} (B; \Q) = \dim \H^{2(n-k)} (B; \Q) = \delta_{n-k}
\]
and
\begin{align}
r_M & = \sum_{k=0}^n (n+1-2k) \delta_k = \frac{1}{2} \left( \sum_{k=0}^n (n+1-2k) \delta_k + \sum_{k=0}^n (2k+1-n) \delta_{n-k} \right) \notag \\
 & = \frac{1}{2} \left( \sum_{k=0}^n 2 \delta_k \right) = \sum_{k=0}^n  \dim \H^{2k} (B; \Q) = \dim \H_*(B;\Q) \,. \notag
\end{align}

\vskip .3cm
\noindent
{\bf Case \ref{U}}
\vskip .2cm

It is an immediate consequence of the fact that Ustilovsky spheres $M$ of dimension $2n+1$ admit non-degenerate lacunary contact forms with precisely $n+1$ closed orbits and the fact $M$ satisfies the hypotheses of Theorem \ref{thm:main} \cite{Ust}.

\vskip .3cm
\noindent
{\bf Case \ref{PO}}
\vskip .2cm

In what follows, we will use the background about symplectic orbifolds presented in \cite{LT}. Here, we define a smooth prequantization of a symplectic orbifold $(B^{2n},\om)$ as a contact manifold $(M^{2n+1},\xi)$ with a contact form $\beta$ such that $\ker \beta=\xi$, the Reeb flow of $\beta$ generates a locally free circle action such that $M/S^1=B$ and $d\beta = \pi^*\om$, where the quotient projection $\pi: M \to B$ is viewed as a smooth orbifold map.

The definition of Hamiltonian vector fields works verbatim to orbifolds and given a Hamiltonian $H: B \to \R$ that generates a circle action with isolated fixed points (given by the critical points of $H$) we have that $H$ is Morse and every critical point has even index \cite[Lemma 5.3]{LT}. Using the Morse theory for orbifolds developed in \cite{LT} (whose construction, as already mentioned in \cite{LT}, is a special case of Morse theory on stratified spaces \cite{GMac}) we can conclude that
\[
\dim \H_k(B;\Q) = \#\text{Crit}_k(H),
\]
where $\text{Crit}_k(H)$ is the set of critical points of $H$ with index $k$. It is a consequence of the relation
\[
M(t) - P(t)=(1+t)Q(t)
\]
proved in \cite[Theorem 4.7]{LT}, where $M$ and $P$ are the Morse and Poincar\'e polynomials respectively and $Q$ is a polynomial with nonnegative coefficients. It is standard in the case that $B$ is a manifold. However, there is a nuance when $B$ is an orbifold: in this case, the Morse polynomial for a Morse function $H$ is given by
\[
M(t) = \sum_{p \in \text{Crit}(H)}\sum_{k=0}^{2n} \dim \H_k(D_p,S_p;\Q)t^k,
\]
where $D_p$ is the disk in the negative subspace $V^-_p$ of the Hessian of $H$ at $p$ and $S_p$ is the corresponding sphere given by the boundary. Let $\Gamma$ be the orbifold group of $p$. If $\Gamma=0$ then, as in the smooth case, $\dim \H_k(D_p,S_p;\Q)$ is equal to 1 if $k=\index_p(H)$, where $\index_p(H)$ is the Morse index of $H$ at $p$, and zero otherwise. As explained in \cite{LT}, if $\Gamma\neq 0$ then $\dim \H_k(D_p/\Gamma,S_p/\Gamma;\Q)=0$ if $k\neq \index_p(H)$; whereas if $k = \index_p(H)$ then $\dim \H_k(D_p\Gamma,S_p\Gamma;\Q)=1$ if $\Gamma$ preserves the orientation of $D_p$ and is zero otherwise. However, since the Hamiltonian flow of $H$ generates an $S^1$-action, one can take a compatible almost complex structure on $M$ invariant by this action and in this way we can conclude that $V^-_p$ is a symplectic subspace. On the other hand, the linearized action of $\Gamma$ is symplectic and leaves $V^-_p$ invariant. Consequently, it preserves the orientation of $V^-_p$ (cf. \cite[Remark 5.4]{LT}).

Thus, the number of critical points of $H$ is equal to the total rank of the rational homology of $B$. Therefore, our desired result is a consequence of the following proposition.

\begin{proposition}
Assume that $(B,\om)$ admits a Hamiltonian circle action with isolated fixed points generated by $\bH: B \to \R$. There exists a non-degenerate lacunary contact form $\alpha$ on $M$, with parity $n$ (mod 2), such that the image of the simple orbits of $\alpha$ are the Reeb orbits of $\beta$ that project to the critical points of $\bH$. Moreover, $\HC_*(M)$, given by the chain complex associated to $\alpha$, is periodic.
\end{proposition}

\begin{proof}
The proof proceeds similarly as in the case that $B$ is smooth with some nuances. We will use several ideas from \cite{AGZ} and \cite{Bo1,Bo2}; cf. \cite{AM5} in the smooth case. Let $R_\beta$ be the Reeb vector field of $\beta$. Fixed the contact form $\beta$, it is well known that given a positive smooth function $G: M \to \R$ (called a contact Hamiltonian) we have a contact vector field $X_G$ uniquely defined by the equations $\beta(X_G)=G$ and $i_{X_G}d\beta = dG(R_\beta)\beta - dG$. Now, suppose that $G$ is invariant along the orbits of $R_\beta$ so that $G=\pi^*\bG$, where $\bG: B \to \R$ is a smooth function. Then the equation $i_{X_G}d\beta = dG(R_\beta)\beta - dG$ reads as $i_{X_G}d\beta = -dG$. We claim that $X_G$ is the Reeb flow of $\alpha_G:=\beta/G$. As a matter of fact, $\alpha_G(X_G)=\beta(X_G)/G=1$ and
\[
i_{X_G}d\alpha_G=i_{X_G}\bigg(\frac{d\beta}{G} - \beta\wedge d\bigg(\frac{1}{G}\bigg)\bigg) = -\frac{dG}{G} + \frac{G dG}{G^2} = 0.
\]
We have that $X_G=GR_\beta + X^h_\bG$, where $X^h_\bG$ is the horizontal lift of the Hamiltonian vector field $X_\bG$ to $M$ (note that these constructions make sense in the orbifold setting). Indeed,
\[
\beta(GR_\beta + X^h_\bG) = \beta(GR_\beta) = G
\]
and
\[
i_{GR_\beta + X^h_\bG}d\beta = i_{X^h_\bG}d\beta = -dG
\]
which are precisely the equations that define $X_G$ (note that the last equation follows from the fact that $d\beta = \pi^*\om$ and that $i_{X_\bG}\om=-d\bG$). 

The vector fields $R_\beta$ and $X_G$ commute. In fact,
\[
\beta([R_\beta,X_G]) = L_{R_\beta}(\beta(X_G)) - (L_{R_\beta}\beta)(X_G) = L_{R_\beta}(\beta(X_G)) = L_{R_\beta}G = dG(R_\beta) = 0,
\]
which implies that $[R_\beta,X_G]$ is horizontal. On the other hand,
\[
d\pi([R_\beta,X_G]) = [d\pi(R_\beta),d\pi(X_G)] = 0.
\]

Now, let $\bH$ be the Hamiltonian that generates the circle action on $B$ as in the statement of the proposition. Assume, without loss of generality, that $\bH$ is a positive function and that the least common period of the orbits of $R_\beta$ and $X_H$ coincide and are equal to 1. Let $H: M \to \R$ be the lift of $\bH$, that is, $H=\pi^*\bH$. Since $H$ is positive we clearly have a smooth function $f: M \to (-1,0)$, which is also invariant along the orbits of $R_\beta$, such that we can write
\[
H = -\frac{f}{1+f}.
\]
Let $X_H$ be the contact vector field associated to $H$ as above. From our previous discussion,
\[
X_H = HR_\beta + X^h_\bH\quad\text{and}\quad X_H = R_{\beta/H}.
\]
We claim that
\[
R_{(1+f)\beta} = R_\beta + X_H.
\]
As we have seen, $R_{(1+f)\beta}=Y$, where $Y$ is the contact vector field associated to $F:=1/(1+f)$ uniquely characterized by the equations (note that $dF(R_\beta)=0$ since $df(R_\beta)=0$)
\[
\beta(Y)=F\quad\text{and}\quad i_Yd\beta=-dF.
\]
But
\begin{align*}
\beta(R_\beta + X_H) & = \beta(R_\beta) + \beta(X_H) \\
& = 1 - \frac{f}{1+f} = \frac{1+f}{1+f} - \frac{f}{1+f}= \frac{1}{1+f} \\
& = F
\end{align*}
and
\begin{align*}
i_{R_\beta + X_H}d\beta & = i_{X_H}d\beta \\
& = -dH \\
& = -dF,
\end{align*}
where the last equality follows from the fact that $F-H$ is constant equal to 1. This proves the claim.

Now, let $c$ be a positive real number. We can easily see that
\begin{equation}
\label{eq:f_c}
cH = -\frac{f_c}{1+f_c},
\end{equation}
where
\[
f_c = \frac{cf}{1+(1-c)f}.
\]
Using the fact that $X_H=HR_\beta + X^h_\bH$, it is easy to see that the argument in the proof of \cite[Proposition 3.6]{AGZ} works verbatim to show that $X_H$ generates a circle action with period 1.  Thus, since $X_{cH}=cHR_\beta + cX^h_\bH=cX_H$ we have that it generates a circle action with minimal common period $1/c$. Hence $R_\beta$ generates a circle action with minimal common period 1, $X_{cH}$ generates a circle action with minimal common period $1/c$ and $[R_\beta,X_{cH}]=0$. Moreover, these two vector fields are linearly independent except precisely at the fibres (i.e. the orbits of $R_\beta$) over the critical points of $\bH$.

Recall that
\[
R_{(1+f_c)\beta} = R_\beta + X_{cH}.
\]
We have that if $c$ is irrational then the periodic orbits $\ga_{c,p}$ of $(1+f_c)\beta$ are precisely the fibers over the critical points $p$ of $H$ (cf. \cite[Proposition 3.9]{AGZ}) and that all these orbits are elliptic. Indeed, since $R_{(1+f_c)\beta} = R_\beta + X_{cH}.$ and the vector fields $R_\beta$  and $X_{cH}$ commute, it is easy to see that every closed orbit of $R_{(1+f_c)\beta} $ is elliptic because every singularity of $X_H$ is elliptic. Therefore, $(1+f_c)\beta$ is lacunary with parity equal to $n\,(\bmod\,2)$ because every elliptic closed orbit has index equal to $n\,(\bmod\,2)$. It is also easy to see that, choosing $c$ properly, $(1+f_c)\beta$ is non-degenerate as well (because one can choose $c$ such that, for every critical point $p$ of $\bH$, all eigenvalues of $d\phi_p^c(p)$ are of the form $e^{2\pi i\theta}$ with $\theta \in \R\setminus\Q$, where $\phi_p^c$ is the time $c/(c\bH(p)+1)$ map of the flow of $X_{\bH}$). It proves the first assertion of the proposition.

In order to prove the second assertion, we have to compute the indices of the iterates of $\ga_{c,p}$ and for this we will use \cite{Bo1,Bo2}. First, note that the critical points of $c\bH$ and $\barf_c$ coincide since, by \eqref{eq:f_c},
\[
d(c\bH) = -\frac{d\barf_c}{(1+\barf_c)^2}.
\]
The Hessians at a critical point $p$ are related by
\begin{equation}
\label{eq:hessians}
\Hess_p c\bH = -\frac{1}{(1+\barf_c(p))^2} \Hess_p \barf_c
\end{equation}
and consequently $H$ is Morse if, and only if, so is $f_c$ for every $c$.

Since the Reeb flow generates a locally free circle action then, as a consequence of the slice theorem, it is Morse-Bott in the sense of \cite{Bo1,Bo2}. Given $T \in \R$, let $N_T$ be the submanifold of $M$ given by the union of the closed orbits of $\beta$ with (not necessarily minimal) period $T$. Let $S_T=\pi(N_T)$ be the corresponding orbit space in $B$. Note that $B=S_1$ and that each $S_T$ is an orbifold having as singularities the orbit spaces $S_{T'}$ such that $T'$ divides $T$. Note that, given a smooth function $f: B \to \R$, since it is a smooth orbifold map, for each critical point $p \in S_T$ we have that the Hessian of $f$ at $p$ leaves $T_pS_T$ invariant. Therefore, if$f$ is Morse on $B$, its restriction $f|_{S_T}$ is Morse as well for every $T$.

Let $p$ be a critical point of $\bH$ and $\psi_{p}$ be the simple closed orbit of $R_\beta$ whose image is the fiber over $p$. Let $T_p$ be the period of $\psi_p$ so that $p \in S_{T_p}$. By the arguments from \cite[Lemma 2.4]{Bo1} we can conclude that, given $k_0 \in \N$, one can choose $c$ sufficiently small such that
\[
\cz(\ga^k_{c,p}) = \rs(\psi^k_{p}) - \frac 12 \dim S_{kT_p} + \index_p(\barf_c|_{S_{kT_p}}),
\]
for every $1\leq k \leq k_0$, where $\rs(\psi^k_{p})$ is the Robbin-Salamon index of  the $k$-th iterate of $\psi_{p}$, $\barf_c: B \to \R$ is the smooth function such that $f_c=\pi^*\barf_c$ and $\index_p(\barf_c|_{S_{T_p}})$ is the Morse index of the restriction $\barf_c|_{S_{T_p}}$ at $p$. Hence, by \eqref{eq:hessians},
\begin{equation}
\label{eq:index ga_{c,p}}
\cz(\ga^k_{c,p}) = \rs(\psi^k_{p}) + \frac 12 \dim S_{kT_p} - \index_p(\bH|_{S_{kT_p}}).
\end{equation}

Since the lowest common period of the orbits of $\beta$ is equal to 1, we can write $T_p=1/m_p$ for some integer $m_p\geq 1$. Let $T_{c,p}=T_p-c\bH(p)/(m_p+c\bH(p)m_p)$ be the action of $\ga_{c,p}$. The hypothesis that $\beta$ has index-positivity implies that given $L>0$ there exists $T>0$ such that for every sufficiently small non-degenerate perturbation $\beta'$ of $\beta$ we have that every periodic orbit of $\beta'$ with action bigger than $T$ has index bigger than $L$. Thus, given $T>0$ there exist $\delta>0$ and $\ep>0$ small such that $T+\ep$ is outside the action spectrum of $\beta$ and if $c \in (0,\delta)$ is irrational we have
\[
\CC_*^{T+\ep}(\beta) = \bigoplus_{p \in \text{Crit}(\bH)} \bigoplus_{k\in \N;\, kT_{p} \leq T} \Nov[*-(\rs(\psi^k_{p}) + \frac 12 \dim S_{kT_p} - \index_p(\bH|_{S_{kT_p}})],
\]
where the index shift above means that each orbit $\ga^k_{c,p}$ contributes to $\CC_*^{T+\ep}(\beta') \cong \CC_*^{T+\ep}(\beta)$ with a summand $\Nov$ in the degree $\rs(\psi^k_{p}) + \frac 12 \dim S_{kT_p} - \index_p(\bH|_{S_{kT_p}})$ due to \eqref{eq:index ga_{c,p}}.

Now, the key point is that
\[
\rs(\psi^k_{p}) + \frac 12 \dim S_{kT_p} - \index_p(\bH|_{S_{kT_p}}) = \cz(\ga^k_{c,p}) = n\,(\bmod\,2)
\]
because $\ga^k_{c,p}$ is elliptic and therefore the differential in $\CC_*^{T+\ep}(\beta)$ vanishes. Consequently,
\[
\HC_*^{T+\ep}(\beta) \cong \bigoplus_{p \in \text{Crit}(\bH)} \bigoplus_{k\in \N;\, kT_{p} \leq T} \Nov[*-(\rs(\psi^k_{p}) + \frac 12 \dim S_{kT_p} - \index_p(\bH|_{S_{kT_p}}))].
\]

\begin{remark}
Note that, since $\ga^k_{c,p}$ is a small perturbation of $\psi^k_p$, we have that $\ga^k_{c,p}$ is contractible if and only if so is $\psi^k_p$. Therefore, if $\rs(\psi^k_p)\geq 4$ for every contractible $\psi^k_p$ we have that
\[
\cz(\ga^k_{c,p}) = \rs(\psi^k_{p}) + \frac 12 \dim S_{kT_p} - \index_p(\bH|_{S_{kT_p}}) \geq \rs(\psi^k_{p}) - \frac 12 \dim S_{kT_p} > 3-n
\]
and therefore, by index-positivity, $(1+f_c)\beta$ is index-admissible for $c$ sufficiently small.
\end{remark}

Given $p \in \text{Crit}(\bH)$, let $\ci=\rs(\psi^{m_p}_p)$ and note that $\ci$ is an even positive integer that does not depend on the choice of $p$ that is equal to the Robbin-Salamon index of the generic orbit (i.e. any simple orbit with period 1) of $\beta$. Note that, given $q \in \N$ which is a multiple of $m_p$, $\rs(\psi^{k+q}_p) = \rs(\psi^k_p) + (q/m_p)\ci$. Moreover, clearly $S_{(k+q)T_p}=S_{kT_p}$. Hence, given $T \in \N$ there exist $\delta>0$ and $\ep>0$ small such that $T+\ep$ is outside the action spectrum of $\beta$ and if $c \in (0,\delta)$ is irrational we have
\[
\HC_*^{T+\ep}(\beta) \cong \bigoplus_{p \in \text{Crit}(\bH)} \bigoplus_{m=1}^T \bigoplus_{k=1}^{m_p} \Nov[*-((m-1)\ci + \rs(\psi^k_{p}) + \frac 12 \dim S_{kT_p} - \index_p(\bH|_{S_{kT_p}}))].
\]
Using this, and the index-positivity of $\beta$, we arrive at
\begin{equation}
\label{eq:HCorb}
\HC_*(M) \cong \bigoplus_{p \in \text{Crit}(\bH)} \bigoplus_{m=1}^\infty \bigoplus_{k=1}^{m_p} \Nov[*-((m-1)\ci + \rs(\psi^k_{p}) + \frac 12 \dim S_{kT_p} - \index_p(\bH|_{S_{kT_p}}))],
\end{equation}
which implies clearly that $\HC_*(M)$ is periodic.
\end{proof}

In what follows, we use the above computation and notation to prove the following result, which is a particular consequence of Proposition \ref{prop:PD finite}. As in Proposition \ref{prop:PD finite}, it shows that $\HC_*(M)$ exhibits a symmetry in sufficiently large degrees;  cf. Remark \ref{rmk:PD}. When $M$ is a prequantization of an orbifold admitting a Hamiltonian circle action, as in case \ref{PO}, we provide an alternative proof. (Note that, conjecturally, every closed contact manifold admitting a non-degenerate lacunary contact form is a prequantization of this type.) This argument is entirely different from that of Proposition \ref{prop:PD finite} and instead follows the spirit of the proof of Proposition \ref{prop:PD preq}. The argument also relies on Poincar\'e duality on $B$, but is more involved than that of Proposition \ref{prop:PD preq}, as it requires handling the orbit spaces $S_{kT_p}$. We recall that $b_j=\dim \HC_j(M)$ denotes the $j$-th contact Betti number.

\begin{proposition}
\label{prop:PD preq_orb}
Let $M$ be a prequantization of an orbifold as in case \ref{PO} of Theorem \ref{thm:r_M}. Given an integer $p\geq 0$ there exists $C>0$ such that if $d=s\ci$ is bigger than $C$ then
\[
b_{d-j} = b_{d+j}
\]
for every $j \in [-p,p]$, where, as above, $\ci$ is the Robbin-Salamon index of the generic orbit of $\beta$.
\end{proposition}

\begin{proof}
Given $p \in \text{Crit}(\bH)$ and $1\leq k\leq m_p$ set
\[
d_{k,p} = \frac 12 \dim S_{kT_p} - \index_p(\bH|_{S_{kT_p}})).
\]
Let
\[
C = \max_{p \in \text{Crit}(\bH)} \max_{1\leq k\leq m_p} (|\rs(\psi^k_{p})| + |d_{k,p}|) + p
\]
and take $s$ big enough such that $d:=s\ci>C$.

From \eqref{eq:HCorb} we have that, for any $j\geq -p$,
\begin{align*}
\label{eq:b_d+j}
b_{d+j} & = \sum_{p \in \text{Crit}(\bH)} \#\{1\leq k\leq m_p;\ m\ci + \rs(\psi^k_{p}) + d_{k,p} = d + j,\ m \in \N_0\} \\
& = \sum_{p \in \text{Crit}(\bH)} \#\{1\leq k\leq m_p;\ (m-s)\ci + \rs(\psi^k_{p}) + d_{k,p} = j,\ m \in \N_0\} \\
& = \sum_{p \in \text{Crit}(\bH)} \#\{1\leq k\leq m_p;\ m\ci + \rs(\psi^k_{p}) + d_{k,p} = j,\ m \geq -s\} \\
& = \sum_{p \in \text{Crit}(\bH)} \#\{1\leq k\leq m_p;\ m\ci + \rs(\psi^k_{p}) + d_{k,p} = j,\ m \in \Z\}, \numberthis
\end{align*}
where the last equality holds because $-(s+i)\ci + \rs(\psi^k_{p}) + d_{k,p} < -p$ for every $i\geq 1$.

Now, note that in \eqref{eq:HCorb} we can clearly replace $\bH$ with $-\bH$ and
\[
\frac 12 \dim S_{kT_p} - \index_p(-\bH|_{S_{kT_p}})) = - \frac 12 \dim S_{kT_p} + \index_p(\bH|_{S_{kT_p}})).
\]
(The replacement of a Morse function $f$ with $-f$ is essentially the way we prove Poincar\'e duality using Morse theory.) Therefore,
\begin{align*}
\label{eq:PD}
b_{d+j} & =  \sum_{p \in \text{Crit}(\bH)} \#\{1\leq k\leq m_p;\ m\ci + \rs(\psi^k_{p}) + d_{k,p} = j,\ m \in \Z\} \\
& = \sum_{p \in \text{Crit}(\bH)} \#\{1\leq k\leq m_p;\ m\ci + \rs(\psi^k_{p}) - d_{k,p} = j,\ m \in \Z\}. \numberthis
\end{align*}

Set $\bk = m_p - k$, $\rs(\psi_p^0)=0$ and $d_{0,p}= -n +  \index_p(\bH)$. The following lemmata will be useful for us to prove the proposition.

\begin{lemma}
\label{lemma:index circle}
$\rs(\psi_p^k) + \rs(\psi_p^\bk) = \nu$.
\end{lemma}

\begin{lemma}
\label{lemma:d_k,p}
$d_{k,p} = d_{\bk,p}$.
\end{lemma}

Let us prove the proposition assuming the lemmata. From \eqref{eq:b_d+j},
\begin{align*}
b_{d-j} & = \sum_{p \in \text{Crit}(\bH)} \#\{1\leq k\leq m_p;\ m\ci + \rs(\psi^k_{p}) + d_{k,p} = -j,\ m \in \Z\} \\
& = \sum_{p \in \text{Crit}(\bH)} \#\{0\leq \bk\leq m_p-1;\ m\ci + \rs(\psi^\bk_{p}) + d_{\bk,p} = -j,\ m \in \Z\} \\
& = \sum_{p \in \text{Crit}(\bH)} \#\{1\leq k\leq m_p;\ (m+1)\ci - \rs(\psi^k_{p}) + d_{k,p} = -j,\ m \in \Z\} \\
& = \sum_{p \in \text{Crit}(\bH)} \#\{1\leq k\leq m_p;\ (m+1)\ci - \rs(\psi^k_{p}) - d_{k,p} = -j,\ m \in \Z\} \\
& = \sum_{p \in \text{Crit}(\bH)} \#\{1\leq k\leq m_p;\ -(m+1)\ci - \rs(\psi^k_{p}) - d_{k,p} = -j,\ m \in \Z\} \\
& = \sum_{p \in \text{Crit}(\bH)} \#\{1\leq k\leq m_p;\ (m+1)\ci + \rs(\psi^k_{p}) + d_{k,p} = j,\ m \in \Z\} \\
& = b_{d+j},
\end{align*}
for every $j\in [-p,p]$, where the second equality holds because
\[
\rs(\psi^0_p)=0=\ci=\rs(\psi^{m_p}_p)\ (\bmod\,\ci)
\]
and $d_{0,p}=d_{m_p,p}$, the third equation follows from Lemmas \ref{lemma:index circle} and \ref{lemma:d_k,p},  the fourth hold by \eqref{eq:PD} and the fifth is a simple consequence of the fact that we are summing over all $m\in\Z$. This proves the desired result.

\begin{proof}[Proof of Lemma \ref{lemma:index circle}]
The lemma is obvious when $k=m_p$. So assume that $m_p>1$ and $1\leq k < m_p$. Let $\tPhi: \R \to \Sp(2n)$ be the linearized Reeb flow of $R_\beta$ along $\psi_p$. For simplicity, reparametrize $\tPhi$ taking $\Phi(t)=\tPhi(t/m_p)$. Take $1\leq k < m_p$ and set
\[
\Phi_1(t) = \Phi(t),\quad t \in [0,\bk]
\]
and
\[
\Phi_2(t) = \Phi(t) \circ \Phi(\bk),\quad t \in [0,k].
\]
By the property of the concatenation of the Robbin-Salamon index \cite{RS} we have
\begin{equation}
\label{eq:concatenation}
\ci = \rs(\Phi|_{[0,m_p]}) = \rs(\Phi_2*\Phi_1) = \rs(\Phi_2) + \rs(\Phi_1).
\end{equation}
where $\Phi|_{[0,m_p]}$ means the restriction of $\Phi$ to the time interval $[0,m_p]$.

Note that
\begin{align*}
\Phi_2(t) & = \Phi(t+\bk),\quad t \in [0,k] \\
& = \Phi(t) \circ \Phi(\bk) \\
& = \Phi(t) \circ \Phi(k)^{-1} \\
& = \Phi(t-k) \\
& = \Phi(k-t)^{-1}
\end{align*}
where in the third equation we are using the fact that $\Id = \Phi(m_p) = \Phi(k+\bk) = \Phi(k)\circ \Phi(\bk)$. Since the loop $\Phi(k-t)^{-1}*\Phi(t)^{-1}$, $t\in [0,k]$, is contractible, by the concatenation property,
\[
\rs(\Phi(k-t)^{-1}) = -\rs(\Phi(t)^{-1}) = \rs(\Phi(t)),\quad t \in [0,k],
\]
since $\rs(\Phi(t)^{-1}) = -\rs(\Phi(t))$ \cite{RS}. So
\[
\rs(\Phi_2(t)) = \rs(\Phi(t)),\quad t \in [0,k].
\]
Therefore, since the index does not change under the reparametrization $t \mapsto t/m_p$,
\[
\rs(\psi_p^\bk) = \rs(\Phi_1(t)),\quad t \in [0,\bk]
\]
and
\[
\rs(\psi_p^k) = \rs(\Phi(t)) = \rs(\Phi_2(t)),\quad t \in [0,k].
\]
Hence, we conclude from \eqref{eq:concatenation} the proof of the lemma.
\end{proof}

\begin{proof}[Proof of Lemma \ref{lemma:d_k,p}]
Set $S_0=B$. We claim that
\begin{equation}
\label{eq:S_j}
S_{kT_p} = S_{\bk T_p}
\end{equation}
which obviously implies the lemma. This is obvious when $k=m_p$. So we can assume that $m_p>1$ and $1\leq k < m_p$.

First notice that
\begin{equation}
\label{eq:equal}
S_{kT_p} = S_{(im_p+k)T_p}\quad\forall i\geq 0
\end{equation}
since the common period of the orbits is $1$ and $T_p=1/m_p$, and
\begin{equation}
\label{eq:inclusion}
S_{kT_p} \subset S_{ikT_p}\quad\forall i\geq 1
\end{equation}
since $k$ divides $ik$.

To prove \eqref{eq:S_j} we first show that $S_{\bk T_p} \subset S_{kT_p}$. We have that
\begin{align*}
S_{\bk T_p} & = S_{(m_p-k)T_p} \\
& \subset S_{(m_p-1)(m_p-k)T_p} \\
& = S_{(m_p(m_p-k-1)+k)T_p} \\
& = S_{kT_p},
\end{align*}
where the inclusion follows from \eqref{eq:inclusion} and the last equality holds by \eqref{eq:equal}. To prove the other direction $S_{kT_p} \subset S_{\bk T_p}$ we proceed in a similar fashion:
\begin{align*}
S_{kT_p} & \subset S_{(m_p-1)kT_p} \\
& = S_{(m_pk - k)T_p} \\
& = S_{(m_p(k-1)+m_p-k)T_p} \\
& = S_{(m_p-k)T_p} \\
& = S_{\bk T_p},
\end{align*}
where, as before, the inclusion follows from \eqref{eq:inclusion} and the second to last equality holds by \eqref{eq:equal}.
\end{proof}

\end{proof}

\vskip .3cm
\noindent
{\bf Case \ref{VSH}}
\vskip .2cm

From \cite{BO17} we conclude that
\begin{equation}
\label{eq:ESH zeroSH}
\HC_*(M) \cong \oplus_{k\geq 0} \H_{*+n-2k}(W,\partial W) \cong \oplus_{k\geq 1} \H^{n-*+2k}(W),
\end{equation}
where the last isomorphism follows from Lefschetz duality (note that the dimension of $W$ is $2n+2$ and that first sum is for $k\geq 0$ while the second is for $k\geq 1$). Let $\b_j=\dim H^j(W)$ and, as before, $b_j = \dim \HC_j(M)$.

By Lefschetz duality, $\b_j=0$ for every $j\geq 2n+2$. Hence, from \eqref{eq:ESH zeroSH},  $b_j=0$ for every $j<-n+1$ and $b_{-n+1} = \b_{2n+1}$, $b_{-n+2} = \b_{2n}$, $b_{-n+3} = \b_{2n-1} + \b_{2n+1}$, $b_{-n+4} = \b_{2n-2} + \b_{2n}$,...,$b_{n+1}=\sum_{j\in 2\N-1} \b_j$, $b_{n+2}=\sum_{j\in 2\N_0} \b_j$ with $b_{n+1}=b_{n+1+2k}$ and $b_{n+2}=b_{n+2+2k}$ for every $k \in \N$.

Thus, $\HC_*(M)$ is periodic with $P=n+1$ and $\Delta=2$. From \eqref{eq:ESH zeroSH} we also have that
\begin{equation}
\label{eq:mec zeroSH}
\chi_+(M) = (-1)^n\frac{\chi(W)}{2}
\end{equation}

Now, let us split the argument according to the parity of $n$.

\vskip .3cm
\noindent
{\bf Case 1. $n$ is odd}
\vskip .2cm

In this case, we can take $q=n+1$ and, from the previous computations of $b_j$ and $\chi_+(M)$, we have that $r_M$ for $p=0$ is given by 
\begin{align*}
& 2(-q\chi_+(M) + \sum_{j=-\infty}^q (-1)^j b_j) \\
& = 2(-(n+1)\chi_+(M) + \sum_{j=-\infty}^{n+1} (-1)^j b_j) \\
& = (n+1)\chi(W) + 2\sum_{j=-n+1}^{n+1} (-1)^j b_j  \\
& = (n+1)\chi(W) + 2\sum_{j=-n+1}^{n+1} (-1)^j\sum_{i=1}^{\lfloor (j+n+1)/2 \rfloor} \b_{n-j+2i},
\end{align*}
and
\begin{align*}
\label{eq:r_M zeroSH odd;p>0}
& 2(-q\chi_+(M) + \sum_{j=-\infty}^q (-1)^j b_j) + 2b_{q+p} - b_q \\
& = 2(-(n+1)\chi_+(M) + \sum_{j=-\infty}^{n+1} (-1)^j b_j) + b_{n+1} \\
& = (n+1)\chi(W) + 2\sum_{j=-n+1}^{n+1} (-1)^j b_j + \sum_{j\in 2\N-1} \b_j \\
& = (n+1)\chi(W) + 2\sum_{j=-n+1}^{n+1} (-1)^j\sum_{i=1}^{\lfloor (j+n+1)/2 \rfloor} \b_{n-j+2i} + \sum_{j\in 2\N - 1} \b_j, \numberthis
\end{align*}
when $p$ is even and $\geq 2$. Note here that we use the fact that $p$ is even to conclude that $2b_{n+1+p} - b_{n+1} = b_{n+1}$ since $\D=2$.

To illustrate the previous computation, consider the case that $W$ is a ball of dimension $2n+2$ so that $M=S^{2n+1}$. Then $\b_0=1$ and $\b_j=0$ for every $j\neq 0$. Hence the first term in \eqref{eq:r_M zeroSH odd;p>0} equals $n+1$, the third term vanishes as well as the second term: $n-j+2i=0$ $\iff$ $j = n+2i$ and $i\geq 1$ $\implies$ $j\geq n+2$, contradicting the fact that $j\leq n+1$. Thus, in this case $r_M=n+1$.

\vskip .3cm
\noindent
{\bf Case 2. $n$ is even}
\vskip .2cm

In this case, we can take $q=n+2$ and, similarly as in the previous case, $r_M$ for any $p$ odd is given by
\begin{align*}
\label{eq:r_M zeroSH even}
& 2(q\chi_+(M) - \sum_{j=-\infty}^q (-1)^j b_j) + 2b_{q+p} + b_q \\
& = 2((n+2)\chi_+(M) - \sum_{j=-\infty}^{n+2} (-1)^j b_j) + 2b_{n+3} + b_{n+2} \\
& = (n+2)\chi(W) - 2\sum_{j=-n+1}^{n+2} (-1)^j b_j + 2b_{n+3} + b_{n+2} \\
& = (n+2)\chi(W) - 2\sum_{j=-n+1}^{n+2} (-1)^j\sum_{i=1}^{\lfloor (j+n+1)/2 \rfloor} \b_{n-j+2i} + 2\sum_{j\in 2\N-1} \b_j + \sum_{j\in 2\N_0} \b_j, \numberthis
\end{align*}
where the first equality holds because $b_{n+3}=b_{n+2+p}$ for any odd $p$ since $\D=2$.

As before, let us illustrate this computation when $W$ is a ball of dimension $2n+2$ so that $M=S^{2n+1}$. Then $\b_0=1$ and $\b_j=0$ for every $j\neq 0$. Hence the first term in \eqref{eq:r_M zeroSH even} equals $n+2$, the third vanishes, the fourth is equal to $1$ and the second equals $-2$: $n-j+2i=0$ $\iff$ $j=n+2i$ and $i\geq 1$ $\iff$ $j\geq n+2$. Since $j\leq n+2$, we have only the term $j=n+2$ counting. Thus, in this case $r_M=n+2+1-2=n+1$.

\section{Final questions}
\label{sec:questions}

To conclude this work, we pose some final questions. These questions concern four complementary aspects of the common underlying picture: the number of simple closed orbits, the underlying contact manifolds admitting Reeb flows with finitely many closed orbits, the equivariant symplectic homology of such manifolds, and the rigidity of contact forms with finitely many closed orbits.

To start with, as mentioned earlier, to the best of our knowledge, all currently known examples of contact manifolds $M$ admitting a contact form with finitely many simple periodic orbits arise as prequantizations of orbifolds. Moreover, in these examples, the number of periodic orbits coincides with $\dim \H_*(M/S^1;\Q)$.

\vskip .3cm
\noindent {\bf Question 1:} Is it true that any closed contact manifold admitting a contact form with finitely many simple periodic orbits is necessarily a prequantization $M$ of an orbifold? Furthermore, in this case, is the minimal number of periodic orbits equal to $\dim \H_*(M/S^1;\Q)$?
\vskip .2cm

We conjecture that both questions have positive answers. The results of the present work provide evidence for the second part. Both questions are known to be true in dimension three under the additional assumption that the first Chern class of the contact structure is torsion \cite{CGHHL1,CGHHL2}. In higher dimensions, the expected lower bound was established for dynamically convex contact forms on $S^{2n+1}$ in \cite{CGG}.

As discussed in the introduction, the work \cite{HM} (cf.\,\cite{McL1}) shows that if the linearized contact homology of $M$ is unbounded, then every contact form on $M$ admits infinitely many closed orbits. As shown in \cite{BO17}, linearized contact homology is isomorphic to positive equivariant symplectic homology whenever the former is well-defined. Therefore, hypothetically, if a contact manifold admits a contact form with finitely many closed orbits, then $\HC_*(M)$ must be bounded. While there are several examples of such manifolds, in all cases known to us $\HC_*(M)$ is periodic in the sense of Definition \ref{def:periodic}.

\vskip .3cm
\noindent {\bf Question 2:} Is it true that every homologically bounded contact manifold has periodic positive equivariant symplectic homology?
\vskip .2cm

We also conjecture that the answer is positive, although we are not aware of any results that directly support this claim.

To conclude this section, we pose a question in the spirit of the contact Hofer–Zehnder conjecture \cite{CGG}. Essentially, this conjecture suggests a dichotomy: either the number of periodic orbits is minimal, in the sense that it is somehow forced by the homology, or there exist infinitely many periodic orbits. As emphasized throughout this paper, all currently known examples of contact forms with finitely many closed orbits are non-degenerate and lacunary. In this work, we proved that, under suitable assumptions, the number of closed orbits of such contact forms is a contact invariant and provides a lower bound for general non-degenerate contact forms $\alpha$ satisfying weak index conditions. Moreover, this bound is attained if and only if $\alpha$ is lacunary. This leads to the following question, closely related to Question 1.

\vskip .3cm
\noindent {\bf Question 3:} Let $M$ be a prequantization of an orbifold. Is it true that every contact form on $M$ has either $\dim \H_*(M/S^1;\Q)$ or infinitely many closed orbits? Furthermore, is it true that any contact form on $M$ with finitely many closed orbits is necessarily non-degenerate and lacunary?
\vskip .2cm

By the results of this work, a positive answer to the second part would imply a positive answer to the first, at least when $\HC_*(M)$ is periodic (a property expected to hold in general, although this remains unknown). Conversely, a positive answer to the first part would imply a positive answer to the second, at least for non-degenerate contact forms satisfying the weak index assumptions of Theorem \ref{thm:main}. Both statements are known to hold in dimension three under the additional assumption that the first Chern class of the contact structure is torsion \cite{CGHHL1,CGHHL2}. In higher dimensions, they were established in \cite{CGG} for non-degenerate dynamically convex contact forms on $S^{2n+1}$ that are invariant under the antipodal map.

\end{document}